\documentclass[final,1p,times,nopreprintline]{elsarticle}

\usepackage{amssymb,amsthm,amsmath, amsfonts}
\usepackage{colortbl}
\usepackage{float}
\usepackage{booktabs}
\usepackage{algorithm}
\usepackage{algpseudocode}
\usepackage{xcolor}
\usepackage{placeins}
\usepackage{enumitem}
\usepackage{hyperref}
\hypersetup{
    colorlinks=true,
    urlcolor=blue,     
    linkcolor=black,
    citecolor=black,
    hypertexnames=false
}

\DeclareMathOperator{\tr}{tr}

\newtheorem{theorem}{Theorem}[section]
\newtheorem{corollary}{Corollary}[section]
\newtheorem{definition}{Definition}[section]
\newtheorem{lemma}{Lemma}[section]
\newtheorem{assumption}{Assumption}
\newtheorem{remark}{Remark}[section]

\journal{Journal of Computational Physics}

\begin{document}
\begin{frontmatter}

\title{A Data-Consistent Approach to Ensemble Filtering} 

\author[label1]{Rylan Spence}
\author[label2]{Troy Butler}
\author[label1]{Clint Dawson}
\affiliation[label1]{organization={Oden Institute for Computational Science and Engineering, University of Texas},
            addressline={201 E 24th St},
            city={Austin},
            postcode={78712},
            state={TX},
            country={U.S.}}

\affiliation[label2]{organization={Department of Mathematical and Statistical Sciences, University of Colorado Denver},
            addressline={4000 1201 Larimer St},
            city={Denver},
            postcode={80202},
            state={CO},
            country={U.S.}}

\begin{abstract}
Ensemble data assimilation is challenging when the ensemble size is substantially smaller than the state and observation dimensions, leaving empirical covariance information rank deficient and sensitive to finite sampling. We introduce the Quantity-of-Interest Principal Component Analysis Ensemble Data-Consistent Filter (QPCA-EnDCF), a deterministic filtering method motivated by data-consistent inversion that regularizes the update in the whitened forecast--observation mismatch space. The method restricts the correction to a rank-$\kappa$ empirical spectral subspace before mapping it to state space through an empirical EnDCF gain. We establish concentration and spectral-projector perturbation results and derive a bias--variance decomposition showing that truncation bias is governed by the discarded component of the population mean residual, while finite-ensemble variability depends on estimation of the retained correction. QPCA-EnDCF avoids the analysis-mean variance contribution from perturbed observations in stochastic EnKF updates, and Lorenz--96 experiments show improved calibration and lower mean-squared error relative to the tested stochastic and deterministic controls, primarily through lower across-realization variance despite increased bias. An idealized inlet shallow-water experiment further shows that QPCA-EnDCF can be implemented in a high-dimensional, severely rank-deficient nonlinear model and remains numerically viable through the prescribed shelf-crossing endpoints.
\end{abstract}


\end{frontmatter}


\section{Introduction}\label{sec:introduction}

\subsection{Small-Ensemble Data Assimilation}\label{subsec:motivation_problem}

Ensemble-based data assimilation provides a computationally tractable approach to uncertainty quantification in nonlinear dynamical systems for which direct propagation of a full $n\times n$ state covariance is infeasible. 
The ensemble Kalman filter (EnKF) and its variants replace the forecast distribution by an empirical measure constructed from $N$ model realizations, allowing the forecast mean and covariance to be approximated \cite{kalman1960,evensen1994,evensen2003,evensen2009}, and enabling their application to problems with state dimensions far exceeding the available ensemble size \cite{houtekamer2016review,carrassi2018geosciences,asch2016data,law2015data,reich2015probabilistic}.

An ensemble of size $N$ produces an empirical covariance of rank at most $N-1$. 
Consequently, when $N\ll n$, the analysis update is determined by a low-dimensional sample representation of the forecast uncertainty, which introduces a fundamental rank deficiency.
This deficiency also exists in observation spaces of dimension $d$ whenever $N\ll d$. Directions outside the empirical ensemble span are represented as having zero sample variance, even though the underlying forecast and observation distributions may remain uncertain in those directions.

This rank deficiency affects both stability and calibration. A calibrated ensemble should represent the magnitude and temporal variation of the analysis error and should be statistically consistent with the verifying state. With a small ensemble, however, sampling variability can distort the empirical covariance within the retained subspace, while unresolved directions cannot be reliably inferred from the available samples. 
These effects become more pronounced in windowed assimilation, where $m$ observations from $L$ times are stacked into a vector of dimension $d=mL$. 
Increasing the window length supplies additional information by enlarging the observation-mismatch space at the cost of increasing the magnitude of the rank deficiency.
The resulting problem is therefore to identify which directions of the observation mismatch are sufficiently supported by the finite ensemble to define a stable correction.

\subsection{Finite-Ensemble Regularization and a DCI Perspective}\label{subsec:related_work}

Several classes of ensemble methods address the effects of finite sampling. 
Stochastic EnKF methods preserve the covariance structure of the Kalman update in expectation by perturbing the observations independently for each ensemble member \cite{burgers1998}. 
Conditional on the forecast ensemble, these perturbations introduce an additional source of sampling variability in the analysis mean that is asymptotically consistent under standard assumptions.
However, for small $N$ the perturbations contribute sampling variability to both the estimated gain and the resulting analysis ensemble.
Deterministic square-root filters avoid observation perturbations by transforming the ensemble anomalies directly. Canonical examples include the EAKF \cite{anderson2001eakf}, ETKF \cite{bishop2001etkf}, EnSRF \cite{whitaker2002ensrf}, and LETKF \cite{hunt2007letkf}, with the common framework synthesized in \cite{tippett2003ensrf}. Because the transform is constructed from the available ensemble, the correction remains confined to the empirical ensemble subspace.

Inflation and localization are commonly introduced to mitigate sampling error.
Inflation compensates for underestimated ensemble spread, while localization suppresses weak or spatially distant sample correlations \cite{hamill2001localization,bocquet2011inflation,houtekamer2016review}. These techniques have been successful in large-scale applications, but their regularization is imposed through the state covariance, the ensemble transform, or a prescribed spatial structure. They do not explicitly rank the directions of the whitened observation mismatch according to their empirical support.

More general covariance regularization methods, including spectral and structured covariance approximations, provide additional mechanisms for stabilizing empirical covariance estimates. Spectral diagonal ensemble Kalman filters restrict empirical covariance to a diagonal form in a prescribed spectral basis \cite{kasanicky2015spectral}, whereas process-convolution models estimate structured, nonstationary covariance surrogates from the ensemble \cite{tsyrulnikov2024regularization}. Their effect nevertheless depends on the space in which regularization is applied. In the setting considered here, the quantity that determines disagreement between the forecast and the data is the observation mismatch rather than the state vector itself. Regularizing this mismatch space therefore offers a distinct alternative: the update may retain directions supported by the ensemble while excluding directions for which the empirical mismatch covariance is unresolved or unstable.

This perspective is related to data-consistent inversion (DCI) and the maximal updated density (MUD) framework, in which information supplied by quantities of interest is used to update a prior distribution subject to consistency between observed and predicted distributions. DCI constructs updated distributions whose pushforwards reproduce observed distributions \cite{butler2012part1,butler2012part2,butler2014sip,butler2018consistent}, while the MUD formulation provides a corresponding point-estimation principle \cite{pilosov2023mud}. In this setting, the update acts through observable directions while retaining prior structure along directions not informed by the chosen quantity of interest. Sequential and variational extensions demonstrate how this principle can be used in time-dependent settings \cite{delcastillo2025smud,spence2026variational}. However, the finite-ensemble filtering problem considered here additionally requires a practical mechanism for selecting stable mismatch directions and mapping the resulting correction back to the state. Existing stochastic, square-root, and covariance-regularized ensemble methods do not directly provide this combination.

\subsection{Main Contributions}\label{subsec:numerical_overview}

This work introduces the Quantity-of-Interest Principal Component Analysis Ensemble Data-Consistent Filter (QPCA-EnDCF) as a deterministic ensemble update that regularizes the inverse problem in the whitened observation-mismatch space. 
The method constructs forecast residuals over the assimilation window, whitens them with respect to the observation-error covariance, and retains a rank-$\kappa$ principal subspace of their empirical covariance. 
The projected mismatch is then mapped back to the state through an empirical EnDCF gain. 
Thus, QPCA-EnDCF regularizes the correction in the space in which agreement with the observations is measured, rather than through localization or transformation imposed directly on the state or ensemble space.

The accompanying analysis separates the effects of spectral truncation and finite sampling. 
Concentration and spectral-projector perturbation results control estimation of the whitened mismatch covariance and its retained subspace, including an $\mathcal{O}(1/N)$ covariance-driven projector contribution in the bias bound. 
Combined with an exact mean-squared-error (MSE) decomposition, these results show that the leading truncation contribution is governed by the component of the population mean residual discarded by the retained subspace, while finite-ensemble variability depends on estimation of the retained correction through the ensemble size, truncation rank, and spectral gap. This identifies conditions under which truncation can reduce finite-ensemble variability at the cost of controlled bias.

The numerical experiments examine these conclusions in two complementary settings. 
Repeated Lorenz--96 experiments compare QPCA-EnDCF with sequential and windowed stochastic and deterministic ensemble filters. 
In the strongly undersampled regime, QPCA-EnDCF improves calibration and reduces MSE, exchanging increased bias for an approximately $82\%$--$84\%$ reduction in across-realization variance relative to the four tested baselines.
Ensemble-size and truncation-rank studies further support the interpretation of rank as a regularization parameter.
A separate high-dimensional shallow-water experiment on an idealized inlet tests computational feasibility under severe ensemble rank deficiency. QPCA-EnDCF remains numerically viable through the prescribed shelf-crossing endpoints when neither the state nor the stacked observation space can be adequately spanned by the ensemble, although the experiment does not establish physically accurate storm-surge prediction or long-horizon nonlinear stability.

\subsection{Outline}\label{subsec:paper_organization}

The remainder of the paper is organized as follows. Section~\ref{sec:problem} establishes the ensemble data assimilation framework. Section~\ref{sec:algorithms_uq} presents the ensemble filtering methods and develops QPCA-EnDCF. Section~\ref{sec:motivation} presents the Lorenz--96 small-ensemble experiments, and Section~\ref{sec:bias_variance_results} gives the empirical bias--variance analysis across ensemble sizes and truncation ranks. Section~\ref{sec:idealized_inlet_filter} presents the high-dimensional shallow-water experiment, and Section~\ref{sec:conclusion} concludes the paper. The appendices provide notation, complete filter specifications, concentration and spectral-projector stability results, and bias and variance bounds for the analysis mean.


\section{Ensemble Data Assimilation Framework}\label{sec:problem}
To define the filtering problems considered in this work, we first introduce the state, observation, and ensemble notation used throughout the paper. We begin with the dynamical evolution and the indexing of assimilation windows, then define the observation model and its windowed form. We next introduce the finite-ensemble representation, the empirical covariances, and the EnDCF gain. This construction leads to the rank limitation that motivates the algorithms in Section~\ref{sec:algorithms_uq}.

\subsection{Dynamical Model and Assimilation Windows}
\label{subsec:state_windows}

Let $\mathbf x_k\in\mathbb R^n$ denote the state at observation time $t_k=kT_{\mathrm{obs}}$. The discrete-time forecast model is written as $\mathbf x_{k+1}=\mathcal M(\mathbf x_k)$, where $\mathcal M$ advances the state over one observation interval. A dagger identifies the corresponding true state, as in $\mathbf x_k^\dagger$. We use $\mathbf x_{t_k}$ when emphasizing continuous time and $\mathbf x_k$ when the time dependence is implicit.

For four-dimensional methods, we group $L$ consecutive observation times into windows indexed by $w\in\{1,\ldots,W\}$. 
The total number of observation times is $K:=WL$. 
Window $w$ contains global indices $k=(w-1)L+1,\ldots, wL$, and we use the local notation
\begin{equation}
\mathbf{x}_{w,\ell}:=\mathbf{x}_{(w-1)L+\ell},
\qquad
\mathbf{X}_{w,\ell}:=\mathbf{X}_{(w-1)L+\ell},
\qquad \ell=1,\ldots,L,
\label{eq:window_local_state_notation}
\end{equation}
so that $\ell=0$ and $L$ denote the window-initial and endpoint states, respectively. 


\subsection{Sequential and Stacked Observations}
\label{subsec:obs_problem}

At time $t_k$, the observation $\mathbf z_k\in\mathbb R^m$ is modeled by
\begin{equation}
\mathbf{z}_k
=
\mathbf{H}\mathbf{x}_{k}^\dagger+\boldsymbol{\eta}_k,
\qquad
\boldsymbol{\eta}_k\sim\mathcal{N}(\mathbf{0},\mathbf{R}),
\label{eq:obs_model}
\end{equation}
where $\mathbf H\in\mathbb R^{m\times n}$ is the observation operator and $\boldsymbol\eta_k$ is the observation error. The variables $\{\boldsymbol\eta_k\}$ are independent across time and independent of the state, and they represent measurement error and unresolved representativeness effects.
Under~\eqref{eq:obs_model}, the likelihood depends only on $\mathbf H\mathbf x$ and therefore supplies no direct information along $\ker(\mathbf H)$. Nevertheless, the dynamics can transmit observational information to unobserved components, and an assimilation method can update those components through forecast cross-covariances. The quality of these indirect updates depends on how reliably the finite ensemble estimates the relevant covariance relationships.

Windowed methods condition jointly on observations from multiple consecutive times. For $\ell=1,\ldots,L$, let $\mathbf z_{w,\ell}:=\mathbf z_{(w-1)L+\ell}$. Then, $d:=mL$ and the stacked observation vector and covariance are given by
\begin{equation}
\mathbf{z}_w
=
\begin{bmatrix}
\mathbf{z}_{(w-1)L+1}^\top & 
\cdots &
\mathbf{z}_{wL}^\top
\end{bmatrix}^\top
\in\mathbb{R}^{d}, \qquad \text{ and } \qquad \mathbf{R}_L = \mathbf{I}_L \otimes \mathbf{R}
\in\mathbb{R}^{d\times d},
\label{eq:stacked_obs}
\end{equation}
where the covariance is block diagonal since the observation errors are temporally independent.

For the theoretical development, Definition~\ref{def:data_space} introduces a tangent-linear stacked observation operator $\mathbf H_L$ that maps perturbations of a window-initial state to perturbations of the stacked observations. 
We denote the resulting predicted-observation ensemble by $\mathbf Z_w\in\mathbb R^{d\times N}$. Section~\ref{sec:algorithms_uq} constructs whitened mismatches from $\mathbf Z_w-\mathbf z_w\mathbf1^\top$. Throughout the paper, a residual is forecast minus observation, while an innovation is observation minus forecast.


\subsection{Ensemble Representation and Empirical Gain}
\label{subsec:ensemble_uq}

Given $N$ states $\mathbf x^1,\ldots,\mathbf x^N\in\mathbb R^n$, let
\begin{equation}
\mathbf{X} = [\mathbf{x}^{1},\ldots,\mathbf{x}^{N}]\in\mathbb{R}^{n\times N}, \qquad \text{ and } \qquad
\bar{\mathbf{x}}
=
\frac{1}{N}\mathbf{X}\mathbf{1}
=
\frac{1}{N}\sum_{j=1}^N \mathbf{x}^{j} \in\mathbb{R}^n,
\label{eq:ensemble_mean}
\end{equation}
where $\mathbf1\in\mathbb R^N$ denotes the vector of ones. 
To describe fluctuations about the mean, define the centering matrix, state anomaly matrix, and associated empirical covariance, respectively, as
\begin{equation}
\mathbf{J}_N:=\mathbf{I}_N-\frac{1}{N}\mathbf{1}\mathbf{1}^\top, \qquad 
\mathbf{A}_{x}
:=
\mathbf{X}\mathbf{J}_N
=
[\mathbf{x}^{1}-\bar{\mathbf{x}},\ldots,\mathbf{x}^{N}-\bar{\mathbf{x}}], \qquad \text{ and } \qquad
\widehat{\mathbf{P}}
=
\frac{1}{N-1}\mathbf{A}_{x}\mathbf{A}_{x}^\top.
\label{eq:sample_covariance}
\end{equation}

At time $t_k$, an unstarred ensemble represents a forecast quantity, while a superscript $\ast$ denotes the corresponding conditioned analysis quantity.
For example, we write $\mathbf X_k$ and $\mathbf X_k^\ast$ for the forecast and analysis ensembles. 

For any ensemble $\mathbf X=[\mathbf x^1,\ldots,\mathbf x^N]$, columnwise model propagation is denoted by $\mathcal M_N(\mathbf X)$.
The window-initial ensemble is defined by cycling. Specifically, $\mathbf X_{1,0}:=\mathbf X_0$ and $\mathbf X_{w,0}:=\mathbf X_{w-1}^\ast$ for $w\ge2$, with $j$th column $\mathbf x_{w,0}^j$. We write $\mathbf X_{w,0:L}:=(\mathbf X_{w,0},\ldots,\mathbf X_{w,L})$ for the forecast trajectory generated from this initial ensemble. The matrix $\mathbf X_{w,L}$ is the endpoint forecast, whereas $\mathbf X_w^\ast$ is the endpoint analysis obtained after assimilating $\mathbf z_w$.


Ensemble filters form analysis increments from covariances between state and predicted-observation anomalies. For a linear stacked observation operator $\mathbf H_L$, the generic predicted-observation ensemble, its mean, and its anomaly matrix are
\begin{equation}
\mathbf{Z}:=\mathbf H_L\mathbf{X}\in\mathbb{R}^{d\times N},
\qquad
\bar{\mathbf{z}}:=\frac{1}{N}\mathbf{Z}\mathbf{1}\in\mathbb{R}^{d},
\qquad
\mathbf{A}_{z}:=\mathbf{Z}\mathbf{J}_N\in\mathbb{R}^{d\times N}.
\end{equation}
The analogous windowed quantities are constructed from the endpoint state ensemble and the stacked predicted observations. The following definition makes the empirical EnDCF gain precise.

\begin{definition}[EnDCF ensemble covariances and empirical gain]
\label{def:endcf_gain_ensemble}
Fix a window index $w\in\{1,\ldots,W\}$,  
the anomaly matrices are denoted as
\begin{equation}
\mathbf{A}_{x,w}
:=
\mathbf{X}_{w,L}\mathbf{J}_N
\in\mathbb{R}^{n\times N},
\qquad \text{ and } \qquad
\mathbf{A}_{z,w}
:=
\mathbf{Z}_w\mathbf{J}_N
\in\mathbb{R}^{d\times N}, 
\end{equation}
and the empirical cross-covariance and observation-space covariance are defined by
\begin{equation}
\mathbf{C}_{xz,w}
:=
\frac{1}{N-1}\mathbf{A}_{x,w}\mathbf{A}_{z,w}^\top
\in\mathbb{R}^{n\times d},
\qquad \text{ and } \qquad
\mathbf{C}_{zz,w}
:=
\frac{1}{N-1}\mathbf{A}_{z,w}\mathbf{A}_{z,w}^\top
\in\mathbb{R}^{d\times d}.
\end{equation}
The {\em EnDCF gain} used in implementation is then the empirical gain given by
\begin{equation}
\mathbf K_{w,\mathrm{DC}}
:=
\mathbf{C}_{xz,w}\,\mathbf{C}_{zz,w}^{\dagger}
\in\mathbb{R}^{n\times d},
\label{eq:endcf_gain_pinv}
\end{equation}
where $(\cdot)^{\dagger}$ denotes the Moore--Penrose pseudoinverse.
\end{definition}

\par\smallskip\noindent\refstepcounter{remark}\textbf{Remark~\theremark}
\label{rem:timing_convention}
Stacked predicted observations, $\mathbf{Z}_w$, are generated from trajectories initialized at $\mathbf{X}_{w,0}$, whereas $\mathbf{K}_{w, \rm DC}$ maps the resulting observation-space correction to the forecast endpoint $\mathbf{X}_{w,L}$. 
Thus, $\mathbf{X}_w^\ast$ denotes the conditioned endpoint ensemble, not a conditioned trajectory.

\subsection{Rank Deficiency in State and Observation Spaces}
\label{subsec:finite_ensemble_rank}

An ensemble covariance has rank at most $N-1$, which can contribute to variance collapse and filter divergence when left unmitigated. 
In particular,
\begin{equation}
\operatorname{rank}(\mathbf C_{zz,w})\leq \min(d,N-1),
\end{equation}
so $\mathbf C_{zz,w}$ is singular whenever $N-1<d$. In this regime, the pseudoinverse in~\eqref{eq:endcf_gain_pinv} is algebraically well defined, whereas $\mathbf C_{zz,w}^{-1}$ does not exist.

Pseudoinversion resolves the algebraic singularity but does not ensure that the sampled observation-space directions are statistically reliable. When $N$ is small, even the nonzero singular directions can be poorly estimated, and stochastic observation perturbations can add variability to the correction. QPCA-EnDCF therefore imposes an additional rank-$\kappa$ restriction on the whitened mismatch, where $\kappa$ denotes the spectral truncation rank. More generally, $\mathbf C_{zz,w}^{\dagger}$ may be replaced by a stabilized inverse, such as the Tikhonov form $(\mathbf C_{zz,w}+\varepsilon\mathbf I)^{-1}$. The subsequent bounds require the implemented gain to satisfy the operator-moment condition stated in Section~\ref{subsec:regularity_assumptions}.



The truncation rank $\kappa$ controls the dimension of the retained observation-mismatch correction. 
The rank and noise considerations motivate the analysis mechanisms introduced next in Section~\ref{sec:algorithms_uq} to distinguish stochastic corrections based on perturbed observations, deterministic transforms in the finite ensemble space, and an explicitly rank-restricted correction in the whitened observation-mismatch space.


\section{Ensemble Filtering and Observation-Mismatch Regularization}\label{sec:algorithms_uq}

This section develops QPCA-EnDCF and places it in the context of the ensemble filtering methods used as comparisons in the numerical experiments. 
We first present the QPCA-EnDCF construction, which regularizes the forecast-observation mismatch in a whitened observation space before mapping the retained correction to the state through the empirical EnDCF gain. 
The construction follows the DCI and MUD perspectives introduced in Section~\ref{sec:introduction}, in which updates act through quantities of interest (QoI) while retaining prior structure along directions not informed by those quantities. 
Here, the data-defined QoI is the whitened forecast--observation residual, and its leading empirical principal directions define the subspace in which the correction is constructed. The resulting update therefore regularizes the observation mismatch directly rather than imposing a quadratic penalty or other regularization on the full state increment.

We then interpret the construction in terms of population residual statistics and spectral truncation. Finally, we summarize the stochastic and deterministic ensemble-transform filters used as comparison methods and identify the finite-ensemble effects that motivate the numerical studies.

\subsection{QPCA-EnDCF: Spectral Regularization of the Observation Mismatch}
\label{subsec:qpca_method}


The QPCA-EnDCF analysis has three stages: whitening the residuals, identifying an empirical spectral subspace, and constructing the corresponding rank-restricted correction.

\paragraph{Stage 1: Whitening the observation mismatch}
First, normalize the stacked forecast-observation residuals by the observation uncertainty. For
$\mathbf R_L=\sigma_{\mathrm{obs}}^2\mathbf I_d$, define
\begin{equation}
\mathbf E_w
=
\sigma_{\mathrm{obs}}^{-1}
\bigl(
\mathbf Z_w-\mathbf z_w\mathbf1^\top
\bigr)
=
\mathbf R_L^{-1/2}
\bigl(
\mathbf Z_w-\mathbf z_w\mathbf1^\top
\bigr),
\label{eq:whitening_problem}
\end{equation}
where $\mathbf Z_w\in\mathbb R^{d\times N}$ contains the stacked predicted observations. 
The factor $\mathbf R_L^{-1/2}$ expresses each mismatch relative to observation uncertainty, so that the observation errors have identity covariance in the whitened coordinates. The scalar form applies to the isotropic covariance used in the experiments.
Subsection~\ref{app:qpca_endcf} describes the factors used for correlated errors.

\paragraph{Stage 2: Identifying the retained residual subspace}
Let
$\mathbf E_{c,w}:=\mathbf E_w\mathbf J_N$
denote the centered whitened residual matrix, let
$\bar{\mathbf e}_w:=N^{-1}\mathbf E_w\mathbf1$,
and set
$r_E:=\min(d,N-1)$.
The empirical covariance and its spectral decomposition are
\begin{equation}
\widehat{\boldsymbol{\Sigma}}_{E,w}
=
\frac{1}{N-1}
\mathbf E_{c,w}\mathbf E_{c,w}^{\top}
=
\sum_{i=1}^{r_E}
\hat{\lambda}_{i,w}
\hat{\mathbf v}_{i,w}\hat{\mathbf v}_{i,w}^{\top}
=
\widehat{\mathbf V}_{r_E,w}
\widehat{\boldsymbol{\Lambda}}_{r_E,w}
\widehat{\mathbf V}_{r_E,w}^{\top},
\label{eq:spectral_decomp_problem}
\end{equation}
where
$\mathbf E_w=[\mathbf e_w^1,\ldots,\mathbf e_w^N]\in\mathbb R^{d\times N}$.
The eigenvalue $\hat{\lambda}_{i,w}$ gives the empirical residual variance along
$\hat{\mathbf v}_{i,w}$. For
$1\le\kappa\le\min(d,N-1)$,
let
\[
\widehat{\mathbf V}_{\kappa,w}
:=
[\hat{\mathbf v}_{1,w},\ldots,\hat{\mathbf v}_{\kappa,w}]
\]
denote the leading empirical eigenvectors. The retained QPCA coordinates are then given by
\begin{equation}
\mathbf Q_{\mathrm{PCA},w}
:=
\widehat{\mathbf V}_{\kappa,w}^{\top}\mathbf E_w
\in\mathbb R^{\kappa\times N}.
\label{eq:qpca_correction}
\end{equation}
The basis is estimated from the centered residuals $\mathbf E_{c,w}$, whereas the coordinates are computed from the uncentered residuals $\mathbf E_w$. Thus, the retained correction includes the sample mean mismatch and fluctuations about that mean to the extent that they lie in the retained subspace. A component of the mean mismatch orthogonal to the retained subspace is not corrected and can contribute to truncation bias.

\paragraph{Stage 3: Constructing the rank-restricted correction}
Mapping the retained coordinates back to the whitened observation space gives
\begin{equation}
\boldsymbol{\Delta}_{\mathrm{white},w}
:=
-\widehat{\mathbf V}_{\kappa,w}\mathbf Q_{\mathrm{PCA},w}
=
-\widehat{\mathbf V}_{\kappa,w}
\widehat{\mathbf V}_{\kappa,w}^{\top}
\mathbf E_w
\in\mathbb R^{d\times N}.
\label{eq:qpca_delta_white}
\end{equation}
The negative sign drives the retained forecast--observation residual toward zero. Unwhitening gives the stacked observation-space correction
\begin{equation}
\boldsymbol{\Delta}_{\mathrm{obs},w}
:=
\mathbf R_L^{1/2}
\boldsymbol{\Delta}_{\mathrm{white},w}
=
-\mathbf R_L^{1/2}
\widehat{\mathbf V}_{\kappa,w}
\widehat{\mathbf V}_{\kappa,w}^{\top}
\mathbf E_w.
\label{eq:qpca_delta_obs}
\end{equation}
The empirical EnDCF gain from Definition~\ref{def:endcf_gain_ensemble} then maps this correction to the endpoint state:
\begin{equation}
\mathbf X_w^\ast
=
\mathbf X_{w,L}
+
\mathbf K_{w,\mathrm{DC}}
\boldsymbol{\Delta}_{\mathrm{obs},w},
\qquad
\mathbf K_{w,\mathrm{DC}}
=
\mathbf C_{xz,w}\mathbf C_{zz,w}^{\dagger}.
\label{eq:qpca_analysis_update}
\end{equation}
Thus, QPCA-EnDCF first restricts the correction in the whitened observation-mismatch space and only then maps the retained correction to the state. Complete implementation details are given in Algorithm~\ref{alg:qpca_endcf}.

\subsection{Population Interpretation and Spectral Truncation}
\label{subsec:qpca_population}

To distinguish spectral truncation from finite-sample estimation, we introduce the corresponding population quantities for a fixed assimilation window. Throughout this subsection, the realized observation vector $\mathbf z_w$ is held fixed. Population quantities are defined with respect to the forecast distribution, while sample quantities are computed from the finite ensemble.

\begin{definition}[Whitened residuals and associated covariances]
\label{def:whitened_residual_covariances}
Let $\mathbf R_L\in\mathbb R^{d\times d}$ be symmetric positive definite, and let $\mathbf e_w^j\in\mathbb R^d$ denote the $j$th column of the whitened residual matrix $\mathbf E_w$ in \eqref{eq:whitening_problem}. Because the theoretical analysis fixes one window, we suppress the window subscript on residual quantities, covariances, spectral projectors, and the EnDCF gain; for example, we write $\mathbf e^j$, $\bar{\mathbf e}$, $\widehat{\boldsymbol\Sigma}_E$, and $\mathbf K_{\mathrm{DC}}$ for their window-indexed counterparts.

Conditional on the realized observations $\mathbf z_w$, the population residual mean and covariance are
\begin{align*}
\boldsymbol\mu_E
&:=
\mathbb E[\mathbf e^j\mid\mathbf z_w]
=
\mathbf R_L^{-1/2}
\bigl(
\mathbf H_L\boldsymbol\mu_{w,0}
-
\mathbf z_w
\bigr),
\\
\boldsymbol\Sigma_E
&:=
\operatorname{Cov}(\mathbf e^j\mid\mathbf z_w)
=
\mathbf R_L^{-1/2}
\mathbf H_L\mathbf P(\mathbf H_L)^\top
\mathbf R_L^{-1/2}.
\end{align*}
The expectation is taken over the forecast distribution $P_{\mathcal X}$, with the observations held fixed. The sample mean $\bar{\mathbf e}$ and covariance $\widehat{\boldsymbol\Sigma}_E$ are defined from $\{\mathbf e^j\}_{j=1}^N$ as in \eqref{eq:spectral_decomp_problem}.
\end{definition}

In the theoretical development, expectations are therefore taken over ensemble sampling and any algorithmic randomness with $\mathbf z_w$ held fixed. We suppress the explicit conditioning ``$\mid\mathbf z_w$'' in $\mathbb E[\cdot]$ and $\operatorname{Cov}(\cdot)$ except where needed for clarity.

The population and empirical residual covariances determine the ideal and finite-ensemble spectral subspaces, respectively.

\begin{definition}[Population and sample spectral decompositions]
\label{def:spectral_decompositions}
Let $d:=mL$. Write the population spectral decomposition as
\begin{equation}
\boldsymbol\Sigma_E
=
\sum_{i=1}^{d}
\lambda_i\mathbf v_i\mathbf v_i^\top,
\qquad
\lambda_1\ge\cdots\ge\lambda_d\ge0.
\end{equation}
Let $\{(\hat\lambda_i,\hat{\mathbf v}_i)\}_{i=1}^{d}$ denote the eigenpairs of the sample covariance $\widehat{\boldsymbol\Sigma}_E$, ordered so that
$\hat\lambda_1\ge\cdots\ge\hat\lambda_d\ge0$, with the convention that
$\hat\lambda_i=0$ for $i>\min(d,N-1)$. For
$1\le\kappa\le\min(d,N-1)$, define the corresponding rank-$\kappa$ covariance truncations by
\begin{equation}
\boldsymbol\Sigma_{E,\kappa}
:=
\sum_{i=1}^{\kappa}
\lambda_i\mathbf v_i\mathbf v_i^\top,
\qquad
\widehat{\boldsymbol\Sigma}_{E,\kappa}
:=
\sum_{i=1}^{\kappa}
\hat\lambda_i\hat{\mathbf v}_i\hat{\mathbf v}_i^\top.
\end{equation}
\end{definition}

For the QPCA correction, the relevant objects are the invariant subspaces associated with these truncations.

\begin{definition}[Spectral projectors and QPCA truncation]
\label{def:projectors_truncation}
For a symmetric matrix $\mathbf C\in\mathbb R^{d\times d}$ with orthonormal eigenvectors
$\{\mathbf v_i(\mathbf C)\}_{i=1}^{d}$ ordered by nonincreasing eigenvalue, define the rank-$\kappa$ orthogonal projector
\begin{equation}
\boldsymbol\Pi_\kappa(\mathbf C)
:=
\sum_{i=1}^{\kappa}
\mathbf v_i(\mathbf C)\mathbf v_i(\mathbf C)^\top.
\end{equation}
We write
$
\boldsymbol\Pi_\kappa
:=
\boldsymbol\Pi_\kappa(\boldsymbol\Sigma_E)$
and
$\widehat{\boldsymbol\Pi}_\kappa
:=
\boldsymbol\Pi_\kappa(\widehat{\boldsymbol\Sigma}_E)
$
for the population and empirical projectors, respectively. In QPCA-EnDCF,
$\widehat{\boldsymbol\Pi}_\kappa$ is the finite-ensemble projector used to restrict the observation-mismatch correction.
\end{definition}

A positive population cutoff gap, 
$\lambda_\kappa-\lambda_{\kappa+1}
\ge
\delta_\kappa>0$, 
makes the retained population subspace unique and provides the spectral separation needed to control
$\widehat{\boldsymbol\Pi}_\kappa$
through perturbations of
$\widehat{\boldsymbol\Sigma}_E$.
The concentration and spectral-projector results used for this control are developed in the appendices and lead to the bias--variance bounds in Theorem~\ref{thm:bias_variance}.

The distinction between covariance truncation and mean-residual truncation is central to the analysis. In particular,
$
\|(\mathbf I-\boldsymbol\Pi_\kappa)\boldsymbol\mu_E\|^2
$
measures the component of the population mean whitened residual discarded by the rank-$\kappa$ correction. This quantity governs the leading truncation contribution to the bias. In general, it cannot be bounded solely by the covariance eigenvalue tail
$\sum_{i>\kappa}\lambda_i$
without an additional condition relating $\boldsymbol\mu_E$ to $\boldsymbol\Sigma_E$. Thus, retaining the leading covariance directions does not by itself guarantee that the population mean mismatch is fully represented. The same distinction applies to the approximate maps $Q_s$ used in the later bias analysis.

This also clarifies the role of $\kappa$ as a regularization parameter. Increasing $\kappa$ retains more directions of the observation mismatch and can reduce the discarded-mean contribution, but it also requires estimating a larger empirical subspace from the same finite ensemble. Conversely, smaller $\kappa$ imposes stronger truncation and can introduce additional bias while limiting the correction to fewer empirically estimated directions. The theory developed in the appendices quantifies these two effects through the population truncation and sample-projector terms.

Consequently, QPCA-EnDCF does not enlarge the ensemble span or populate unresolved directions. Rather, it selects a lower-dimensional correction within the empirically resolved whitened observation-mismatch space and leaves the orthogonal residual component unchanged. The resulting analysis is conditionally deterministic, with the retained rank explicitly controlling the strength of the spectral regularization.

\subsection{Comparison Filters: Alternative Analysis Mechanisms}
\label{subsec:comparison_filters}

To place the QPCA-EnDCF construction in context, we compare it with two established classes of ensemble filters: stochastic gain-based updates and deterministic ensemble-transform updates. These methods provide complementary controls because they differ from QPCA-EnDCF in both the source of finite-ensemble variability and the space in which the analysis correction is constructed. Sequential and four-dimensional variants further distinguish the analysis mechanism from the update cadence. Complete cycling and implementation specifications for all comparison filters are provided in~\ref{app:algorithms}.

\subsubsection{Stochastic Gain Updates}
\label{subsec:seq_enkf_method}

The sequential stochastic ensemble Kalman filter (Seq-EnKF) \citep{burgers1998,houtekamer1998data} applies an empirical Kalman gain to independently perturbed observations. Let
$\boldsymbol{\mathcal E}_k:=[\boldsymbol\epsilon_k^1,\ldots,\boldsymbol\epsilon_k^N]$,
where
$\boldsymbol\epsilon_k^j\sim\mathcal N(\mathbf0,\mathbf R)$
independently. At observation time $t_k$, the analysis is
\begin{equation}
\mathbf X_k^\ast
=
\mathbf X_k
+
\mathbf K_k
\bigl(
\mathbf z_k\mathbf1^\top
+
\boldsymbol{\mathcal E}_k
-
\mathbf Z_k
\bigr),
\qquad
\mathbf Z_k:=\mathbf H\mathbf X_k,
\label{eq:enkf_update}
\end{equation}
with empirical Kalman gain
\begin{equation}
\mathbf K_k
=
\widehat{\mathbf P}_k\mathbf H^\top
\left(
\mathbf H\widehat{\mathbf P}_k\mathbf H^\top+\mathbf R
\right)^{-1}.
\label{eq:kalman_gain}
\end{equation}
The additive observation covariance $\mathbf R$ regularizes the innovation covariance, while the memberwise perturbations reproduce the observation-error contribution to the Kalman covariance in expectation. These perturbations also introduce an additional finite-$N$ contribution to the analysis-mean variability. Conditional on the gain, this contribution has expected squared norm
$\mathcal O(\|\mathbf K_k\mathbf R^{1/2}\|_F^2/N)$.
Unlike QPCA-EnDCF, the stochastic update does not explicitly select a lower-dimensional set of forecast--observation mismatch directions.

The four-dimensional stochastic ensemble Kalman filter (4D-EnKF) \citep{hunt2004four,evensen2009} applies the same gain-based mechanism to the stacked observations over a window of length $L$. For
$\boldsymbol{\mathcal E}_w:=[\boldsymbol\epsilon_w^1,\ldots,\boldsymbol\epsilon_w^N]$
with independent columns
$\boldsymbol\epsilon_w^j\sim\mathcal N(\mathbf0,\mathbf R_L)$,
the endpoint analysis is
\begin{equation}
\mathbf X_w^\ast
=
\mathbf X_{w,L}
+
\mathbf K_w
\bigl(
\mathbf z_w\mathbf1^\top
+
\boldsymbol{\mathcal E}_w
-
\mathbf Z_w
\bigr).
\label{eq:4d_perturb}
\end{equation}
The gain is formed from empirical cross-covariances between the endpoint state and the stacked predicted observations. Thus, observations from all $L$ times can constrain the endpoint through the model dynamics. The method updates only $\mathbf X_{w,L}$ and is therefore an endpoint filter rather than a smoother.

The sequential and four-dimensional formulations use the same stochastic analysis mechanism but differ in update cadence and in the dimension in which observation perturbations are sampled. Seq-EnKF samples perturbations in $\mathbb R^m$ at each observation time, whereas 4D-EnKF samples them jointly in $\mathbb R^d$, where $d=mL$. 

This difference gives a concrete finite-sample consequence. If
$\boldsymbol\epsilon_w^j\overset{\mathrm{i.i.d.}}{\sim}
\mathcal N(\mathbf0,\mathbf R_L)$
with
$\mathbf R_L=\sigma_{\mathrm{obs}}^2\mathbf I_d$,
and $\widehat{\mathbf R}_L$ denotes their unbiased sample covariance, then
\begin{equation}
\left(
\mathbb E
\bigl[
\|\widehat{\mathbf R}_L-\mathbf R_L\|_F^2
\bigr]
\right)^{1/2}
=
\sigma_{\mathrm{obs}}^2
\sqrt{\frac{d(d+1)}{N-1}}
=
\mathcal O\!\left(
\sigma_{\mathrm{obs}}^2\frac{d}{\sqrt N}
\right).
\label{eq:perturbation_covariance_fluctuation}
\end{equation}
Because $d=mL$, this perturbation-covariance fluctuation grows linearly with $L$ for fixed $m$, $N$, and $\sigma_{\mathrm{obs}}$, whereas the corresponding sequential calculation uses $d=m$. This result isolates one source of stochastic finite-ensemble variability; it does not determine the total filter error because temporal stacking also supplies additional observational information.

\label{subsec:4d_enkf_method}

\subsubsection{Deterministic Ensemble-Transform Updates}
\label{subsec:letkf}
\label{subsubsec:etkf}

Ensemble transform Kalman filters replace the memberwise observation perturbations by deterministic square-root updates constructed in the $N$-dimensional ensemble space. For the sequential global ETKF, define the innovation
$\mathbf d_k:=\mathbf z_k-\bar{\mathbf z}_k$.
The ensemble-space analysis covariance, mean weight, and symmetric positive-semidefinite transform are
\begin{equation}
\widetilde{\mathbf P}_k
=
\left[
(N-1)\mathbf I_N
+
\mathbf A_{z,k}^{\top}\mathbf R^{-1}\mathbf A_{z,k}
\right]^{-1},
\qquad
\bar{\mathbf w}_k
=
\widetilde{\mathbf P}_k
\mathbf A_{z,k}^{\top}\mathbf R^{-1}\mathbf d_k,
\qquad
\widetilde{\mathbf W}_k
=
\left[
(N-1)\widetilde{\mathbf P}_k
\right]^{1/2}.
\label{eq:seq_etkf_weights_main}
\end{equation}
The corresponding Seq-ETKF analysis is
\begin{equation}
\mathbf X_k^\ast
=
\bar{\mathbf x}_k\mathbf1^\top
+
\mathbf A_{x,k}
\left(
\bar{\mathbf w}_k\mathbf1^\top
+
\widetilde{\mathbf W}_k
\right).
\label{eq:seq_etkf_analysis_main}
\end{equation}
The 4D-ETKF uses the same ensemble-space construction with the stacked quantities
$\mathbf A_{z,w}$, $\mathbf R_L$, and
$\mathbf d_w:=\mathbf z_w-\bar{\mathbf z}_w$,
and applies the resulting transform to the endpoint ensemble $\mathbf X_{w,L}$. We refer to these unlocalized formulations as Seq-ETKF and 4D-ETKF. 

The local ensemble transform Kalman filter (LETKF) applies the same ensemble-space construction separately to local state regions using tapered observation precision. 
For a region $q$, localization weights $\boldsymbol\rho_q$ modify the observation precision, for example through
$\operatorname{diag}(\boldsymbol\rho_q)\mathbf R^{-1}$
in the sequential case and its stacked counterpart in the four-dimensional case. Uniform weights and a single full-state region recover the global ETKF formulation. We denote the corresponding sequential and windowed localized filters by Seq-LETKF and 4D-LETKF, respectively.

Both ETKF and LETKF eliminate the memberwise observation perturbations used by the stochastic EnKF while retaining a correction constructed through the finite ensemble space. Localization changes the observations and precision entering each local transform, but it does not explicitly rank the whitened mismatch directions spectrally. The Seq-ETKF, 4D-ETKF, Seq-LETKF, and 4D-LETKF implementations used in the experiments are specified completely in~\ref{app:algorithms}.
Taken together, the comparison filters distinguish two alternatives to the QPCA-EnDCF construction. Stochastic EnKF methods use perturbed innovations in observation space, whereas ETKF/LETKF methods eliminate this perturbation randomness but construct deterministic corrections through the finite ensemble space. QPCA-EnDCF is also conditionally deterministic, but differs by imposing an explicit rank-$\kappa$ restriction in the whitened observation-mismatch space. These distinctions motivate the finite-ensemble comparisons considered in the numerical experiments.

\label{subsubsec:letkf_localized}
\label{subsubsec:seq_letkf}
\label{subsubsec:4d_letkf}



\section{Lorenz--96 Small-Ensemble Experiments}\label{sec:motivation}

We compare the filters from Section~\ref{sec:algorithms_uq} in a controlled Lorenz--96 experiment designed to impose severe ensemble undersampling. 
The study addresses three key details.
First, we examine calibration at the baseline ensemble size $N=10$. 
Second, we determine whether the observed behavior persists after burn-in and across update types and cadences. 
Third, we assess how the accuracy comparison changes as the ensemble size increases.


\subsection{Lorenz--96 Model}\label{subsec:lorenz96_problem}

We consider the Lorenz--96 model \citep{lorenz1996predictability}, which is a standard test problem for ensemble data assimilation and targeted-observation studies \citep{houtekamer1998data,lorenz1998optimal,bishop2001etkf,hunt2007letkf}. The state $\mathbf{x}=(x_1,\ldots,x_{40})^\top\in\mathbb{R}^{40}$ satisfies
\begin{equation}
\frac{\mathrm{d}x_i}{\mathrm{d}t}
=
(x_{i+1}-x_{i-2})x_{i-1}-x_i+F,
\qquad i=1,\ldots,40,
\label{eq:lorenz96}
\end{equation}
with periodic indexing and forcing $F=8$. The sustained chaotic dynamics and modest computational cost permit repeated long-horizon cycling experiments. We integrate \eqref{eq:lorenz96} using the classical fourth-order Runge--Kutta method with timestep $\Delta t=0.01$.

\subsection{Observation and Filter Configuration}\label{subsec:benchmark}

We observe every second state component, so that $m=20$, through the linear model
\begin{equation}
\mathbf{z}_k
=
\mathbf{H}\mathbf{x}_k^\dagger+\boldsymbol{\eta}_k,
\qquad
\boldsymbol{\eta}_k\sim\mathcal{N}(\mathbf{0},\sigma_{\mathrm{obs}}^2\mathbf{I}_m),
\qquad
\sigma_{\mathrm{obs}}=1.5,
\label{eq:lorenz96_observation}
\end{equation}
where $\mathbf H$ extracts $x_1,x_3,\ldots,x_{39}$. Observations are available every $T_{\mathrm{obs}}=0.1$ time units. Each windowed method assimilates $L=5$ consecutive observations. In the notation of Section~\ref{sec:problem}, the resulting stacked observation dimension is $d=mL=100$.

The baseline experiment uses $N=10$. Consequently, the empirical observation-space covariance has rank at most $N-1=9$ in the $100$-dimensional stacked space. QPCA-EnDCF uses the fixed truncation rank $\kappa=1$ and retains one leading direction of the whitened residual covariance in the correction~\eqref{eq:qpca_delta_obs}.

We perform $W=1000$ assimilation windows for each of $100$ independent seeds. Within each seed, all filters use the same truth trajectory, observations, and initial ensemble. Every filter is uninflated, with $\lambda_{\mathrm{infl}}=1$. Sequential filters produce an analysis at every observation time, whereas windowed filters produce one endpoint analysis per window. We evaluate all diagnostics at the common window endpoints. The first $40$ windows are treated as burn-in, and settled summaries use windows $41$ through $1000$. Uncertainty intervals are constructed by resampling complete paired seeds, which retains the temporal and componentwise dependence within each experiment.

We compare Seq-EnKF, Seq-ETKF, 4D-EnKF, 4D-ETKF, and QPCA-EnDCF. 
The four conventional controls form a $2\times2$ comparison of stochastic versus deterministic analysis mechanisms and sequential versus windowed update cadences.
Note that localization is disabled for all methods. 
These comparisons are interpreted at the algorithmic level because the methods differ in more than one component of their analysis constructions. 
In particular, replacing EnKF by ETKF changes both the use of observation perturbations and the anomaly-update operator, while QPCA-EnDCF differs from the controls through its complete rank-restricted analysis construction. Thus, the comparisons test whether QPCA-EnDCF behaves differently from stochastic and deterministic alternatives without explicitly isolating observation perturbations or spectral truncation as the cause of that difference.

\subsection{Calibration and Accuracy at Small Ensemble Size}\label{subsec:spread_skill}

\begin{figure}[htbp]
\centering
\includegraphics[width=\textwidth]{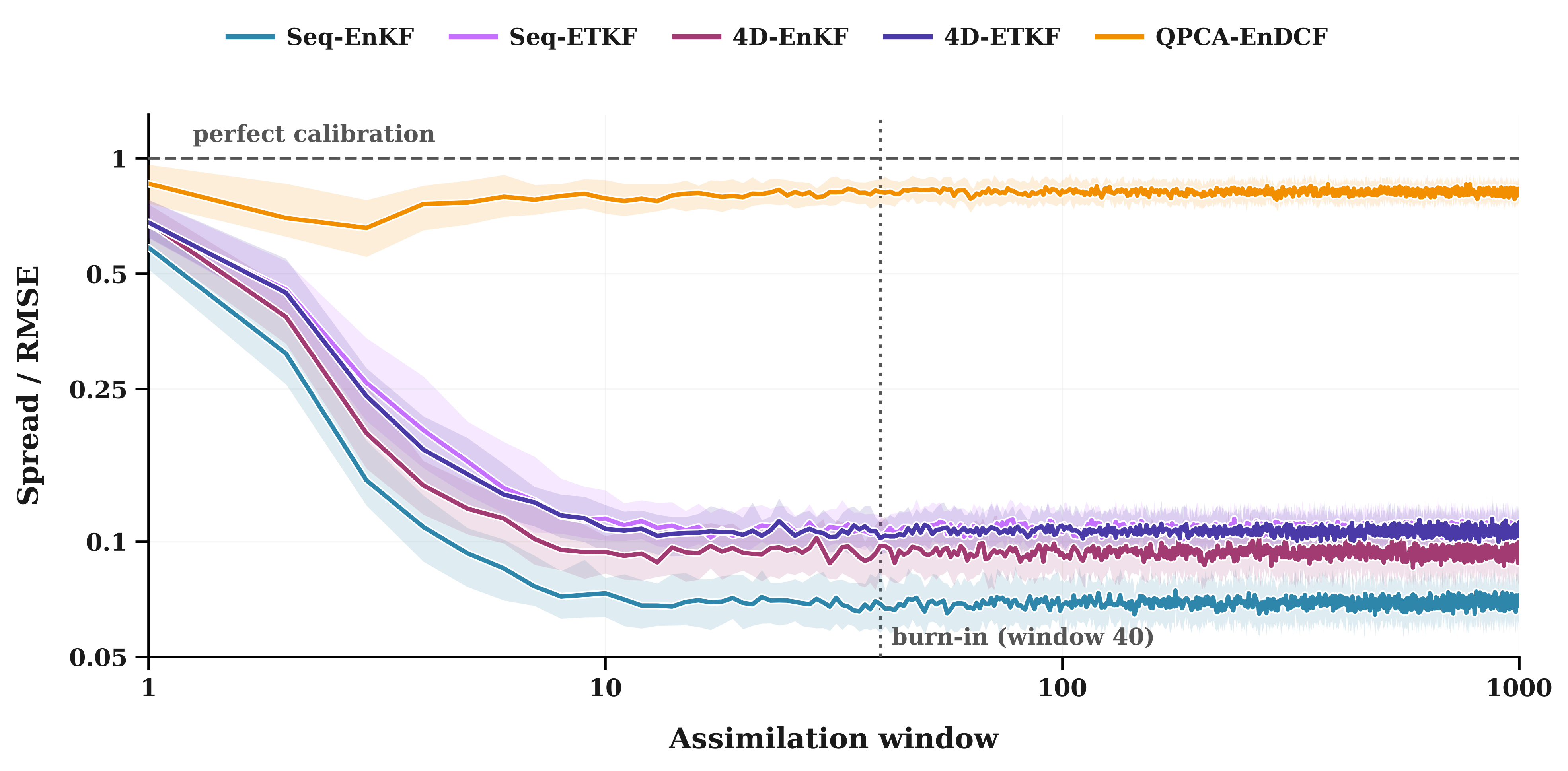}
\caption{Median cycle-wise spread--skill ratio for the five ensemble filters, with interquartile ranges across $100$ paired Lorenz--96 experiments. All methods use $N=10$ and $L=5$, and QPCA-EnDCF uses $\kappa=1$. Sequential analyses are evaluated at the common window endpoints. The horizontal dashed line marks $\gamma_w=1$, and the vertical dotted line marks the end of the $40$-window burn-in period.}
\label{fig:calibration_ratio_ribbons}
\end{figure}

\paragraph{Spread--skill calibration}
For the endpoint analysis ensemble $\mathbf X_w^\ast$, let $\widehat{\mathbf P}_w^\ast$ denote its empirical covariance and $\bar{\mathbf x}_w^\ast$ its mean. We define the cycle-wise spread and root-mean-square error by
\[
\sigma_w
:=
\left[\frac{1}{n}\tr\!\left(\widehat{\mathbf P}_w^\ast\right)\right]^{1/2}
\qquad\text{ and }\qquad
\mathrm{RMSE}_w
:=
\left[\frac{1}{n}\left\|\bar{\mathbf x}_w^\ast-\mathbf x_{k_w}^\dagger\right\|^2\right]^{1/2}.
\]
Their ratio
\begin{equation}
\gamma_w:=\frac{\sigma_w}{\mathrm{RMSE}_w}
\label{eq:gamma_framework}
\end{equation}
compares the uncertainty represented by the ensemble with its realized error.
Values near one indicate agreement between ensemble spread and realized error, whereas values below one indicate underdispersion.
Figure~\ref{fig:calibration_ratio_ribbons} shows the median of $\gamma_w$ and its interquartile range across the $100$ paired experiments at each common window endpoint. This diagnostic displays both the initial adjustment and the settled calibration behavior.



As seen in Figure~\ref{fig:calibration_ratio_ribbons}, the four conventional controls lose most of their relative spread during the first ten windows and settle at spread--skill ratios between $0.069$ and $0.108$. This behavior occurs for both stochastic and deterministic analyses and at both update cadences. QPCA-EnDCF instead settles at $0.817$, with the interquartile bands remaining separated from those of the four controls throughout the settled interval.
Thus, neither replacing the stochastic update by a deterministic ensemble transform nor changing the update cadence prevents severe underdispersion in this configuration. 
QPCA-EnDCF, however, exhibits substantially different behavior, consistent with restricting the correction to a leading empirical direction of the whitened residual covariance. 

\paragraph{Central ensemble coverage}
We next assess distributional calibration through central coverage at the common endpoint times. 
Central ensemble coverage measures how often the truth falls within intervals defined by the ensemble members. 
A well-calibrated ensemble should contain the truth at approximately the nominal frequency of each interval, whereas systematic undercoverage indicates that the ensemble understates its uncertainty.

\begin{figure}[htbp]
\centering
\makebox[\textwidth][c]{\includegraphics[width=\textwidth]{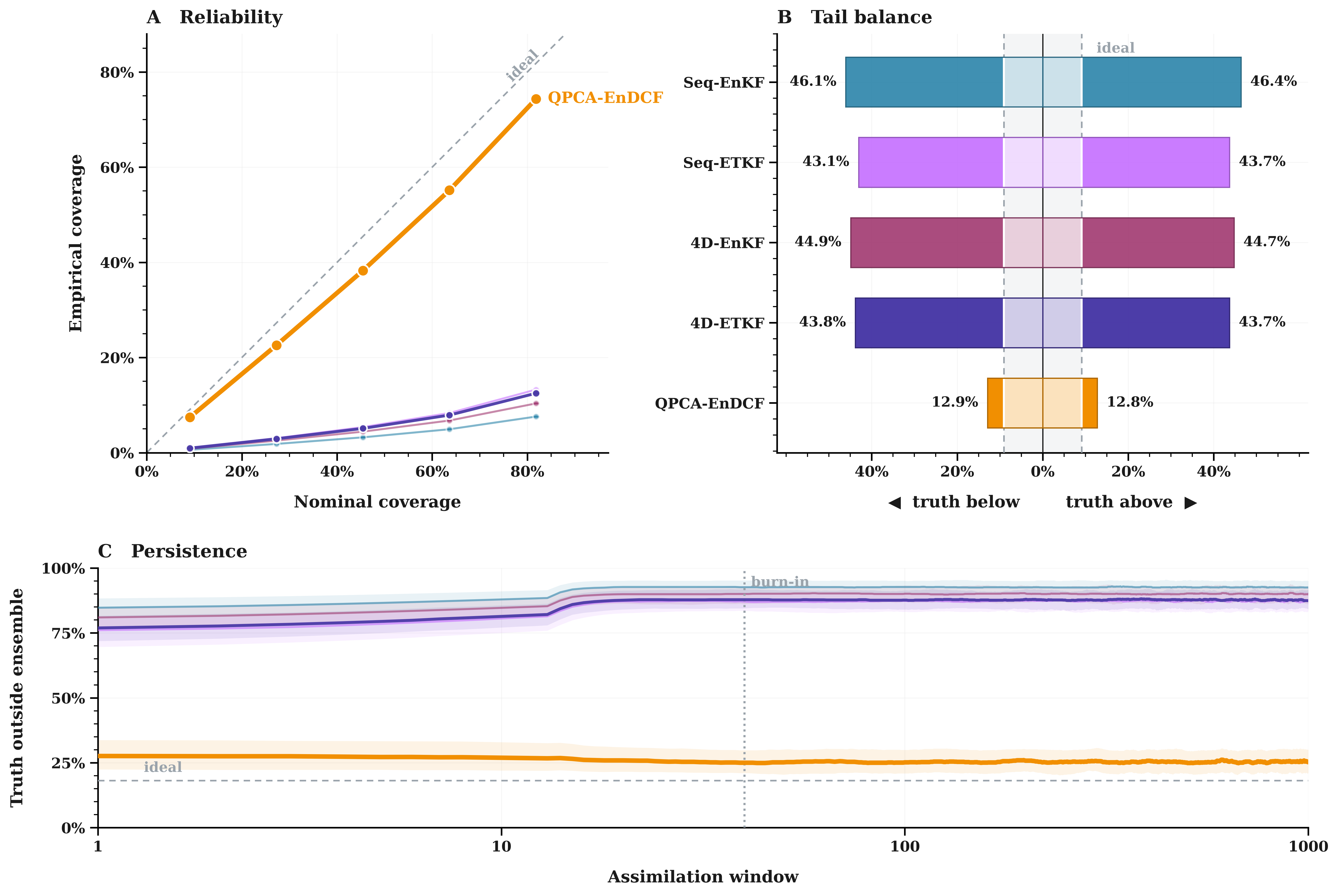}}
\caption{Post-burn-in calibration diagnostics for the five ensemble filters with $N=10$. Panel A compares empirical and nominal coverage for the attainable central ensemble intervals. Panel B reports the frequencies with which the truth falls below and above the ensemble range. Panel C shows the outside-ensemble frequency across assimilation windows, with medians and interquartile ranges over $100$ paired experiments. Settled summaries use windows $41$ through $1000$.}
\label{fig:postburnin_coverage}
\end{figure}

If the truth is exchangeable with an ensemble of size $N$, meaning that the truth and ensemble members are statistically indistinguishable under relabeling (i.e., the truth has no statistically privileged position relative to the ensemble members), its rank is uniform over $N+1$ positions. 
Removing $q$ ranks from each tail gives nominal central coverage $(N-2q-1)/(N+1)$. For $N=10$, the outer interval defined by the ensemble minimum and maximum therefore has nominal coverage $9/11=81.8\%$, with a calibrated outside rate of $2/11=18.2\%$. Figure~\ref{fig:postburnin_coverage} compares empirical and nominal coverage, separates misses below and above the ensemble range, and shows whether the outside rate persists through the cycling experiment.

Over the settled interval, the outer ranges of the four controls cover only $7.6\%$ to $13.3\%$ of the component-window truths. QPCA-EnDCF covers $74.3\%$, substantially closer to the nominal $81.8\%$ but still indicating residual underdispersion. For every filter, the lower- and upper-tail frequencies differ by at most $0.6$ percentage points, indicating that the coverage failures are nearly symmetric rather than persistently one-sided. The outside rates remain stable after burn-in, showing that the behavior persists beyond initialization.

\paragraph{Analysis accuracy}
Calibration does not by itself imply analysis accuracy. We therefore compare the filters with a simple reference that ignores the current observations and always predicts the long-run mean state.
For the truth trajectories considered here, this reference RMSE is approximately $3.64$. At $N=10$, QPCA-EnDCF has a settled median RMSE of $3.76$, compared with $4.62$--$4.88$ for the four conventional controls. 
Thus, QPCA-EnDCF is better calibrated and has the smallest analysis error among the tested filters, but this improvement relative to the filtering controls is not yet sufficient to outperform the reference at the baseline ensemble size.

\subsection{Ensemble-Size Sensitivity}\label{subsec:ensemble_size_calibration}

\begin{figure}[htbp]
\centering
\makebox[\textwidth][c]{\includegraphics[width=\textwidth]{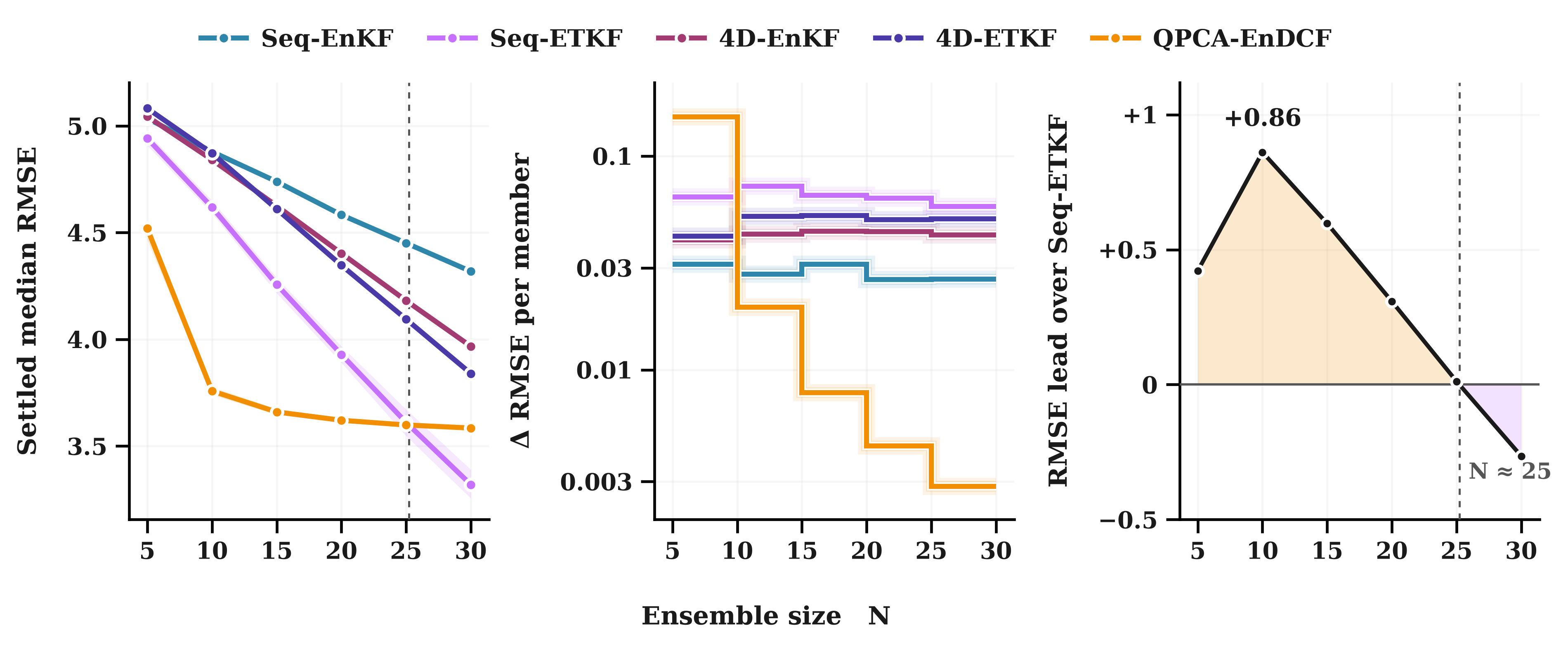}}
\caption{Settled analysis RMSE as a function of ensemble size. The left panel reports median RMSE over windows $41$ through $1000$, with interquartile ranges shown for QPCA-EnDCF and Seq-ETKF. The center panel gives the decrease in RMSE per additional ensemble member over each tested interval. The right panel shows the RMSE advantage of QPCA-EnDCF relative to Seq-ETKF. All settings other than $N$ are fixed, including $\kappa=1$ for QPCA-EnDCF.}
\label{fig:calibration_ensemble_size}
\end{figure}

We next determine whether the baseline comparison is specific to $N=10$. To that end, we repeat the complete experiment for $N\in\{5,10,15,20,25,30\}$
using the same $100$ seeds and holding every other setting fixed, including $\kappa=1$. Figure~\ref{fig:calibration_ensemble_size} reports three complementary summaries: the settled analysis RMSE (left plot) and its improvement with increasing ensemble size (center plot), together with the RMSE difference between QPCA-EnDCF and Seq-ETKF (right plot).

QPCA-EnDCF has the lowest RMSE for the smallest ensembles. Relative to Seq-ETKF, its RMSE advantage is $0.42$ at $N=5$ and $0.86$ at $N=10$, but the difference vanishes near $N=25$. QPCA-EnDCF first outperforms the long-run-mean reference at $N=20$, whereas Seq-ETKF does so at $N=25$. Its RMSE improvement per additional ensemble member decreases sharply beyond $N=10$, while the four conventional controls continue to improve over the tested range.
The fixed choice $\kappa=1$ is important when interpreting this plateau. The experiment does not determine whether increasing or adaptively selecting the truncation rank would alter the larger-$N$ behavior, nor does it establish a scaling law for $\kappa$ as a function of $N$.

Taken together, these experiments locate the clearest QPCA-EnDCF advantage in the strongly undersampled regime for which the method was designed. They do not, however, determine how systematic error and variability across finite-ensemble realizations contribute to the observed mean-squared error. Section~\ref{sec:bias_variance_results} addresses this question through an empirical bias--variance decomposition.

\section{Empirical Bias--Variance Analysis}\label{sec:bias_variance_results}

We use two complementary experiments to determine why the filters have different analysis MSEs. The first uses long cycled runs at the baseline choice $N=10$ and $\kappa=1$ to separate systematic error from variability across finite-ensemble realizations under repeated forecast--analysis feedback. The second uses independent restarts to examine how ensemble size $N$ and truncation rank $\kappa$ affect this bias--variance tradeoff when each analysis begins from a freshly sampled ensemble rather than the result of previous assimilation cycles.

\subsection{Cycled Bias--Variance Decomposition}\label{subsec:hierarchical_bias_variance}

The cycled experiment uses repeated conditional realizations to separate the systematic error associated with a fixed truth and observation sequence from the variability that arises when the same experiment is repeated with different finite ensembles and algorithmic randomness. The sampling design has two levels. The outer level consists of $N_{\mathrm{truth}}=100$ independent truth trajectories and observation-noise realizations. Conditional on each truth and observation sequence, the inner level consists of $N_{\mathrm{trial}}=100$ matched ensemble/filter realizations. Each realization uses $N=10$ and is cycled for $W=1000$ assimilation windows. We compare Seq-EnKF, Seq-ETKF, 4D-EnKF, 4D-ETKF, and QPCA-EnDCF using common random numbers across filters, with no localization or inflation; QPCA-EnDCF uses $\kappa=1$.

For paired comparisons, we use a complete-case sample so that every reported truth sequence is represented by all five filters. 
From 117 candidate truth seeds, we retain the 104 for which every filter completes and use the first 100. Twelve exclusions are due to Seq-ETKF divergence and one to Seq-EnKF divergence; QPCA-EnDCF, 4D-EnKF, and 4D-ETKF have no excluded truth seeds. The reported Seq-ETKF statistics are therefore conditional on successful completion and, relative to an unconditional comparison, favor that method.

For truth sequence $r$, inner realization $t$, and window $w$, let $\bar{\mathbf{x}}_{r,t,w}^\ast$ denote the analysis ensemble mean and let $\mathbf{x}_{r,k_w}^\dagger$ denote the corresponding true state. We define the analysis-mean error by $\mathbf e_{r,t,w}:=\bar{\mathbf{x}}_{r,t,w}^\ast-\mathbf{x}_{r,k_w}^\dagger$. Its mean over the inner realizations, conditional on truth sequence $r$ and window $w$, is $\bar{\mathbf e}_{r,w}:=N_{\mathrm{trial}}^{-1}\sum_{t=1}^{N_{\mathrm{trial}}}\mathbf e_{r,t,w}$.
Averaging over the inner realizations estimates the systematic error for
the fixed truth and observation sequence, while deviations from this
average measure sensitivity to the finite ensemble and algorithmic
randomness.
The state-dimension-normalized estimators are
\begin{align}
\widehat{\mathrm{Bias}}_{r,w}^{2}
&=\frac{1}{n}\|\bar{\mathbf e}_{r,w}\|^{2},
&
\widehat{\mathrm{Variance}}_{r,w}
&=\frac{1}{nN_{\mathrm{trial}}}
  \sum_{t=1}^{N_{\mathrm{trial}}}
  \|\mathbf e_{r,t,w}-\bar{\mathbf e}_{r,w}\|^{2},
\label{eq:empirical_bias_variance}\\
\widehat{\mathrm{MSE}}_{r,w}
&=\widehat{\mathrm{Bias}}_{r,w}^{2}
 +\widehat{\mathrm{Variance}}_{r,w},
&
\widehat{\mathrm{Bias}}_{r,w,\mathrm{db}}^{2}
&=\widehat{\mathrm{Bias}}_{r,w}^{2}
 -\frac{\widehat{\mathrm{Variance}}_{r,w}}{N_{\mathrm{trial}}-1}.
\label{eq:empirical_mse_identity}
\end{align}
For each computed $(r,w)$ sample, the plug-in squared bias and empirical variance sum exactly to the displayed MSE. 
Because the plug-in squared bias also contains finite-sample variability from estimating the conditional mean error with $N_{\mathrm{trial}}$ realizations, aggregate inferential comparisons instead use the debiased squared-bias estimator in the second identity of \eqref{eq:empirical_mse_identity}.
The empirical variance measures all sensitivity to the conditional ensemble and algorithmic randomness. 
It therefore includes, rather than separately identifies, effects from observation perturbations, gain fluctuations, projector fluctuations, and differences in the anomaly-update operator.

Figure~\ref{fig:hierarchical_bias_variance_summary} summarizes the full-run bias--variance composition and its temporal stability. 
The left plot shows the averages of the plug-in squared bias and variance over all windows and truth sequences, while the right plot illustrates the 25-window rolling means of these components as percentage deviations from their post-burn-in run means.

\begin{figure}[htbp]
\centering
\makebox[\textwidth][c]{\includegraphics[width=\textwidth]{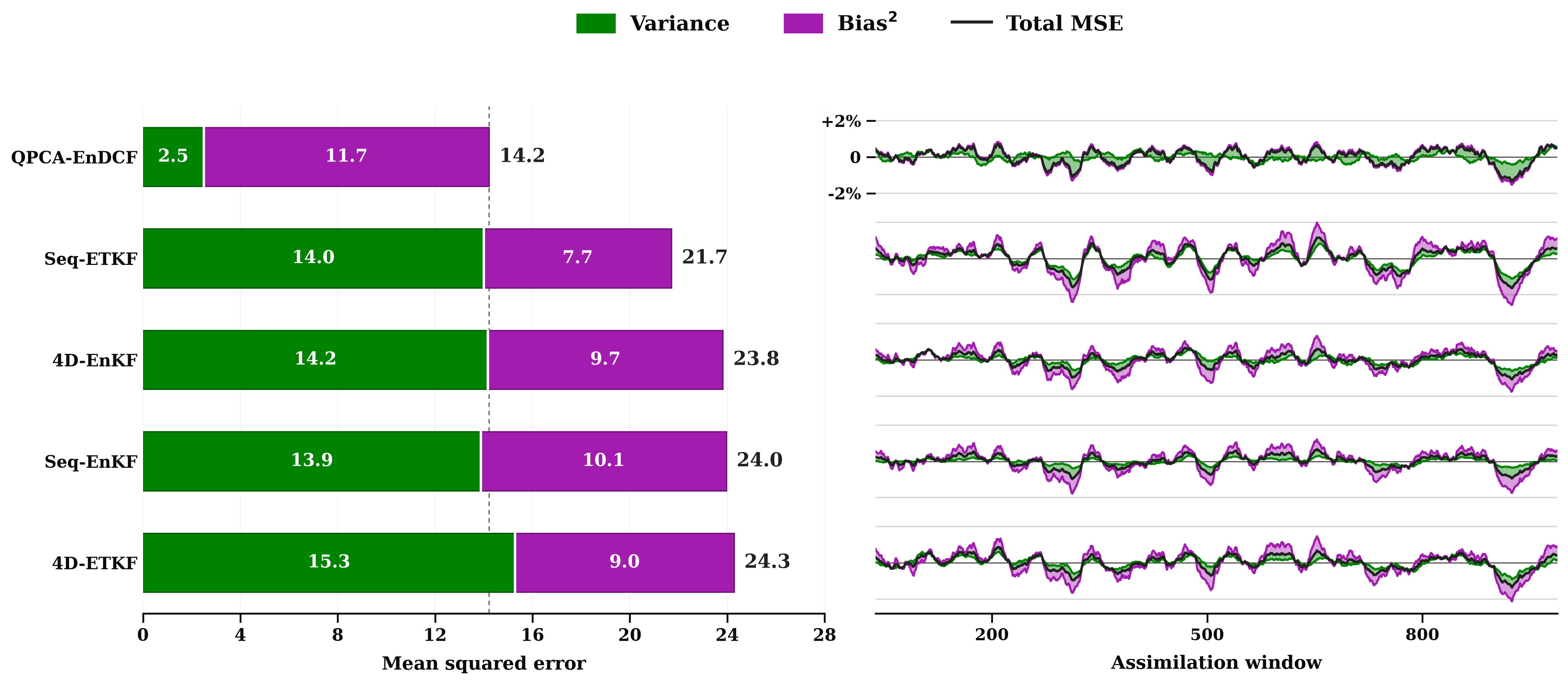}}
\caption{Full-run bias--variance composition in the hierarchical Lorenz--96 experiment. Left: mean plug-in squared bias, across-realization variance, and their sum over $100$ truth and observation sequences, $100$ conditional ensemble/filter realizations, and $1000$ cycled assimilation windows. Right: 25-window rolling means of the same quantities, expressed as percentage deviations from each method's run mean over the post-burn-in interval. All methods use $N=10$; QPCA-EnDCF uses $\kappa=1$.}
\label{fig:hierarchical_bias_variance_summary}
\end{figure}

QPCA-EnDCF exchanges greater squared bias for substantially lower variance across realizations. Its squared bias is $11.73$, the largest of the five methods, but its variance is only $2.49$, compared with $13.87$--$15.28$ for the four conventional controls. Thus, QPCA-EnDCF reduces the across-realization variance by $82.1\%$--$83.7\%$, or by a factor of approximately $5.6$--$6.1$. This reduction more than offsets the increased squared bias: its MSE is $14.22$, compared with $21.72$--$24.28$ for the controls, a reduction of approximately $34.5\%$--$41.5\%$.

The total-MSE traces remain within $1.6\%$ to $2.8\%$ of their respective post-burn-in run means, with QPCA-EnDCF remaining within $2.0\%$. 
The observed bias--variance composition is therefore persistent through the cycling experiment rather than being driven by a short interval.
The ETKF controls further show that eliminating observation perturbations through a deterministic square-root update does not reproduce the QPCA-EnDCF variance reduction in this configuration.

\subsection{Effects of Ensemble Size and Truncation Rank}\label{subsec:n_kappa_bias_variance}

The second experiment examines how the QPCA-EnDCF bias--variance tradeoff changes with ensemble size $N$ and truncation rank $\kappa$. 
To separate these effects from repeated cycling, each analysis begins from a freshly sampled ensemble and the ensemble is discarded after that analysis. 
Thus, changing $N$ or $\kappa$ does not alter the starting distribution used in subsequent analyses. 
The empirical gain and spectral projector remain estimated from the finite ensemble and therefore remain random.

We consider the grid
\begin{equation}
N\in\{5,10,20,40\},
\qquad
\kappa\in\{1,2,4,8\},
\qquad
\kappa\le N-1.
\label{eq:n_kappa_grid}
\end{equation}
The rank constraint excludes $(N,\kappa)=(5,8)$ and leaves $15$ valid design points.
Each point uses $N_{\mathrm{trial}}=100$ conditional realizations and $100$ evaluated windows. 
To preserve a common set of truth sequences across all design points and companion experiments, we use 49 complete truth sequences, excluding one truth seed because QPCA-EnDCF became non-finite in the cycled companion run at $(N,\kappa)=(40,8)$.

At each design point, we use the debiased squared-bias estimator. 
To compare bias and variance on an additive scale, define
\begin{equation}
T
:=
\widehat{\mathrm{Bias}}_{\mathrm{db}}^2
+
\widehat{\mathrm{Variance}}
=
\widehat{\mathrm{MSE}}
-
\frac{\widehat{\mathrm{Variance}}}
{N_{\mathrm{trial}}-1}.
\label{eq:decomposition_consistent_total}
\end{equation}
Figure~\ref{fig:n_kappa_bias_variance_surfaces} shows debiased squared bias, across-realization variance, and their total $T$ over the $(N,\kappa)$ grid. 
Reading down a column shows the response to $N$ at fixed $\kappa$, while reading across a row shows the response to $\kappa$ at fixed $N$.

\begin{figure}[htbp]
\centering
\makebox[\textwidth][c]{\includegraphics[width=\textwidth]{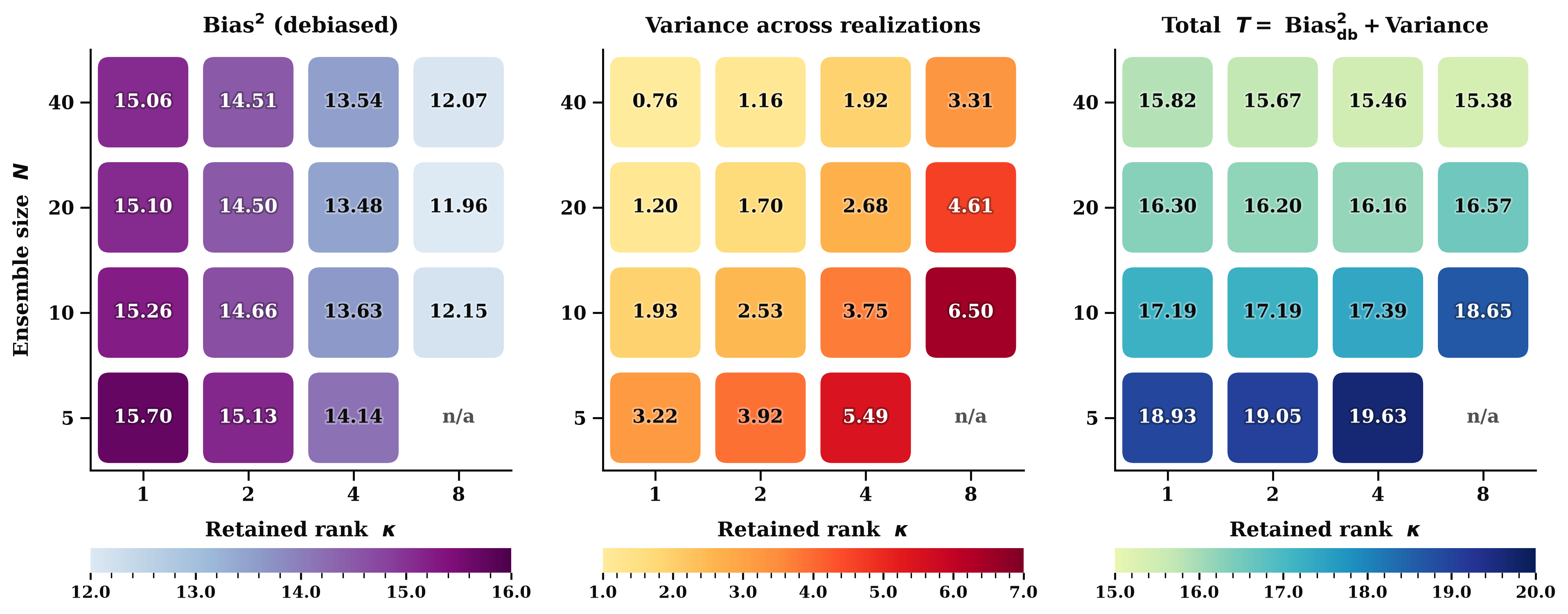}}
\caption{Bias--variance response to ensemble size and retained rank under independent restarts. The panels show debiased squared bias, across-realization variance, and the decomposition-consistent total $T=\widehat{\mathrm{Bias}}_{\mathrm{db}}^2+\widehat{\mathrm{Variance}}$ over $N\in\{5,10,20,40\}$ and $\kappa\in\{1,2,4,8\}$. Values are averaged over $49$ complete truth sequences, $100$ conditional realizations, and $100$ evaluated windows. The unavailable cell is excluded by $\kappa\le N-1$.}
\label{fig:n_kappa_bias_variance_surfaces}
\end{figure}

At each fixed $\kappa$, variance decreases as $N$ increases. 
Descriptive log-log slopes over the admissible ensemble sizes range from approximately $-0.49$ to $-0.69$, supporting the expected direction of decreasing finite-ensemble variability but not establishing an asymptotic rate.
Only four ensemble sizes are tested, and the ensemble remains rank deficient in the $d=100$ stacked observation space at every tested value.

At each fixed $N$, debiased squared bias decreases and variance increases over the tested values of $\kappa$. The bias direction is qualitatively consistent with the discarded population mean-residual component
\begin{equation}
\left\|
(\mathbf I-\boldsymbol{\Pi}_{\kappa,s})
\boldsymbol\mu_{E,s}
\right\|^2,
\label{eq:section5_discarded_mean_component}
\end{equation}
which is non-increasing in $\kappa$ for nested population projectors. 
This population quantity governs the leading truncation contribution but is not identical to the plotted empirical debiased squared bias. 
The empirical variance response is consistent with greater finite-ensemble sensitivity when more directions are retained, although the theory does not prove monotonicity in $\kappa$.

Because the two components move in opposite directions, $T$ varies much less than either component. 
The minimizing ranks over the tested grid are $\kappa=1,2,4,8$ for $N=5,10,20,40$, respectively, suggesting that larger ensembles can support less aggressive truncation.
However, this pattern should not be interpreted as a scaling law since two minima occur at the boundary of the tested rank range, and the two best totals at $N=10$ differ by only $0.007$.

\section{High-Dimensional Shallow-Water Experiment}
\label{sec:idealized_inlet_filter}

The Lorenz--96 experiments provide controlled statistical evidence for the QPCA-EnDCF bias--variance tradeoff. We next examine the method in a nonlinear shallow-water model for which both the state dimension and the stacked observation dimension greatly exceed the ensemble size. The study considers four storm--localization configurations formed by combining slow and fast vortex transit timescales with Gaspari--Cohn localization radii of $2$ and $10\,\text{km}$. These configurations test whether the relative filter behavior persists as the forcing timescale and localization scale change. The experiment addresses three questions. First, we determine whether the observation-mismatch construction can be implemented without assembling the state covariance. Second, we assess whether the filters remain finite through a physically prescribed shelf-crossing endpoint. Third, we examine how the comparison changes with storm timescale and localization radius. The experiment therefore evaluates high-dimensional feasibility and short-horizon behavior rather than repeating the bias--variance analysis of Section~\ref{sec:bias_variance_results}.

\subsection{Shallow-Water Model and Experimental Configuration}\label{subsec:inlet_swe_model}

We consider a two-dimensional coastal inlet governed by the depth-integrated shallow-water equations. The model resolves nonlinear surge propagation while providing a tractable high-dimensional setting for evaluating the filters under storm forcing. We write the governing equations in conservative form as
\begin{equation}
\partial_t\mathbf{Q}+\operatorname{div}\!\bigl(\mathbf{F}(\mathbf{U})\bigr)+\mathbf{f}(\mathbf{U})=\mathbf{0},
\label{eq:inlet_swe_conservation}
\end{equation}
where the conserved and primitive variables are, respectively,
\begin{equation}
\mathbf{Q}=(H,uH,vH)^\top,\qquad \mathbf{U}=(H,u,v)^\top.
\end{equation}
Here, $H$ is water depth and $u$ and $v$ are the depth-averaged velocity components.

The flux tensor is
\begin{equation}
\mathbf{F}(\mathbf{U})=
\begin{pmatrix}
Hu & Hv\\
Hu^2+\frac{1}{2}g(H^2-h_b^2) & Huv\\
Huv & Hv^2+\frac{1}{2}g(H^2-h_b^2)
\end{pmatrix},
\label{eq:inlet_swe_flux}
\end{equation}

The source terms represent bathymetric slope, bottom friction, Coriolis acceleration, wind stress, and the atmospheric-pressure gradient. The corresponding forcing vector is
\begin{equation}
\mathbf{f}(\mathbf{U})=
\begin{pmatrix}
0\\
-g\zeta\frac{\partial h_b}{\partial x}+\tau_{b_x}-fvH-C_D\frac{\rho_{\mathrm{air}}}{\rho_{\mathrm{water}}}\|\mathbf{w}\|w_x+\frac{H}{\rho_{\mathrm{water}}}\frac{\partial P}{\partial x}\\
-g\zeta\frac{\partial h_b}{\partial y}+\tau_{b_y}+fuH-C_D\frac{\rho_{\mathrm{air}}}{\rho_{\mathrm{water}}}\|\mathbf{w}\|w_y+\frac{H}{\rho_{\mathrm{water}}}\frac{\partial P}{\partial y}
\end{pmatrix}.
\label{eq:inlet_swe_forcing}
\end{equation}
The water-surface elevation is $\zeta=H-h_b$, where $h_b$ is the time-independent bathymetric depth. The parameters $g$, $f$, and $C_D$ denote gravitational acceleration, the Coriolis parameter, and the surface-drag coefficient, respectively. The vector $(\tau_{b_x},\tau_{b_y})^\top$ is the bottom-friction contribution, $\rho_{\mathrm{water}}$ and $\rho_{\mathrm{air}}$ are reference densities, $\mathbf{w}=(w_x,w_y)^\top$ is wind velocity, and $P$ is atmospheric pressure.

For the bottom-friction terms, let $\rho_0:=\rho_{\mathrm{water}}$, $Q_x:=uH$, $Q_y:=vH$, and $\mathbf{u}:=(u,v)^\top$. The friction components satisfy
\begin{equation}
\frac{\tau_{b_x}}{\rho_0}=\frac{K_{\mathrm{slip}}Q_x}{H},
\qquad
\frac{\tau_{b_y}}{\rho_0}=\frac{K_{\mathrm{slip}}Q_y}{H},
\end{equation}
with the quadratic drag law
\begin{equation}
K_{\mathrm{slip}}=c_f\|\mathbf{u}\|,
\qquad
c_f=\frac{g n_{\mathrm{M}}^2}{H^{1/3}},
\label{eq:inlet_manning_drag}
\end{equation}
where $n_{\mathrm{M}}$ is Manning's roughness coefficient, distinguished from the state dimension $n$. It is defined nodewise over the mesh, producing a piecewise linear bottom-friction field.

The notation in this subsection is local to the shallow-water model. In particular, $\mathbf{Q}$, $H$, $P$, and $f$ are distinct from the quantity-of-interest map, observation operator, covariance matrices, and filtering notation introduced earlier. We next specify the inlet geometry, storm forcing, ensemble initialization, and observation network.

\subsubsection{Experimental Configuration}
\label{subsec:inlet_scale_matching_design}

The domain, illustrated in Figure~\ref{fig:inlet_scale_matching_setup}, is a $50\,\text{km}\times30.5\,\text{km}$ coastal strip with $11{,}768$ geometric nodes, $23{,}104$ triangular elements, and approximately $2\times10^5$ state degrees of freedom. The depth is $14\,\text{m}$ at the open-ocean boundary, shoals linearly to $5\,\text{m}$ at $y=20\,\text{km}$, and remains constant across the inner shelf. Wetting and drying are disabled in these runs. We therefore define shelf crossing as passage of the vortex across the bathymetric transition at $y=20\,\text{km}$.

The truth and assimilation-prior forcings use translating Rankine vortices with $V_{\max}=15\,\text{m/s}$ and $R_{\max}=4\,\text{km}$. Their tracks differ by the cross-track displacement $\Delta x_{\mathrm{track}}=5\,\text{km}$. The experiment varies the vortex transit duration and Gaspari--Cohn localization radius in a $2\times2$ design.

The slow vortex traverses the track in $43{,}200\,\text{s}$ ($12\,\text{h}$) over eleven assimilation windows. Near peak forcing, its surge disperses to a characteristic scale of approximately $10$--$20\,\text{km}$. The fast vortex traverses the track in $25{,}200\,\text{s}$ ($7\,\text{h}$) over six windows and produces a more concentrated surge with a characteristic scale of approximately $4\,\text{km}$. Each storm is combined with localization radii $r_{\mathrm{loc}}=2$ and $10\,\text{km}$, giving four storm--localization configurations.

All four configurations use the operating point in Table~\ref{tab:inlet_operating_point}. The timestep is $\Delta t=600\,\text{s}$, and each $60$-minute assimilation window contains $L=3$ observation times separated by $2\Delta t$. At each observation time, approximately $m=500$ surface-elevation observations are sampled over the interior, giving the stacked dimension $d=mL\approx1500$.

\begin{table}[htbp]
\centering
\caption{Experimental configuration for the idealized inlet study. Quantities shared by all runs are listed together with the two varied factors: storm transit time and Gaspari--Cohn localization radius.}
\label{tab:inlet_operating_point}
\setlength{\tabcolsep}{5pt}
\begin{tabular}{lll}
\toprule
Category & Parameter & Value \\
\midrule
Storm & $V_{\max}$; $R_{\max}$; track shift & $15\,\text{m/s}$; $4\,\text{km}$; $5\,\text{km}$ \\
Transit & slow; fast & $12\,\text{h}$; $7\,\text{h}$ \\
Cycling & window duration; number of windows & $60\,\text{min}$; $11$ or $6$ \\
Observations & $m$; $L$; $\sigma_{\mathrm{obs}}$ & $\sim500$; $3$; $0.01\,\text{m}$ \\
Ensemble & $N$; $\kappa$; $\lambda_{\mathrm{infl}}$ & $10$; $1$; $1.05$ \\
Localization & Gaspari--Cohn radius & $2$ or $10\,\text{km}$ \\
Sampling & background standard deviation; realizations & $0.10\,\text{m}$; $3$ matched runs \\
\bottomrule
\end{tabular}
\end{table}

Measured from the run start, including the one-hour ramp, the vortex reaches the shelf transition at approximately $411\,\text{min}$ in the slow configuration and $240\,\text{min}$ in the fast configuration. We compare the endpoint analyses at $W_6$ ($420\,\text{min}$) and $W_4$ ($300\,\text{min}$), respectively. These windows are determined by the prescribed track and bathymetry and do not depend on either filter's error. Figure~\ref{fig:inlet_scale_matching_setup} summarizes the experimental geometry and the corresponding scale relationships.

\begin{figure}[H]
\centering
\includegraphics[width=\textwidth]{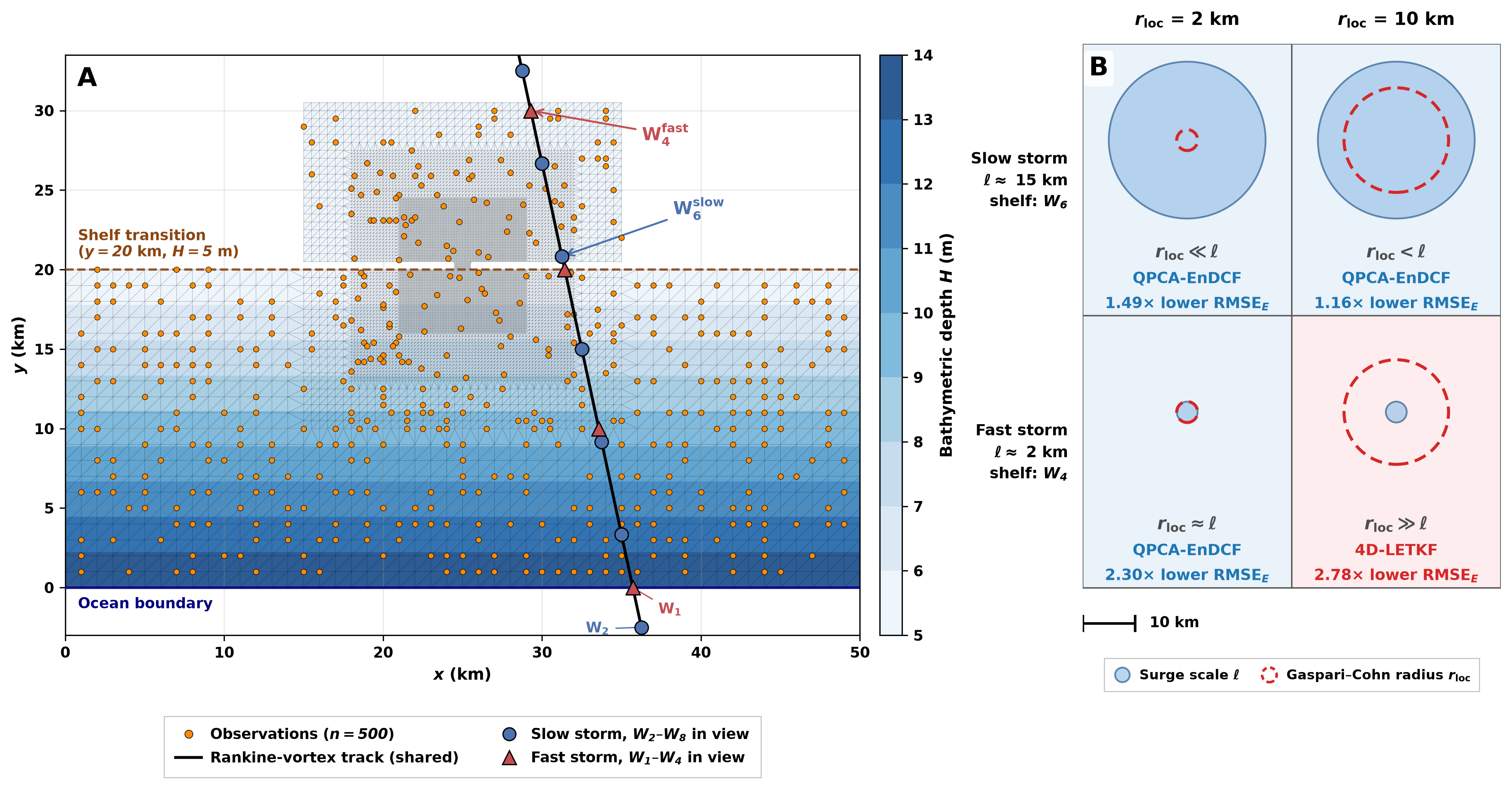}
\caption[Shallow-water experimental configuration.]{Configuration of the idealized inlet experiment. \textbf{A}~Domain, bathymetry, and vortex track: the coastal strip with its linear shelf transition at $y=20\,\text{km}$, the $\sim500$ surface-elevation observation points, and the common Rankine-vortex track used by both the slow and fast storms, including the prescribed $W_6$ and $W_4$ comparison endpoints. \textbf{B}~Localization and surge scales at shelf crossing (schematic): characteristic surge scale $\ell$ against Gaspari--Cohn localization radius $r_{\mathrm{loc}}$ for each of the four $2\times2$ cells, with the three-realization endpoint error ratio and the lower-error filter annotated per cell.}
\label{fig:inlet_scale_matching_setup}
\end{figure}

\subsection{Filter Configuration and Physical Diagnostics}
\label{subsec:inlet_filter_diagnostics}

With $d\approx1500$ and $N=10$, the empirical observation-mismatch covariance has rank at most $N-1=9$. Approximately $1491$ directions in the stacked observation space are therefore unresolved by the ensemble. QPCA-EnDCF uses $\kappa=1$ and retains the leading direction of the empirical mismatch subspace.

The implementation does not assemble the state covariance. Instead, it uses observation-space quantities and a reassociated product for $\mathbf K_{w,\mathrm{DC}}\boldsymbol\Delta_{\mathrm{obs},w}$. The resulting dense operations depend on $d$ and $N$, rather than on the state dimension $n\approx2\times10^5$.

The primary comparison is between QPCA-EnDCF and 4D-LETKF. Within each configuration, both methods use the same Gaspari--Cohn radius, but localization enters their analysis constructions differently. QPCA-EnDCF constructs $\widehat{\boldsymbol\Pi}_{\kappa,w}$ from the full whitened observation mismatch and applies the taper to the state--observation cross-covariance that maps the retained correction to state space. In contrast, 4D-LETKF uses localized observation precision in regional ensemble-space transforms. Because the methods also differ in spectral truncation, update construction, and anomaly transformation, the comparison evaluates the complete algorithms rather than isolating any one of these components.

At the end of each assimilation window, we evaluate the analysis ensemble mean over the diagnostic wet domain
\begin{equation}
\Omega_{\mathrm{wet},w}:=\{\mathbf s\in\Omega:H_{k_w}^{\dagger}(\mathbf s)>h_{\min}\},
\qquad
A_{\mathrm{wet},w}:=\int_{\Omega_{\mathrm{wet},w}}\mathrm dA,
\label{eq:inlet_wet_domain}
\end{equation}
Here, $H_{k_w}^{\dagger}(\mathbf{s})=\zeta_{k_w}^{\dagger}(\mathbf{s})+h_b(\mathbf{s})$ is the true water depth at location $\mathbf{s}$ and at the endpoint $k_w$ of window $w$. The normalization $A_{\mathrm{wet},w}$ is the area of the truth-defined wet domain, with $h_{\min}=0.5\,\text{m}$. This common domain and area normalization permit comparisons across filters and windows even when the wet area changes.

Let $\Delta\zeta_w$, $\Delta u_w$, and $\Delta v_w$ denote the analysis-mean errors in surface elevation and the two velocity components. We define the energy-equivalent error density and the corresponding full-state diagnostic by
\begin{align}
e_{E,w}^{2}
&:=(\Delta\zeta_w)^2+
\frac{H_{k_w}^{\dagger}}{g}\bigl((\Delta u_w)^2+(\Delta v_w)^2\bigr),
\label{eq:inlet_energy_density}\\
\mathrm{RMSE}_{E,w}
&:=\left[\frac{1}{A_{\mathrm{wet},w}}
\int_{\Omega_{\mathrm{wet},w}}e_{E,w}^{2}\,\mathrm dA\right]^{1/2}.
\label{eq:inlet_rmse_energy}
\end{align}
The factor $H_{k_w}^{\dagger}/g$ places the velocity-error contribution in squared-length units. Consequently, $\mathrm{RMSE}_{E,w}$ combines the surface-elevation and velocity errors and is reported in meters.

Let $\sigma_{E,w}$ denote the ensemble spread computed in the same wet-domain norm. We define the spread--skill ratio by
\begin{equation}
\gamma_{E,w}:=\frac{\sigma_{E,w}}{\mathrm{RMSE}_{E,w}}.
\label{eq:inlet_spread_skill_energy}
\end{equation}
This ratio compares the uncertainty represented by the ensemble with its realized analysis error. For a Bessel-corrected exchangeable ensemble whose mean is compared with an independent truth, the leading-order reference is $\sqrt{N/(N+1)}\approx0.953$ at $N=10$. Values well below this reference indicate underdispersion, whereas values well above it indicate overdispersion. After numerical divergence, however, a large value can instead reflect ensemble members diverging in different directions.

\subsection{Short-Horizon Cycling Across Filters}
\label{subsec:inlet_all_filter_cycling}
We additionally compare stochastic 4D-EnKF and deterministic Seq-LETKF as controls at the same four operating points and matched realizations. All four methods are evaluated using $\mathrm{RMSE}_{E,w}$ from~\eqref{eq:inlet_rmse_energy} and $\gamma_{E,w}$ from~\eqref{eq:inlet_spread_skill_energy}. Table~\ref{tab:inlet_all_filter_shelf_trajectories} reports these diagnostics through the prescribed comparison endpoint and shows which methods remain numerically viable over this interval.

\begin{table}[H]
\centering
\caption{Per-window energy-equivalent error $\mathrm{RMSE}_{E,w}$ (meters) and
spread--skill ratio $\gamma_{E,w}$ for the four filters through the prescribed
comparison endpoints. Entries are means over three realizations unless noted, and
$t$ is simulation time measured from the run start, including the one-hour ramp.
Bold marks the lowest
displayed error and the final row of each block; ``floor'' denotes
$\gamma_{E,w}<10^{-3}$. The slow/$10\,\text{km}$ 4D-EnKF entries use two
available runs, and the daggered fast/$2\,\text{km}$ Seq-LETKF endpoint uses
the two finite runs after one realization overflowed. Final-row control values are
post-divergence magnitudes and are not calibration diagnostics.}
\label{tab:inlet_all_filter_shelf_trajectories}
\vspace{2pt}
\scriptsize
\setlength{\tabcolsep}{4.2pt}
\begin{tabular}{@{}l r rr rr rr rr@{}}
\toprule
& & \multicolumn{2}{c}{4D-EnKF} & \multicolumn{2}{c}{Seq-LETKF}
  & \multicolumn{2}{c}{QPCA-EnDCF} & \multicolumn{2}{c}{4D-LETKF} \\
\cmidrule(lr){3-4} \cmidrule(lr){5-6} \cmidrule(lr){7-8} \cmidrule(lr){9-10}
Window & $t$ (min)
  & $\mathrm{RMSE}_{E,w}$ & $\gamma_{E,w}$
  & $\mathrm{RMSE}_{E,w}$ & $\gamma_{E,w}$
  & $\mathrm{RMSE}_{E,w}$ & $\gamma_{E,w}$
  & $\mathrm{RMSE}_{E,w}$ & $\gamma_{E,w}$ \\
\midrule
\multicolumn{10}{@{}l}{\textit{Slow storm, $2$~km localization}} \\
\addlinespace[1pt]
$W_1$ & 120 & $0.020$ & 0.99 & $0.027$ & 0.62 & $0.019$ & 1.00 & $\mathbf{0.018}$ & 0.94 \\
$W_2$ & 180 & $0.224$ & 0.08 & $0.234$ & 0.07 & $\mathbf{0.222}$ & 0.08 & $\mathbf{0.222}$ & 0.07 \\
$W_3$ & 240 & $12.7$ & 0.05 & $0.640$ & 0.03 & $\mathbf{0.634}$ & 0.03 & $0.636$ & 0.03 \\
$W_4$ & 300 & $1.5{\times}10^{5}$ & 1.18 & $0.846$ & 0.07 & $\mathbf{0.764}$ & 0.08 & $0.786$ & 0.07 \\
$W_5$ & 360 & $5.5{\times}10^{16}$ & 0.97 & $7.17$ & 1.34 & $\mathbf{1.03}$ & 0.21 & $1.28$ & 0.16 \\
\addlinespace[2pt]
$\mathbf{W_6}$ & \textbf{420} & $1.9{\times}10^{39}$ & 3.24 & $5.1{\times}10^{56}$ & floor & $\mathbf{44.5}$ & 0.099 & $66.1$ & 0.037 \\
\addlinespace[3pt]
\multicolumn{10}{@{}l}{\textit{Slow storm, $10$~km localization}} \\
\addlinespace[1pt]
$W_1$ & 120 & $0.043$ & 0.74 & $0.027$ & 0.51 & $\mathbf{0.019}$ & 0.99 & $\mathbf{0.019}$ & 0.70 \\
$W_2$ & 180 & $0.662$ & 0.10 & $0.235$ & 0.06 & $\mathbf{0.222}$ & 0.07 & $0.224$ & 0.06 \\
$W_3$ & 240 & $5.6{\times}10^{4}$ & 0.02 & $0.653$ & 0.03 & $\mathbf{0.634}$ & 0.03 & $0.640$ & 0.02 \\
$W_4$ & 300 & $7.4{\times}10^{23}$ & 1.62 & $0.930$ & 0.02 & $\mathbf{0.786}$ & 0.07 & $0.948$ & 0.03 \\
$W_5$ & 360 & $8.3{\times}10^{33}$ & 6.75 & $9.93$ & 0.03 & $\mathbf{1.12}$ & 0.23 & $1.67$ & 0.04 \\
\addlinespace[2pt]
$\mathbf{W_6}$ & \textbf{420} & $1.1{\times}10^{34}$ & 11.5 & $1.3{\times}10^{41}$ & floor & $\mathbf{43.8}$ & 0.070 & $50.7$ & 0.001 \\
\addlinespace[3pt]
\multicolumn{10}{@{}l}{\textit{Fast storm, $2$~km localization}} \\
\addlinespace[1pt]
$W_1$ & 120 & $0.234$ & 0.09 & $0.253$ & 0.07 & $0.234$ & 0.08 & $\mathbf{0.233}$ & 0.07 \\
$W_2$ & 180 & $0.648$ & 0.03 & $\mathbf{0.637}$ & 0.03 & $0.648$ & 0.03 & $0.648$ & 0.03 \\
$W_3$ & 240 & $1.1{\times}10^{4}$ & 0.01 & $20.6$ & 1.62 & $\mathbf{1.06}$ & 1.55 & $2.02$ & 0.56 \\
\addlinespace[2pt]
$\mathbf{W_4}$ & \textbf{300} & $2.2{\times}10^{9}$ & 4.51 & ${2.5{\times}10^{67}}^{\dagger}$ & floor & $\mathbf{10.9}$ & 3.15 & $25.0$ & 2.12 \\
\addlinespace[3pt]
\multicolumn{10}{@{}l}{\textit{Fast storm, $10$~km localization}} \\
\addlinespace[1pt]
$W_1$ & 120 & $0.312$ & 0.11 & $0.253$ & 0.06 & $0.234$ & 0.08 & $\mathbf{0.233}$ & 0.06 \\
$W_2$ & 180 & $2.19$ & 0.03 & $0.664$ & 0.03 & $\mathbf{0.648}$ & 0.03 & $0.661$ & 0.03 \\
$W_3$ & 240 & $7.0{\times}10^{6}$ & 0.004 & $93.0$ & floor & $\mathbf{1.29}$ & 1.98 & $3.34$ & 0.06 \\
\addlinespace[2pt]
$\mathbf{W_4}$ & \textbf{300} & $4.4{\times}10^{20}$ & 6.70 & $1.8{\times}10^{56}$ & floor & $39.6$ & 3.16 & $\mathbf{14.2}$ & 0.004 \\
\bottomrule
\end{tabular}
\end{table}

Through $W_2$, every filter has $\mathrm{RMSE}_{E,w}\leq2.19\,\text{m}$ in each configuration. The methods separate thereafter. At $W_3$, 4D-EnKF reaches errors from $1.1\times10^4$ to $7.0\times10^6\,\text{m}$ in three configurations, while the other filters remain below $100\,\text{m}$.

Seq-LETKF remains close to the primary pair longer in the slow-storm configurations, but separates at $W_3$ in both fast-storm configurations. By the shelf-crossing endpoint, its error has increased by many orders of magnitude and $\gamma_{E,w}<10^{-3}$ in all four configurations. The small ratio occurs because the error grows much faster than the ensemble spread; it does not imply that the absolute spread approaches zero.

At the prescribed endpoints, only QPCA-EnDCF and 4D-LETKF retain finite energy-equivalent errors in all four configurations. Their endpoint errors, approximately $11$--$66\,\text{m}$, remain dimensionally defined and diagnostically interpretable, unlike the post-divergence control values. They are nevertheless large relative to the $5$--$14\,\text{m}$ inlet depths and should not be interpreted as physically accurate storm-surge estimates. The endpoint accuracy and calibration comparisons in Subsection~\ref{subsec:inlet_scale_matching_results} therefore provide a test of comparative short-horizon feasibility rather than operational prediction skill.
\FloatBarrier

\subsection{Shelf-Crossing Accuracy, Calibration, and Scope}
\label{subsec:inlet_scale_matching_results}

Figure~\ref{fig:inlet_scale_matching_trajectories} reports the per-window energy-equivalent error and spread--skill ratio for QPCA-EnDCF and 4D-LETKF.

\begin{figure}[H]
\centering
\includegraphics[width=\textwidth]{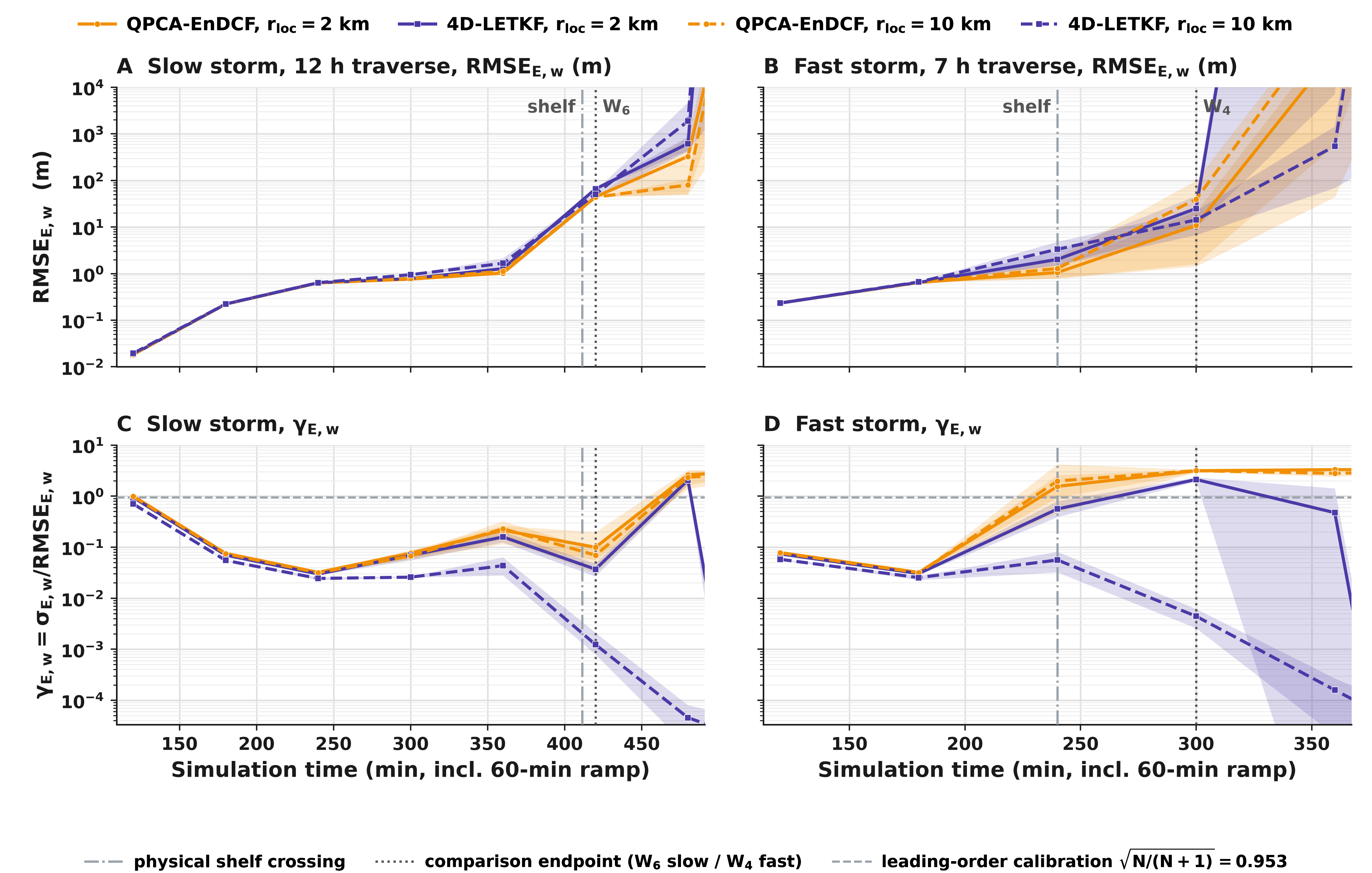}
\caption[Shallow-water error and spread trajectories.]{Per-window error and spread--skill trajectories for QPCA-EnDCF and 4D-LETKF. The upper panels show $\mathrm{RMSE}_{E,w}$, and the lower panels show $\gamma_{E,w}$. Thin curves represent individual realizations and thick curves represent arithmetic means; color identifies the method and line style identifies the localization radius. Vertical lines mark shelf crossing and the prescribed comparison endpoint. Triangles mark overflow or a ratio below the displayed floor, and the horizontal dashed line marks $\sqrt{N/(N+1)}\approx0.953$.}
\label{fig:inlet_scale_matching_trajectories}
\end{figure}

The slow-storm trajectories identify shelf crossing as the transition out of the short-horizon accuracy regime. Both methods remain below $1\,\text{m}$ through $W_4$, increase to errors on the order of $1\,\text{m}$ at $W_5$, and reach errors on the order of $10\,\text{m}$ at the $W_6$ comparison endpoint. Their errors then increase sharply at $W_7$, with the deterioration of 4D-LETKF especially pronounced at the $10\,\text{km}$ radius. The prescribed $W_6$ endpoint therefore captures the last common comparison before post-shelf error growth dominates the trajectories.

The fast storm compresses the same transition into fewer windows and reveals a stronger interaction with localization radius. At $2\,\text{km}$, QPCA-EnDCF has the lower mean before and at shelf crossing. At $10\,\text{km}$, it has the lower mean at $W_3$, but the ordering reverses at the $W_4$ endpoint. The fast-storm result is therefore not a uniform method advantage; it depends on the storm--localization configuration. Both methods subsequently enter a rapid error-growth regime.

The calibration results distinguish uncertainty representation from point accuracy. In the slow-storm configurations, QPCA-EnDCF has $\gamma_{E,w}=0.070$--$0.099$, compared with $0.001$--$0.037$ for 4D-LETKF. QPCA-EnDCF is closer to the reference $0.953$ at both radii, but its spread still represents only approximately $7$--$10\%$ of the realized error.

The fast-storm configurations exhibit two different calibration failures. At $2\,\text{km}$, both methods are overdispersed, with $\gamma_{E,w}=3.15$ for QPCA-EnDCF and $2.12$ for 4D-LETKF. At $10\,\text{km}$, QPCA-EnDCF remains overdispersed with $\gamma_{E,w}=3.16$, whereas 4D-LETKF has the near-zero ratio $0.004$. In particular, the lower 4D-LETKF mean error at fast/$10\,\text{km}$ does not correspond to a better representation of uncertainty. Neither method is consistently calibrated over the four configurations.

Figure~\ref{fig:inlet_energy_metric_sensitivity} compares the endpoint errors and the component contributions to the energy-equivalent diagnostic.

\begin{figure}[H]
\centering
\includegraphics[width=\textwidth]{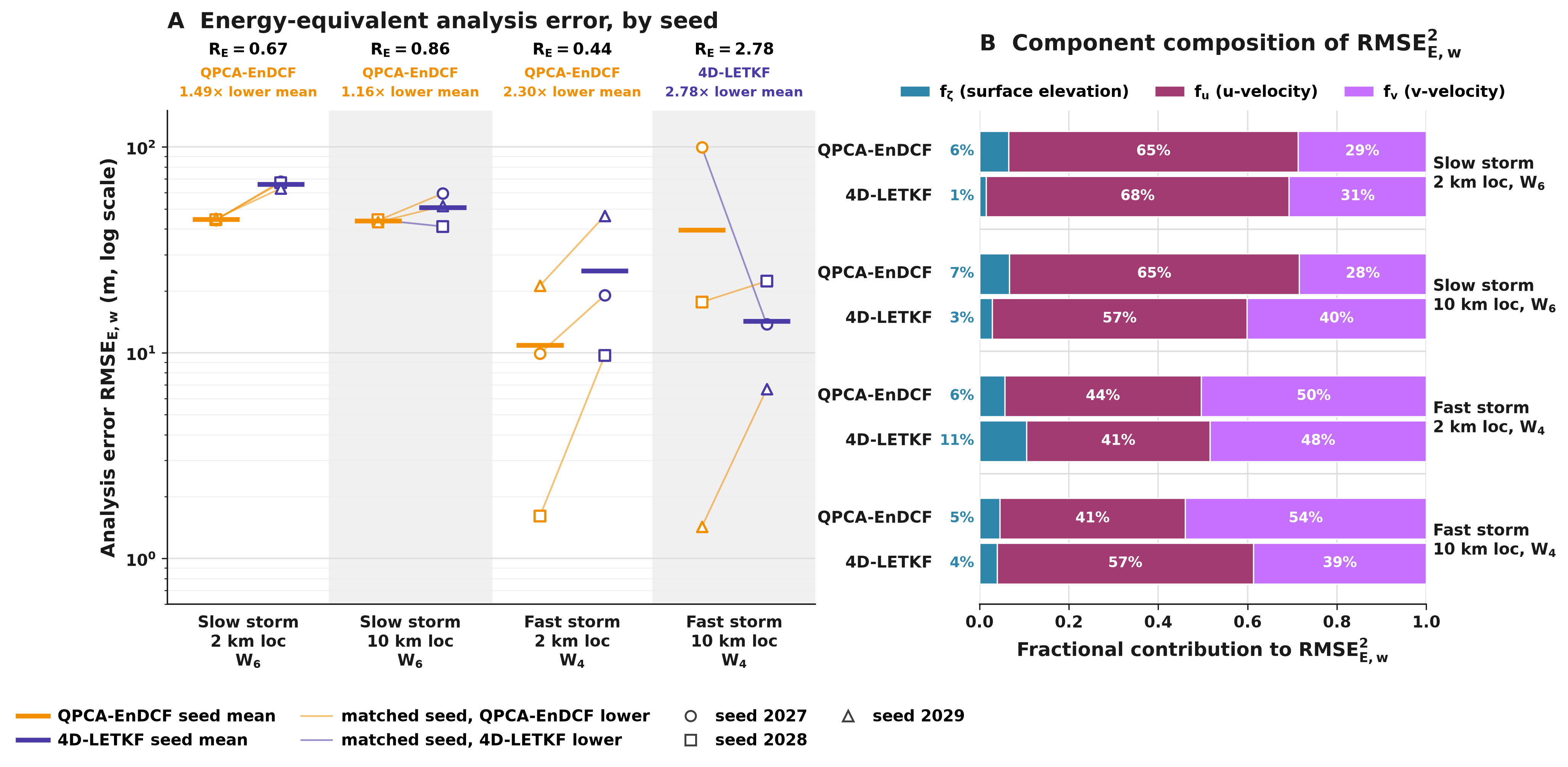}
\caption[Energy-equivalent shelf-crossing error.]{Endpoint energy-equivalent error and component contributions for QPCA-EnDCF and 4D-LETKF. Panel A shows the individual realizations, arithmetic means, and their ratio $R_E$ in each storm--localization configuration. Connector crossings indicate changes in the method ordering across realizations. Panel B shows the mean fractional contributions of surface elevation and the two velocity components to $\mathrm{RMSE}_{E,w}^2$.}
\label{fig:inlet_energy_metric_sensitivity}
\end{figure}

The slow-storm comparison has the same direction at both localization radii: QPCA-EnDCF has the lower three-realization mean, with $R_E=0.673$ at $2\,\text{km}$ and $R_E=0.863$ at $10\,\text{km}$. Its mean changes by less than $2\%$ between the two radii, whereas the 4D-LETKF mean decreases by approximately $23\%$. The QPCA-EnDCF advantage consequently narrows from a factor of $1.49$ to $1.16$ as the localization radius increases. This pattern is consistent with lower localization-radius sensitivity for QPCA-EnDCF in the slow-storm regime.

The fast-storm comparison does not show the same consistency. QPCA-EnDCF has the lower mean at $2\,\text{km}$, where $R_E=0.436$, but 4D-LETKF has the lower mean at $10\,\text{km}$, where $R_E=2.778$. At $2\,\text{km}$, QPCA-EnDCF is lower for each matched realization. At $10\,\text{km}$, however, QPCA-EnDCF is lower in two of the three realizations, while one much larger error reverses the ordering of the arithmetic means.

The horizontal-velocity terms contribute approximately $89$--$99\%$ of $\mathrm{RMSE}_{E,w}^2$ across the eight method--configuration combinations. The full-state diagnostic therefore provides a materially different comparison from one based only on surface elevation.

The slow-storm results are consistent with the hypothesis that constructing a leading residual direction from the global whitened mismatch is less sensitive to localization radius than estimating regional transforms from local observation patches. The results therefore support a limited QPCA-EnDCF advantage in the tested slow-storm configurations, but they do not establish universal superiority or a stable fast-storm ordering.

In summary, this inlet experiment establishes that QPCA-EnDCF can be implemented in a parallel nonlinear shallow-water model with approximately $2\times10^5$ state degrees of freedom without assembling the state covariance. The correction remains computationally feasible when approximately $1500$ stacked observations are represented by an empirical covariance of rank at most nine. QPCA-EnDCF and 4D-LETKF remain finite through the prescribed endpoints in all four configurations. QPCA-EnDCF has the lower mean energy-equivalent error in both slow-storm configurations and in the fast/$2\,\text{km}$ configuration, whereas 4D-LETKF has the lower mean in the fast/$10\,\text{km}$ configuration. The experiment extends the Lorenz--96 evidence to a substantially larger nonlinear model and establishes short-horizon feasibility.

\FloatBarrier

\section{Conclusions}\label{sec:conclusion}

This work introduced QPCA-EnDCF, a deterministic ensemble filter that regularizes the analysis correction in a rank-$\kappa$ subspace of the whitened forecast--observation mismatch. The theoretical analysis separates the bias introduced by spectral truncation from variability caused by finite-ensemble estimation. Concentration and spectral-perturbation results control the empirical residual covariance and retained projector, while the bias--variance decomposition identifies how the tradeoff depends on ensemble size, retained rank, and spectral separation. The comparison with stochastic EnKF methods also isolates the observation-perturbation variance avoided by the deterministic construction.

The Lorenz--96 experiments show that this tradeoff can improve both calibration and mean-squared error in a strongly undersampled setting. QPCA-EnDCF incurs greater squared bias than the stochastic and deterministic controls, but reduces across-realization variance by approximately $82\%$ to $84\%$; this composition persists through the $1000$-window cycled experiment. Independent-restart results further support the interpretation of rank as a regularization parameter, although the tested grid does not establish an asymptotic rate or rank-scaling law. 

The idealized inlet experiment tests the method in a nonlinear shallow-water model with approximately $2\times10^5$ state degrees of freedom and approximately $1500$ stacked observations represented by an empirical covariance of rank at most nine. QPCA-EnDCF and 4D-LETKF are the only tested filters that remain numerically viable through the prescribed endpoints in all four storm--localization configurations. QPCA-EnDCF has the lower three-realization mean energy-equivalent error in both slow-storm configurations and in the fast/$2\,\text{km}$ configuration, whereas 4D-LETKF has the lower mean in the fast/$10\,\text{km}$ configuration. QPCA-EnDCF is also closer to the leading-order spread--skill reference in both slow-storm configurations, although neither method is consistently calibrated. The endpoint errors remain substantial relative to the inlet depths and therefore establish comparative short-horizon feasibility rather than physically accurate storm-surge prediction.

These conclusions remain limited to controlled small-ensemble experiments with restricted choices of model error, observations, rank, and cycling duration. Future work should examine nonlinear and heterogeneous observations, adaptive selection of $\kappa$, and controlled ablations that separate spectral truncation from deterministic ensemble transformation and gain estimation. Longer high-dimensional cycling experiments should also determine whether hybrid or structure-preserving extensions can control the observed error growth while retaining the benefits of regularization in the whitened observation-mismatch space.

\section{Data and Code Availability}
Code for reproducing all examples in this work is available on GitHub in th e repository \href{https://github.com/rspence821505/Data-Consistent-Filtering}{https://github.com/rspence821505/Data-Consistent-Filtering}.

\section{Acknowledgments}

T.~Butler's work is supported by the National Science Foundation under Grant No.~DMS-2208460 and is also supported by the NSF IR/D program while working at the National Science Foundation.
The work of C.~Dawson and R.~Spence is supported in part by the National Science Foundation under Grant No.~DMS-2208461.

However, any opinions, findings, conclusions, or recommendations expressed in this material are those of the authors and do not necessarily reflect the views of the National Science Foundation.


\appendix


\section{Notation}\label{sec:Notation}

\begin{table}[H]
\centering
\footnotesize
\resizebox{\textwidth}{!}{%
\begin{tabular}{ccl}
\toprule
\textbf{Symbol} & \textbf{Dimension} & \textbf{Description} \\
\midrule
$\mathbf{X}$ & $\mathbb{R}^{n\times N}$ & Ensemble matrix, $\mathbf{X}=[\mathbf{x}^{1},\ldots,\mathbf{x}^{N}]$ \\
$\mathbf X_k,\mathbf X_k^\ast$ & $\mathbb R^{n\times N}$ & Sequential forecast and conditioned analysis ensembles at time $k$ \\
$\mathbf X_{w,0:L},\mathbf X_w^\ast$ & $(\mathbb R^{n\times N})^{L+1},\mathbb R^{n\times N}$ & Window forecast trajectory and conditioned endpoint analysis ensemble \\
$\boldsymbol\mu_{w,0},\boldsymbol\mu_{w,L}$ & $\mathbb R^n$ & Population means of the window-initial and endpoint forecast states \\
$\mathbf x_k^\dagger$ & $\mathbb R^n$ & True state at time $k$ \\
$\mathbf J_N$ & $\mathbb R^{N\times N}$ & Centering matrix, $\mathbf J_N=\mathbf I_N-N^{-1}\mathbf1\mathbf1^\top$ \\
$\mathcal M_N$ & $\mathbb R^{n\times N}\to\mathbb R^{n\times N}$ & Columnwise ensemble forecast operator induced by the dynamical model \\
$\mathbf{H}$ & $\mathbb{R}^{m\times n}$ & Observation operator at a single time \\
$\mathbf H_L$ & $\mathbb{R}^{d\times n}$ & Tangent-linear stacked observation operator over a window \\
$\mathbf{R}$ & $\mathbb{R}^{m\times m}$ & Observation-error covariance (single time), $\mathbf{R}=\sigma_{\mathrm{obs}}^2\mathbf{I}_m$ \\
$\mathbf{R}_L$ & $\mathbb{R}^{d\times d}$ & Stacked observation-error covariance, $\mathbf{R}_L=\mathbf{I}_L\otimes \mathbf{R}$ \\
$\mathbf{P},\mathbf P^\ast$ & $\mathbb{R}^{n\times n}$ & Population forecast and analysis covariances \\
$\widehat{\mathbf{P}},\widehat{\mathbf P}^\ast$ & $\mathbb{R}^{n\times n}$ & Sample forecast and analysis covariances \\
$\mathbf{z}_w$ & $\mathbb{R}^{d}$ & Stacked observations in window $w$ \\
$\bar{\mathbf{z}}_w$ & $\mathbb{R}^{d}$ & Mean stacked forecast observation in window $w$ \\
$\mathbf{E}_w$ & $\mathbb{R}^{d\times N}$ & Whitened residual matrix, $\mathbf{E}_w=[\mathbf{e}_w^1,\ldots,\mathbf{e}_w^N]$ \\
$\bar{\mathbf e}_w$ & $\mathbb{R}^{d}$ & Sample mean whitened residual in window $w$ \\
$\widehat{\boldsymbol{\Sigma}}_{E,w}$ & $\mathbb{R}^{d\times d}$ & Sample covariance of whitened residuals in window $w$ \\
$\mathbf E_{c,w}$ & $\mathbb R^{d\times N}$ & Centered window-residual matrix, $\mathbf E_{c,w}=\mathbf E_w\mathbf J_N$ \\
$\boldsymbol{\Sigma}_E$ & $\mathbb{R}^{d\times d}$ & Population covariance of whitened residuals \\
$\boldsymbol{\mu}_E$ & $\mathbb{R}^{d}$ & Population mean whitened residual \\
$\boldsymbol\xi^j,\boldsymbol\rho$ & $\mathbb R^d$ & Population-centered residual $\boldsymbol\xi^j=\mathbf e^j-\boldsymbol\mu_E$ and residual sample-mean fluctuation $\boldsymbol\rho=\bar{\mathbf e}-\boldsymbol\mu_E$ \\
$\boldsymbol{\Sigma}_{E,s}$ & $\mathbb{R}^{d\times d}$ & Population covariance under approximate map $Q_s$ \\
$\boldsymbol{\mu}_{E,s},\widehat{\boldsymbol{\Sigma}}_{E,s}$ & $\mathbb R^d,\mathbb R^{d\times d}$ & Population residual mean and sample residual covariance at approximation stage $s$ \\
$\boldsymbol\xi_s^j,\boldsymbol\rho_s$ & $\mathbb R^d$ & Stage-$s$ centered residual and residual sample-mean fluctuation \\
$\boldsymbol{\Sigma}_{E,\kappa}$ & $\mathbb{R}^{d\times d}$ & Rank-$\kappa$ truncation of $\boldsymbol{\Sigma}_E$ \\
$\widehat{\boldsymbol{\Sigma}}_{E,\kappa}$ & $\mathbb{R}^{d\times d}$ & Rank-$\kappa$ truncation of $\widehat{\boldsymbol{\Sigma}}_E$ in the fixed-window theory \\
$\boldsymbol{\Pi}_\kappa$ & $\mathbb{R}^{d\times d}$ & Population rank-$\kappa$ projector, $\boldsymbol{\Pi}_\kappa=\sum_{i=1}^\kappa \mathbf{v}_i\mathbf{v}_i^\top$ \\
$\widehat{\boldsymbol{\Pi}}_\kappa$ & $\mathbb{R}^{d\times d}$ & Sample rank-$\kappa$ projector, $\widehat{\boldsymbol{\Pi}}_\kappa=\sum_{i=1}^\kappa \hat{\mathbf{v}}_i\hat{\mathbf{v}}_i^\top$ \\
$\boldsymbol{\Pi}_{\kappa,s},\widehat{\boldsymbol{\Pi}}_{\kappa,s}$ & $\mathbb R^{d\times d}$ & Population and sample rank-$\kappa$ projectors at approximation stage $s$ \\
$\mathbf{K}_k$ & $\mathbb{R}^{n\times m}$ & Sequential Kalman gain at time $k$ \\
$\mathbf{K}_w$ & $\mathbb{R}^{n\times d}$ & Windowed Kalman gain for window $w$ ($d=mL$, from sample covariances) \\
$\mathbf K_{w,\mathrm{DC}}$ & $\mathbb{R}^{n\times d}$ & Empirical EnDCF gain in window $w$, defined via a pseudoinverse \\
$\widetilde{\mathbf P}_k,\widetilde{\mathbf P}_w$ & $\mathbb{R}^{N\times N}$ & Global sequential and windowed ETKF analysis covariances in ensemble space \\
$\bar{\mathbf w}_k,\bar{\mathbf w}_w$ & $\mathbb{R}^{N}$ & Global sequential and windowed ETKF analysis-mean weights \\
$\widetilde{\mathbf W}_k,\widetilde{\mathbf W}_w$ & $\mathbb{R}^{N\times N}$ & Global sequential and windowed ETKF symmetric transforms \\
$\widetilde{\mathbf{P}}_{k,q},\widetilde{\mathbf{P}}_{w,q}$ & $\mathbb{R}^{N\times N}$ & Local sequential and windowed LETKF analysis covariances in ensemble space \\
$\bar{\mathbf w}_{k,q},\bar{\mathbf w}_{w,q}$ & $\mathbb{R}^{N}$ & Local sequential and windowed LETKF analysis-mean weights \\
$\mathbf{W}_{k,q},\mathbf{W}_{w,q}$ & $\mathbb{R}^{N\times N}$ & Local sequential and windowed LETKF ensemble transforms \\
$\mathbf d_{k,q},\mathbf d_{w,q}$ & \text{varies} & Local sequential and windowed innovations \\
$\mathbf R_{k,q},\mathbf R_{L,q}$ & \text{varies} & Local sequential and stacked observation-error covariances \\
$\mathbf A_{x,k,q},\mathbf A_{z,k,q}$ & \text{varies} & Local sequential state and predicted-observation anomaly matrices \\
$\mathbf A_{x,w,q},\mathbf A_{z,w,q}$ & \text{varies} & Local window-$w$ endpoint-state and stacked predicted-observation anomaly matrices \\
$\mathbf{X}_{w,L}$ & $\mathbb{R}^{n\times N}$ & Forecast ensemble at the endpoint of window $w$ \\
$\mathbf{Z}_w$ & $\mathbb{R}^{d\times N}$ & Stacked forecast observations generated by the model; represented by $\mathbf H_L\mathbf{x}_{w,0}^j$ in the tangent-linear theory \\
$\mathbf{A}_{x,w}$ & $\mathbb{R}^{n\times N}$ & Window-$w$ endpoint state anomalies, $\mathbf{A}_{x,w}=\mathbf{X}_{w,L}\mathbf J_N$ \\
$\mathbf{A}_{z,w}$ & $\mathbb{R}^{d\times N}$ & Window-$w$ predicted-observation anomalies, $\mathbf{A}_{z,w}=\mathbf{Z}_w\mathbf J_N$ \\
$\mathbf{C}_{xz,w}$ & $\mathbb{R}^{n\times d}$ & Cross-covariance, $\mathbf{C}_{xz,w}=\frac{1}{N-1}\mathbf{A}_{x,w}\mathbf{A}_{z,w}^\top$ \\
$\mathbf{C}_{zz,w}$ & $\mathbb{R}^{d\times d}$ & Auto-covariance, $\mathbf{C}_{zz,w}=\frac{1}{N-1}\mathbf{A}_{z,w}\mathbf{A}_{z,w}^\top$ \\
$\widehat{\mathbf{V}}_{\kappa,w}$ & $\mathbb{R}^{d\times \kappa}$ & Matrix of the first $\kappa$ sample eigenvectors in window $w$, $\widehat{\mathbf{V}}_{\kappa,w}=[\hat{\mathbf{v}}_{1,w},\ldots,\hat{\mathbf{v}}_{\kappa,w}]$ \\
$\lambda_i,\hat{\lambda}_i,\hat{\lambda}_{i,w}$ & $\mathbb{R}$ & Population, fixed-window sample, and explicitly window-indexed sample residual-covariance eigenvalues \\
$\delta_\kappa$ & $\mathbb{R}_{+}$ & Positive lower bound for the relevant population cutoff gap; uniform over $s$ in the approximation-stage theory \\
$\mathbf Q_{\mathrm{PCA},w}$ & $\mathbb R^{\kappa\times N}$ & Retained PCA coordinates in window $w$, $\mathbf Q_{\mathrm{PCA},w}=\widehat{\mathbf V}_{\kappa,w}^\top\mathbf E_w$ \\
$\boldsymbol{\Delta}_{\mathrm{white},w}$ & $\mathbb{R}^{d\times N}$ & Projected correction in whitened observation coordinates \\
$\boldsymbol{\Delta}_{\mathrm{obs},w}$ & $\mathbb{R}^{d\times N}$ & Unwhitened observation-space correction \\
$\boldsymbol{\eta}_k$ & $\mathbb{R}^{m}$ & Realized measurement error in the observation model \\
$\boldsymbol{\epsilon}_k^j,\boldsymbol\epsilon_w^j$ & $\mathbb{R}^{m},\mathbb R^d$ & Artificial sequential and windowed stochastic-EnKF observation perturbations \\
$\boldsymbol{\mathcal E}_k,\boldsymbol{\mathcal E}_w$ & $\mathbb R^{m\times N},\mathbb R^{d\times N}$ & Matrices of memberwise sequential and windowed perturbations \\
$\widehat{\mathbf R}_L$ & $\mathbb R^{d\times d}$ & Sample covariance of windowed observation perturbations \\
$\mathbf E_s$ & $\mathbb R^{d\times N}$ & Whitened residual matrix for approximation stage $s$ \\
$\mathbf X_s^\ast,\bar{\mathbf x}_s^\ast$ & $\mathbb R^{n\times N},\mathbb R^n$ & Stage-$s$ analysis ensemble and its mean in the approximation theory \\
$\widehat{\mathcal U}_{\kappa,s}$ & $\mathbb R^{n\times d}$ & Random stage-$s$ QPCA update operator \\
$\Delta\boldsymbol{\Pi}_{\kappa,s}$ & $\mathbb R^{d\times d}$ & Stage-$s$ sample-projector error \\
$Q,Q_s$ & $\mathcal{X}\to\mathcal{D}$ & QoI map and its approximation at stage $s$ \\
\bottomrule
\end{tabular}%
}
\end{table}

\FloatBarrier

\section{Complete Filter Specifications}\label{app:algorithms}
\subsection{Common Notation and Cycling Conventions}\label{app:algorithm_conventions}

The specifications below implement the methods of Section~\ref{sec:algorithms_uq} using the indexing and ensemble notation of Section~\ref{sec:problem} and the symbol table in~\ref{sec:Notation}.  All methods share the initial ensemble $\mathbf X_0$, columnwise forecast map $\mathcal M_N$, observation operator $\mathbf H$, observation-error covariance $\mathbf R$, and observations $\{\mathbf z_k\}_{k=1}^{K}$.  Sequential methods set $\mathbf X_0^\ast:=\mathbf X_0$ and return $\mathbf X_K^\ast$.  Windowed methods use $K=WL$, $\mathbf R_L=\mathbf I_L\otimes\mathbf R$, $\mathbf X_{1,0}:=\mathbf X_0$, and $\mathbf X_{w,0}:=\mathbf X_{w-1}^\ast$ for $w>1$, and return $\mathbf X_W^\ast$.  At a window, ``forecast and stack'' means
\begin{equation}
\mathbf X_{w,\ell}=\mathcal M_N(\mathbf X_{w,\ell-1}),\quad
\mathbf Z_{w,\ell}=\mathbf H\mathbf X_{w,\ell}\quad(\ell=1,\ldots,L),
\qquad
\mathbf Z_w=\operatorname{col}_{\ell=1}^{L}(\mathbf Z_{w,\ell}),
\label{eq:appendix_forecast_stack}
\end{equation}
with $\mathbf z_w$ stacked as in~\eqref{eq:stacked_obs}.

When multiplicative inflation is enabled, the implementation scales forecast anomalies immediately before analysis: $\mathbf X\leftarrow\bar{\mathbf x}\mathbf1^\top+\lambda_{\mathrm{infl}}\mathbf X\mathbf J_N$.  For a four-dimensional update this operation is applied to the endpoint ensemble $\mathbf X_{w,L}$; consequently it is not a uniform inflation of the entire forecast trajectory used to form $\mathbf Z_w$.  The uninflated choice $\lambda_{\mathrm{infl}}=1$ avoids that distinction.  Remaining implementation-specific choices for localization, matrix solves, and square roots are described in Subsection~\ref{app:algorithm_implementation} rather than repeated in every algorithm.

\begin{table}[H]
\centering
\small
\caption{Correspondence between the main-text formulations and~\ref{app:algorithms}.}
\label{tab:algorithm_crosswalk}
\begin{tabular}{lll}
\toprule
Method & Main-text formulation & Implementation specification \\
\midrule
Seq-EnKF & Sec.~\ref{subsec:seq_enkf_method}, Eqs.~\eqref{eq:enkf_update}--\eqref{eq:kalman_gain} & Alg.~\ref{alg:seq_enkf_appendix} \\
4D-EnKF & Sec.~\ref{subsec:4d_enkf_method}, Eq.~\eqref{eq:4d_perturb} & Alg.~\ref{alg:4d_enkf_appendix} \\
Seq-ETKF & Eqs.~\eqref{eq:seq_etkf_weights_main}--\eqref{eq:seq_etkf_analysis_main} & Alg.~\ref{alg:seq_etkf_appendix} \\
4D-ETKF & Stacked form of Eqs.~\eqref{eq:seq_etkf_weights_main}--\eqref{eq:seq_etkf_analysis_main} & Alg.~\ref{alg:4d_etkf_appendix} \\
Seq-LETKF & Sec.~\ref{subsubsec:seq_letkf} & Alg.~\ref{alg:seq_letkf_appendix} \\
4D-LETKF & Sec.~\ref{subsubsec:4d_letkf} & Alg.~\ref{alg:4d_letkf_appendix} \\
QPCA-EnDCF & Eqs.~\eqref{eq:whitening_problem}--\eqref{eq:qpca_delta_obs} & Alg.~\ref{alg:qpca_endcf} \\
\bottomrule
\end{tabular}
\end{table}

\subsection{Stochastic EnKF Specifications}

\subsubsection{Sequential EnKF}\label{app:seq_enkf}

Algorithm~\ref{alg:seq_enkf_appendix} is the cycling form of Eqs.~\eqref{eq:enkf_update}--\eqref{eq:kalman_gain}.
\begin{algorithm}[H]
\caption{Seq-EnKF}
\label{alg:seq_enkf_appendix}
\small
\begin{algorithmic}[1]
\Require Common sequential inputs from~\ref{app:algorithm_conventions}; $\lambda_{\mathrm{infl}}\ge1$
\For{$k=1$ to $K$}
    \State $\mathbf X_k\gets\mathcal M_N(\mathbf X_{k-1}^\ast)$; scale its anomalies by $\lambda_{\mathrm{infl}}$
    \State $\mathbf Z_k\gets\mathbf H\mathbf X_k$, $\mathbf A_{x,k}\gets\mathbf X_k\mathbf J_N$, $\widehat{\mathbf P}_k\gets(N-1)^{-1}\mathbf A_{x,k}\mathbf A_{x,k}^\top$
    \State $\mathbf K_k\gets\widehat{\mathbf P}_k\mathbf H^\top(\mathbf H\widehat{\mathbf P}_k\mathbf H^\top+\mathbf R)^{-1}$
    \State Draw $\boldsymbol{\mathcal E}_k=[\boldsymbol\epsilon_k^1,\ldots,\boldsymbol\epsilon_k^N]$, $\boldsymbol\epsilon_k^j\stackrel{\mathrm{i.i.d.}}{\sim}\mathcal N(\mathbf0,\mathbf R)$
    \State $\mathbf X_k^\ast\gets\mathbf X_k+\mathbf K_k(\mathbf z_k\mathbf1^\top+\boldsymbol{\mathcal E}_k-\mathbf Z_k)$
\EndFor
\State \Return $\mathbf X_K^\ast$
\end{algorithmic}
\end{algorithm}

\subsubsection{Four-Dimensional EnKF}\label{app:4d_enkf}

Algorithm~\ref{alg:4d_enkf_appendix} uses the same memberwise stochastic analysis in the stacked observation space and updates only the window endpoint.
\begin{algorithm}[H]
\caption{4D-EnKF}
\label{alg:4d_enkf_appendix}
\small
\begin{algorithmic}[1]
\Require Common windowed inputs from~\ref{app:algorithm_conventions}; $\lambda_{\mathrm{infl}}\ge1$
\For{$w=1$ to $W$}
    \State Forecast and stack by~\eqref{eq:appendix_forecast_stack}; scale the anomalies of $\mathbf X_{w,L}$ by $\lambda_{\mathrm{infl}}$
    \State $\mathbf A_{x,w}\gets\mathbf X_{w,L}\mathbf J_N$, $\mathbf A_{z,w}\gets\mathbf Z_w\mathbf J_N$
    \State $\mathbf C_{xz,w}\gets(N-1)^{-1}\mathbf A_{x,w}\mathbf A_{z,w}^\top$, $\mathbf C_{zz,w}\gets(N-1)^{-1}\mathbf A_{z,w}\mathbf A_{z,w}^\top$
    \State $\mathbf K_w\gets\mathbf C_{xz,w}(\mathbf C_{zz,w}+\mathbf R_L)^{-1}$
    \State Draw $\boldsymbol{\mathcal E}_w=[\boldsymbol\epsilon_w^1,\ldots,\boldsymbol\epsilon_w^N]$, $\boldsymbol\epsilon_w^j\stackrel{\mathrm{i.i.d.}}{\sim}\mathcal N(\mathbf0,\mathbf R_L)$
    \State $\mathbf X_w^\ast\gets\mathbf X_{w,L}+\mathbf K_w(\mathbf z_w\mathbf1^\top+\boldsymbol{\mathcal E}_w-\mathbf Z_w)$
\EndFor
\State \Return $\mathbf X_W^\ast$
\end{algorithmic}
\end{algorithm}

The distinction between Algorithms~\ref{alg:seq_enkf_appendix} and~\ref{alg:4d_enkf_appendix} is therefore cadence and observation-space dimension: Seq-EnKF draws an $m$-dimensional perturbation at every $k$, whereas 4D-EnKF draws a $d=mL$-dimensional perturbation once per window.

\subsection{Global ETKF Specifications}
\label{app:etkf_global}

The ETKF variants use the same ensemble-space kernel as the LETKF variants below, but with localization disabled. Equivalently, the localization weights are identically one, so every state component uses the same global transform formed from all observations supplied at that update. This is the configuration used for the Lorenz--96 deterministic controls.

\subsubsection{Sequential ETKF}

\begin{algorithm}[H]
\caption{Seq-ETKF}
\label{alg:seq_etkf_appendix}
\small
\begin{algorithmic}[1]
\Require Common sequential inputs; $\lambda_{\mathrm{infl}}\ge1$
\For{$k=1$ to $K$}
    \State $\mathbf X_k\gets\mathcal M_N(\mathbf X_{k-1}^\ast)$; scale its anomalies by $\lambda_{\mathrm{infl}}$
    \State Form $\bar{\mathbf x}_k$, $\mathbf A_{x,k}$, $\mathbf Z_k=\mathbf H\mathbf X_k$, $\bar{\mathbf z}_k$, and $\mathbf A_{z,k}$
    \State Compute $\widetilde{\mathbf P}_k$, $\bar{\mathbf w}_k$, and $\widetilde{\mathbf W}_k$ by~\eqref{eq:seq_etkf_weights_main}
    \State Compute $\mathbf X_k^\ast$ by~\eqref{eq:seq_etkf_analysis_main}
\EndFor
\State \Return $\mathbf X_K^\ast$
\end{algorithmic}
\end{algorithm}

\subsubsection{Four-Dimensional ETKF}

\begin{algorithm}[H]
\caption{4D-ETKF}
\label{alg:4d_etkf_appendix}
\small
\begin{algorithmic}[1]
\Require Common windowed inputs; $\lambda_{\mathrm{infl}}\ge1$
\For{$w=1$ to $W$}
    \State Forecast and stack by~\eqref{eq:appendix_forecast_stack}; scale the anomalies of $\mathbf X_{w,L}$ by $\lambda_{\mathrm{infl}}$
    \State Form $\bar{\mathbf x}_{w,L}$, $\mathbf A_{x,w}$, $\bar{\mathbf z}_w$, and $\mathbf A_{z,w}$
    \State Compute $\widetilde{\mathbf P}_w$, $\bar{\mathbf w}_w$, and $\widetilde{\mathbf W}_w$ using the stacked form of~\eqref{eq:seq_etkf_weights_main}
    \State Compute $\mathbf X_w^\ast$ using the stacked form of~\eqref{eq:seq_etkf_analysis_main}
\EndFor
\State \Return $\mathbf X_W^\ast$
\end{algorithmic}
\end{algorithm}

\subsection{Localized ETKF Specifications}

\subsubsection{Local LETKF Analysis Transform}\label{app:letkf_kernel}

Both LETKF variants use the local form of the ETKF ensemble-space kernel; only the forecast cadence and the observations supplied to it differ. For the sequential method, the local quantities satisfy
\begin{align}
\widetilde{\mathbf P}_{k,q}&=\bigl[(N-1)\mathbf I_N+\mathbf A_{z,k,q}^\top\mathbf R_{k,q}^{-1}\mathbf A_{z,k,q}\bigr]^{-1}, &
\bar{\mathbf w}_{k,q}&=\widetilde{\mathbf P}_{k,q}\mathbf A_{z,k,q}^\top\mathbf R_{k,q}^{-1}\mathbf d_{k,q},\nonumber\\[-2pt]
\widetilde{\mathbf W}_{k,q}&=\bigl[(N-1)\widetilde{\mathbf P}_{k,q}\bigr]^{1/2}, &
\mathbf X_{k,q}^\ast&=\bar{\mathbf x}_{k,q}\mathbf1^\top+\mathbf A_{x,k,q}\bigl(\bar{\mathbf w}_{k,q}\mathbf1^\top+\widetilde{\mathbf W}_{k,q}\bigr).
\label{eq:appendix_seq_letkf_kernel}
\end{align}

For the four-dimensional method, replace $k$ by $w$, $\bar{\mathbf x}_{k,q}$ by the endpoint mean $\bar{\mathbf x}_{w,L,q}$, and $\mathbf R_{k,q}$ by $\mathbf R_{L,q}$:
\begin{align}
\widetilde{\mathbf P}_{w,q}&=\bigl[(N-1)\mathbf I_N+\mathbf A_{z,w,q}^\top\mathbf R_{L,q}^{-1}\mathbf A_{z,w,q}\bigr]^{-1}, &
\bar{\mathbf w}_{w,q}&=\widetilde{\mathbf P}_{w,q}\mathbf A_{z,w,q}^\top\mathbf R_{L,q}^{-1}\mathbf d_{w,q},\nonumber\\[-2pt]
\widetilde{\mathbf W}_{w,q}&=\bigl[(N-1)\widetilde{\mathbf P}_{w,q}\bigr]^{1/2}, &
\mathbf X_{w,q}^\ast&=\bar{\mathbf x}_{w,L,q}\mathbf1^\top+\mathbf A_{x,w,q}\bigl(\bar{\mathbf w}_{w,q}\mathbf1^\top+\widetilde{\mathbf W}_{w,q}\bigr).
\label{eq:appendix_4d_letkf_kernel}
\end{align}

Here $\mathbf d_{k,q}=\mathbf z_{k,q}-\bar{\mathbf z}_{k,q}$ and $\mathbf d_{w,q}=\mathbf z_{w,q}-\bar{\mathbf z}_{w,q}$. The localization rule tapers the observation precision for region $q$ and the local analyses are then assembled in state-region order. With uniform localization weights and one full-state region, these kernels reduce exactly to Algorithms~\ref{alg:seq_etkf_appendix} and~\ref{alg:4d_etkf_appendix}. The square root in both kernels is the symmetric positive-semidefinite square root.

\subsubsection{Sequential LETKF}\label{app:seq_letkf}

\begin{algorithm}[H]
\caption{Seq-LETKF}
\label{alg:seq_letkf_appendix}
\small
\begin{algorithmic}[1]
\Require Common sequential inputs; localization rule; $\lambda_{\mathrm{infl}}\ge1$
\For{$k=1$ to $K$}
    \State $\mathbf X_k\gets\mathcal M_N(\mathbf X_{k-1}^\ast)$; scale its anomalies by $\lambda_{\mathrm{infl}}$
    \State Form $\bar{\mathbf x}_k$, $\mathbf A_{x,k}$, $\mathbf Z_k=\mathbf H\mathbf X_k$, $\bar{\mathbf z}_k$, and $\mathbf A_{z,k}$
    \For{each local analysis region $q$}
        \State Extract the local quantities and compute $\mathbf X_{k,q}^\ast$ by~\eqref{eq:appendix_seq_letkf_kernel}
    \EndFor
    \State Assemble $\mathbf X_k^\ast$ from $\{\mathbf X_{k,q}^\ast\}_q$
\EndFor
\State \Return $\mathbf X_K^\ast$
\end{algorithmic}
\end{algorithm}

\subsubsection{Four-Dimensional LETKF}\label{app:4d_letkf}

\begin{algorithm}[H]
\caption{4D-LETKF}
\label{alg:4d_letkf_appendix}
\small
\begin{algorithmic}[1]
\Require Common windowed inputs; localization rule; $\lambda_{\mathrm{infl}}\ge1$
\For{$w=1$ to $W$}
    \State Forecast and stack by~\eqref{eq:appendix_forecast_stack}; scale the anomalies of $\mathbf X_{w,L}$ by $\lambda_{\mathrm{infl}}$
    \State Form $\bar{\mathbf x}_{w,L}$, $\mathbf A_{x,w}$, $\bar{\mathbf z}_w$, and $\mathbf A_{z,w}$
    \For{each local analysis region $q$}
        \State Extract the local quantities and compute $\mathbf X_{w,q}^\ast$ by~\eqref{eq:appendix_4d_letkf_kernel}
    \EndFor
    \State Assemble $\mathbf X_w^\ast$ from $\{\mathbf X_{w,q}^\ast\}_q$
\EndFor
\State \Return $\mathbf X_W^\ast$
\end{algorithmic}
\end{algorithm}

\subsection{QPCA-EnDCF Specification}\label{app:qpca_endcf}

Algorithm~\ref{alg:qpca_endcf} is the windowed deterministic update defined by Eqs.~\eqref{eq:whitening_problem}--\eqref{eq:qpca_delta_obs}.  The projector is estimated from centered residuals, whereas the correction acts on the uncentered residual matrix; this distinction retains the component of the mean mismatch that lies in the empirical leading subspace.
\begin{algorithm}[H]
\caption{QPCA-EnDCF}
\label{alg:qpca_endcf}
\small
\begin{algorithmic}[1]
\Require Common windowed inputs; truncation rank $1\le\kappa\le\min(d,N-1)$; $\lambda_{\mathrm{infl}}\ge1$
\For{$w=1$ to $W$}
    \State Forecast and stack by~\eqref{eq:appendix_forecast_stack}; scale the anomalies of $\mathbf X_{w,L}$ by $\lambda_{\mathrm{infl}}$
    \State $\mathbf E_w\gets\mathbf R_L^{-1/2}(\mathbf Z_w-\mathbf z_w\mathbf1^\top)$, $\mathbf E_{c,w}\gets\mathbf E_w\mathbf J_N$
    \State $\widehat{\boldsymbol\Sigma}_{E,w}\gets(N-1)^{-1}\mathbf E_{c,w}\mathbf E_{c,w}^\top$
    \State Compute the leading $\kappa$ eigenvectors $\widehat{\mathbf V}_{\kappa,w}$ and $\widehat{\boldsymbol{\Pi}}_{\kappa,w}\gets\widehat{\mathbf V}_{\kappa,w}\widehat{\mathbf V}_{\kappa,w}^\top$
    \State $\boldsymbol\Delta_{\mathrm{obs},w}\gets-\mathbf R_L^{1/2}\widehat{\boldsymbol{\Pi}}_{\kappa,w}\mathbf E_w$
    \State Form $\mathbf C_{xz,w}$ and $\mathbf C_{zz,w}$ from $\mathbf X_{w,L}$ and $\mathbf Z_w$ as in Definition~\ref{def:endcf_gain_ensemble}
    \State $\mathbf K_{w,\mathrm{DC}}\gets\mathbf C_{xz,w}\mathbf C_{zz,w}^{\dagger}$
    \State $\mathbf X_w^\ast\gets\mathbf X_{w,L}+\mathbf K_{w,\mathrm{DC}}\boldsymbol\Delta_{\mathrm{obs},w}$
\EndFor
\State \Return $\mathbf X_W^\ast$
\end{algorithmic}
\end{algorithm}

The negative sign drives the retained forecast-minus-observation residual toward zero.  The correction is deterministic conditional on the forecast ensemble and observations, but its columns are generally member dependent.

\subsection{Numerical Linear Algebra and Stabilization}\label{app:algorithm_implementation}

All filters use the forecast-anomaly inflation convention specified in Subsection~\ref{app:algorithm_conventions}.  The stochastic EnKF variants additionally draw memberwise observation perturbations, whereas the ETKF, LETKF, and QPCA-EnDCF analyses are deterministic conditional on the forecast ensemble and observations.  ETKF is global, LETKF tapers observation precision locally, and localization of the gain-based methods is implemented by tapering the state--observation cross-covariance.  The transform filters use the symmetric positive-semidefinite square root of $(N-1)\widetilde{\mathbf P}$, while QPCA-EnDCF uses factors of $\mathbf R_L$ for whitening and unwhitening and the leading $\kappa$ eigenvectors of $\widehat{\boldsymbol\Sigma}_{E,w}$ to construct its projector.

For $\mathbf R=\sigma_{\mathrm{obs}}^2\mathbf I_m$, whitening and unwhitening reduce to multiplication by $\sigma_{\mathrm{obs}}^{-1}$ and $\sigma_{\mathrm{obs}}$; for correlated errors the implementation uses Cholesky or symmetric spectral factors.

The reference EnDCF gain in Definition~\ref{def:endcf_gain_ensemble} uses $\mathbf C_{zz,w}^{\dagger}$.  In the unlocalized numerical implementation used for the reported QPCA-EnDCF results, the spectrum of $\mathbf C_{zz,w}$ is instead raised to at least $\max\{10^{-6},10^{-2}\lambda_{\max}(\mathbf R_L)\}$ before a symmetric positive-definite solve (with a further $10^{-6}$ numerical jitter in the inverse routine).

The optional localized implementation uses a thin singular-value decomposition and retains covariance eigenvalues above $10^{-2}\lambda_{\max}(\mathbf R_L)$, which realizes the corresponding range-restricted pseudoinverse.  These substitutions change only the numerical realization of the gain; whitening, projector construction, and the deterministic memberwise correction remain those in Algorithm~\ref{alg:qpca_endcf}.

\FloatBarrier

\FloatBarrier

\section{Concentration and Spectral Stability of the Empirical Projector}
\label{sec:concentration_spectral}

The analysis of this section is organized around a decomposition that makes explicit how spectral truncation and finite-ensemble sampling jointly determine performance.
The development proceeds in three stages.
First, we establish baseline identities and moment-based concentration estimates for the sample covariance of the whitened forecast--observation residuals, quantifying how empirical second-order statistics fluctuate around their population counterparts.
Second, we invoke deterministic matrix perturbation theory to translate these covariance deviations into stability results for eigenvalues and eigenspaces, thereby controlling the accuracy of empirical spectral projectors.
Third, we derive a bias--variance decomposition for the mean-squared error of the analysis mean, which cleanly separates truncation effects, governed by the spectral truncation rank $\kappa$ and the alignment of the mean residual with the retained subspace, from sampling effects, which depend on the ensemble size $N$ and on the stability of the empirical projector.

The EnDCF gain used in practice is an empirical quantity constructed from finite ensembles. In the regime where $N-1<d$, the observation-space sample covariance is singular, requiring a stabilized inverse (Definition~\ref{def:endcf_gain_ensemble}, Subsection~\ref{subsec:ensemble_uq}). Whenever expectations involve this random gain, we invoke the moment condition in Assumption~\ref{ass:gain_moment}.



\subsection{Probabilistic and Spectral Assumptions}
\label{subsec:regularity_assumptions}

We collect here the regularity conditions required for the convergence analysis of QPCA-EnDCF. These assumptions are deliberately mild and are designed to isolate the essential probabilistic and spectral inputs needed for the bias--variance decomposition and the associated concentration estimates.

\begin{assumption}[Regularity conditions]
\label{ass:regularity_bias_variance}
Let $d:=mL$. We assume the following:
\begin{enumerate}[label=(\roman*)]
\item The forecast distribution $P_{\mathcal{X}}$ has a finite second moment, that is,
\begin{equation}
\mathbb{E}\|\mathbf{x}\|^{2}<\infty .
\end{equation}

\item The stacked observation-error covariance $\mathbf{R}_L\in\mathbb{R}^{d\times d}$ is symmetric positive definite and satisfies
\begin{equation}
\lambda_{\min}\!\bigl(\mathbf{R}_L\bigr)\ge r_{\min}>0.
\end{equation}

\item The window-initial forecast ensemble $\{\mathbf{x}_{w,0}^j\}_{j=1}^N$ consists of independent and identically distributed samples from $P_{\mathcal{X}}$. The corresponding whitened residuals $\{\mathbf{e}^j\}_{j=1}^N$, defined in Definition~\ref{def:whitened_residual_covariances}, satisfy
\begin{equation}
\mathbb{E}\|\mathbf{e}^1\|^{4}<\infty,
\end{equation}
or, equivalently, $\mathbb{E}\|\mathbf{e}^1-\boldsymbol{\mu}_E\|^{4}<\infty$.

\item The population covariance $\boldsymbol{\Sigma}_E$ introduced in Definition~\ref{def:whitened_residual_covariances} admits the spectral decomposition
\begin{equation}
\boldsymbol{\Sigma}_E=\sum_{i=1}^{d}\lambda_i\,\mathbf{v}_i\mathbf{v}_i^\top,
\qquad
\lambda_1\ge\lambda_2\ge\cdots\ge\lambda_d\ge 0.
\end{equation}
\end{enumerate}
\end{assumption}

\noindent
Assumption~\ref{ass:regularity_bias_variance}(i) ensures that all first and second-moment quantities appearing in the filter are well defined. Assumption~\ref{ass:regularity_bias_variance}(ii) guarantees that the whitening transformation is well posed. Assumption~\ref{ass:regularity_bias_variance}(iii) provides the probabilistic input for the covariance concentration estimates developed below; it is strictly weaker than Gaussianity and suffices for mean-square convergence results, while sharper fluctuation bounds may require stronger distributional assumptions. Finally, Assumption~\ref{ass:regularity_bias_variance}(iv) fixes notation for the population spectrum; no spectral gap or decay condition is imposed unless explicitly stated in later results.

The i.i.d.\ sampling assumption in (iii) is a mathematical idealization adopted for deriving the convergence bounds and concentration estimates in~\ref{subsec:fundamental_concentration}--\ref{subsec:main_bv}.
In practical cycling experiments, 
the analysis ensemble from window $w$ becomes the forecast ensemble for window $w+1$, inducing temporal dependence that violates strict i.i.d.\ conditions.

Nonetheless, the theoretical results provide qualitative guidance for the cycling experiments. Under suitable mixing conditions, dependence between windows may weaken as their temporal separation increases, although consecutive cycled ensembles are not independent. The theoretical bounds therefore motivate the dependence of sampling variability on $N$, $\kappa$, and the cutoff gap $\delta_\kappa$, but they do not directly establish the corresponding rates under cycling. The experiments assess whether these qualitative relationships persist when the formal independence assumptions are violated.

\begin{assumption}[Moment control for the empirical gain]
\label{ass:gain_moment}
Whenever expectations involve the random empirical gain, we assume that
\begin{equation}
\mathbb{E}\!\left[\|\mathbf K_{\mathrm{DC}}\|_2^{2}\right]<\infty.
\end{equation}
\end{assumption}

This moment condition is mild in practice: when a stabilized inverse such as $(\mathbf{C}_{zz}+\varepsilon\mathbf{I})^{-1}$ is used, the inverse is uniformly bounded, and the condition follows from standard moment bounds on the ensemble anomalies (e.g., if the forecast distribution has exponential or sub-exponential tails, as in the Gaussian case). The pseudoinverse satisfies the condition under similar regularity conditions on the forecast distribution. This is the only property of the gain required to pass from conditional to unconditional mean-square bounds in the analysis.


\subsection{Mean and Covariance Concentration}
\label{subsec:fundamental_concentration}

Here, we first provide the basic probabilistic identities and concentration estimates that form the backbone of the convergence analysis.
The results in this subsection are elementary but essential as they separate deterministic bias from sampling variance and quantify how empirical covariance operators concentrate around their population counterparts.
Together, these estimates provide the input required for subsequent spectral perturbation and bias--variance analyses.


\begin{lemma}[Bias--variance decomposition for the analysis mean]
\label{lem:mean_error_decomp}
Let
\(
\bar{\mathbf{x}}^\ast=\frac{1}{N}\sum_{j=1}^N\mathbf{x}^{\ast,j}
\)
denote the analysis ensemble mean at a fixed assimilation time, and let $\mathbf{x}^\dagger$ denote the corresponding true state; we suppress the common time index. Then
\begin{equation}
\bar{\mathbf{x}}^\ast-\mathbf{x}^\dagger
=
\bigl(\mathbb{E}[\bar{\mathbf{x}}^\ast]-\mathbf{x}^\dagger\bigr)
+
\bigl(\bar{\mathbf{x}}^\ast-\mathbb{E}[\bar{\mathbf{x}}^\ast]\bigr),
\end{equation}
and consequently
\begin{equation}
\mathbb{E}\!\left[\|\bar{\mathbf{x}}^\ast-\mathbf{x}^\dagger\|^2\right]
=
\|\mathbb{E}[\bar{\mathbf{x}}^\ast]-\mathbf{x}^\dagger\|^2
+
\mathbb{E}\!\left[\|\bar{\mathbf{x}}^\ast-\mathbb{E}[\bar{\mathbf{x}}^\ast]\|^2\right].
\end{equation}
\end{lemma}

\begin{proof}
Define the deterministic bias vector
\(
\mathbf{b}:=\mathbb{E}[\bar{\mathbf{x}}^\ast]-\mathbf{x}^\dagger
\)
and expand about the zero-mean fluctuation $\bar{\mathbf{x}}^\ast-\mathbb{E}[\bar{\mathbf{x}}^\ast]$:
\begin{equation}
	\mathbb{E}\|\bar{\mathbf{x}}^\ast-\mathbf{x}^\dagger\|^2
	=
	\|\mathbf{b}\|^2
	+
	2\left\langle \mathbf{b},\mathbb{E}\!\left[\bar{\mathbf{x}}^\ast-\mathbb{E}[\bar{\mathbf{x}}^\ast]\right]\right\rangle
	+
	\mathbb{E}\|\bar{\mathbf{x}}^\ast-\mathbb{E}[\bar{\mathbf{x}}^\ast]\|^2
	=
	\|\mathbf{b}\|^2+\mathbb{E}\|\bar{\mathbf{x}}^\ast-\mathbb{E}[\bar{\mathbf{x}}^\ast]\|^2,
\end{equation}
	which yields the stated identity.
	\end{proof}

\begin{lemma}[Unbiasedness of the sample covariance]
\label{lem:unbiasedness_sample_cov}
Let $\boldsymbol{\Sigma}_E$ and $\widehat{\boldsymbol{\Sigma}}_E$ be as in Definition~\ref{def:whitened_residual_covariances}. Under Assumption~\ref{ass:regularity_bias_variance}(iii) and for all $N\ge 2$,
\begin{equation}
\mathbb{E}[\widehat{\boldsymbol{\Sigma}}_E]=\boldsymbol{\Sigma}_E.
\end{equation}
\end{lemma}

\begin{proof}
Unbiasedness of the sample covariance with Bessel correction holds for independent samples with finite second moments. Let
\(
\boldsymbol\xi^j:=\mathbf{e}^j-\boldsymbol{\mu}_E
\)
and
\(
\boldsymbol\rho:=\frac{1}{N}\sum_{j=1}^N\boldsymbol\xi^j
=\bar{\mathbf e}-\boldsymbol\mu_E.
\)
Then $\mathbf{e}^j-\bar{\mathbf{e}}=\boldsymbol\xi^j-\boldsymbol\rho$ and
\begin{equation}
\widehat{\boldsymbol{\Sigma}}_E
=
\frac{1}{N-1}\sum_{j=1}^N \boldsymbol\xi^j(\boldsymbol\xi^j)^\top
-
\frac{N}{N-1}\boldsymbol\rho\boldsymbol\rho^\top.
\end{equation}
Taking expectations and using
\(
\mathbb{E}[\boldsymbol\xi^j(\boldsymbol\xi^j)^\top]=\boldsymbol{\Sigma}_E
\)
together with
\(
\mathbb{E}[\boldsymbol\rho\boldsymbol\rho^\top]
=\boldsymbol{\Sigma}_E/N,
\)
we obtain $\mathbb{E}[\widehat{\boldsymbol{\Sigma}}_E]=\boldsymbol{\Sigma}_E$.
\end{proof}

\begin{lemma}[Mean-square Frobenius deviation of the sample covariance]
\label{lem:frobenius_concentration}
Let $\boldsymbol{\Sigma}_E$ and $\widehat{\boldsymbol{\Sigma}}_E$ be as in Definition~\ref{def:whitened_residual_covariances}. Under Assumption~\ref{ass:regularity_bias_variance}(iii), there exists a constant $C_{\mathrm{cov}}>0$, depending only on the dimension $d$ and on standardized fourth moments of $\mathbf{e}^1$, such that for all $N\ge 2$,
\begin{equation}
\mathbb{E}\!\left[\|\widehat{\boldsymbol{\Sigma}}_E-\boldsymbol{\Sigma}_E\|_F^2\right]
\le
\frac{C_{\mathrm{cov}}}{N-1}\,
\mathrm{tr}\!\left(\boldsymbol{\Sigma}_E^2\right).
\end{equation}
In particular, $\mathbb{E}\|\widehat{\boldsymbol{\Sigma}}_E-\boldsymbol{\Sigma}_E\|_F^2=\mathcal{O}(N^{-1})$ as $N\to\infty$.
\end{lemma}

\begin{proof}
Let
\(
\boldsymbol\xi^j:=\mathbf{e}^j-\boldsymbol{\mu}_E
\)
and
\(
\boldsymbol\rho:=\frac{1}{N}\sum_{j=1}^N\boldsymbol\xi^j
=\bar{\mathbf e}-\boldsymbol\mu_E.
\)
A direct expansion yields
\begin{equation}
\widehat{\boldsymbol{\Sigma}}_E-\boldsymbol{\Sigma}_E
=
\frac{1}{N-1}\sum_{j=1}^N
\bigl(\boldsymbol\xi^j(\boldsymbol\xi^j)^\top-\boldsymbol{\Sigma}_E\bigr)
-
\frac{N}{N-1}
\Bigl(\boldsymbol\rho\boldsymbol\rho^\top-\frac{1}{N}\boldsymbol{\Sigma}_E\Bigr).
\end{equation}

Applying the inequality $(a+b)^2\le 2a^2+2b^2$, using independence of the samples, and invoking the fourth-moment bound
\(
\mathbb{E}\|\boldsymbol\rho\|^4\lesssim N^{-2}\mathbb{E}\|\boldsymbol\xi^1\|^4
\)
(which follows from Assumption~\ref{ass:regularity_bias_variance}(iii)), we obtain
\begin{equation}
\mathbb{E}\!\left[\|\widehat{\boldsymbol{\Sigma}}_E-\boldsymbol{\Sigma}_E\|_F^2\right]
\le
\frac{C_{\mathrm{cov}}}{N-1}\,
\mathbb{E}\!\left[
\bigl\|
\boldsymbol\xi^1(\boldsymbol\xi^1)^\top-\boldsymbol{\Sigma}_E
\bigr\|_F^2
\right],
\end{equation}

	for a constant $C_{\mathrm{cov}}>0$ depending only on $d$ and standardized fourth moments.

Finally, since
	\(
	\|\boldsymbol\xi^1(\boldsymbol\xi^1)^\top-\boldsymbol{\Sigma}_E\|_F^2
	\le
	2\|\boldsymbol\xi^1(\boldsymbol\xi^1)^\top\|_F^2+2\|\boldsymbol{\Sigma}_E\|_F^2
	=
		2\|\boldsymbol\xi^1\|^4+2\,\mathrm{tr}(\boldsymbol{\Sigma}_E^2),
		\)
		and the standardized fourth-moment control implicit in the constant $C_{\mathrm{cov}}$ ensures that $\mathbb{E}\|\boldsymbol\xi^1\|^4$ is bounded by a dimension-dependent multiple of $\mathrm{tr}(\boldsymbol{\Sigma}_E^2)$. In particular, if $\boldsymbol\xi^1$ is Gaussian then
\begin{equation}
		\mathbb{E}\|\boldsymbol\xi^1\|^4
		=
		\bigl(\mathrm{tr}(\boldsymbol{\Sigma}_E)\bigr)^2 + 2\,\mathrm{tr}(\boldsymbol{\Sigma}_E^2)
		\le
		(d+2)\,\mathrm{tr}(\boldsymbol{\Sigma}_E^2),
\end{equation}
		where the final inequality uses $\bigl(\mathrm{tr}(\boldsymbol{\Sigma}_E)\bigr)^2\le d\,\mathrm{tr}(\boldsymbol{\Sigma}_E^2)$. Absorbing this ratio into $C_{\mathrm{cov}}$ yields the stated bound.
\end{proof}

\begin{remark}[Gaussian sharpening of Lemma~\ref{lem:frobenius_concentration}]
\label{rem:gaussian_frobenius}
If, in addition, $\boldsymbol\xi^1$ is Gaussian, then $(N-1)\widehat{\boldsymbol{\Sigma}}_E$ follows a Wishart distribution and the bound of Lemma~\ref{lem:frobenius_concentration} can be replaced by the exact identity
\begin{equation}
\mathbb{E}\|\widehat{\boldsymbol{\Sigma}}_E-\boldsymbol{\Sigma}_E\|_F^2
=
\frac{\bigl(\mathrm{tr}(\boldsymbol{\Sigma}_E)\bigr)^2+\mathrm{tr}(\boldsymbol{\Sigma}_E^2)}{N-1}.
\end{equation}
Applying $\bigl(\mathrm{tr}(\boldsymbol{\Sigma}_E)\bigr)^2\le d\,\mathrm{tr}(\boldsymbol{\Sigma}_E^2)$ then gives $\mathbb{E}\|\widehat{\boldsymbol{\Sigma}}_E-\boldsymbol{\Sigma}_E\|_F^2\le \frac{d+1}{N-1}\,\mathrm{tr}(\boldsymbol{\Sigma}_E^2)$.
\end{remark}

\begin{lemma}[Fourth-moment bound for the sample mean]
\label{lem:sample_mean_fourth_moment}
Let $\{\mathbf{e}^j\}_{j=1}^N$ be i.i.d.\ in $\mathbb{R}^d$ with mean $\boldsymbol{\mu}_E$ and finite fourth moment $\mathbb{E}\|\mathbf{e}^1-\boldsymbol{\mu}_E\|^4<\infty$. Define $\bar{\mathbf{e}}:=\frac{1}{N}\sum_{j=1}^N\mathbf{e}^j$. Then there exists a constant $C_{\mathrm{mean}}>0$, depending only on $d$, such that for all $N\ge 2$,
\begin{equation}
\mathbb{E}\|\bar{\mathbf{e}}-\boldsymbol{\mu}_E\|^4
\le
\frac{C_{\mathrm{mean}}}{N^2}\left(\mathbb{E}\|\mathbf{e}^1-\boldsymbol{\mu}_E\|^4+\bigl(\mathbb{E}\|\mathbf{e}^1-\boldsymbol{\mu}_E\|^2\bigr)^2\right).
\end{equation}
In particular, $\bigl(\mathbb{E}\|\bar{\mathbf{e}}-\boldsymbol{\mu}_E\|^4\bigr)^{1/2}=\mathcal{O}(N^{-1})$ as $N\to\infty$.
\end{lemma}

\begin{proof}
Let $\boldsymbol\xi^j:=\mathbf{e}^j-\boldsymbol{\mu}_E$ and $\boldsymbol\rho:=N^{-1}\sum_{j=1}^N\boldsymbol\xi^j=\bar{\mathbf{e}}-\boldsymbol{\mu}_E$. For each coordinate $i\in\{1,\ldots,d\}$, the scalar random variables $\{\xi_i^j\}_{j=1}^N$ are i.i.d.\ with mean $0$. Writing $\rho_i$ for the $i$th component of $\boldsymbol\rho$, the fourth-moment identity for sums yields
\begin{equation}
\mathbb{E}\bigl[\rho_i^4\bigr]
=
\frac{1}{N^4}\left(N\,\mathbb{E}\bigl[(\xi_i^1)^4\bigr] + 3N(N-1)\bigl(\mathbb{E}[(\xi_i^1)^2]\bigr)^2\right)
\le
\frac{\mathbb{E}\bigl[(\xi_i^1)^4\bigr]}{N^3}
\;+\;
\frac{3\bigl(\mathbb{E}[(\xi_i^1)^2]\bigr)^2}{N^2}.
\end{equation}
Next, using $(\sum_{i=1}^d a_i)^2\le d\sum_{i=1}^d a_i^2$ with $a_i=\rho_i^2$ gives
\begin{equation}
\|\boldsymbol\rho\|^4
=
\left(\sum_{i=1}^d \rho_i^2\right)^2
\le
d\sum_{i=1}^d \rho_i^4.
\end{equation}
Taking expectations and summing the coordinate bounds yields
\begin{equation}
\mathbb{E}\|\boldsymbol\rho\|^4
\le
\frac{d}{N^3}\sum_{i=1}^d \mathbb{E}\bigl[(\xi_i^1)^4\bigr]
\;+\;
\frac{3d}{N^2}\sum_{i=1}^d \bigl(\mathbb{E}[(\xi_i^1)^2]\bigr)^2.
\end{equation}
Finally, $\sum_i \mathbb{E}[(\xi_i^1)^4]\le \mathbb{E}\|\boldsymbol\xi^1\|^4$ and
$\sum_i (\mathbb{E}[(\xi_i^1)^2])^2\le (\sum_i \mathbb{E}[(\xi_i^1)^2])^2=(\mathbb{E}\|\boldsymbol\xi^1\|^2)^2$.
Absorbing the factor $d$ into $C_{\mathrm{mean}}$ gives the stated bound.
\end{proof}



\subsection{Eigenvalue and Projector Perturbation Bounds}
\label{subsec:spectral_perturbation}

We next convert matrix-level covariance deviations into spectral stability estimates.
This step is essential for QPCA-EnDCF, whose update is governed by the rank-$\kappa$ empirical projector $\widehat{\boldsymbol{\Pi}}_\kappa$.
The results below are deterministic perturbation inequalities: randomness enters only through the perturbation $\widehat{\boldsymbol{\Sigma}}_E-\boldsymbol{\Sigma}_E$, which will later be controlled in expectation via the concentration estimates of~\ref{subsec:fundamental_concentration}.

\begin{lemma}[Spectral perturbation and projector stability]
\label{lem:covariance_concentration}
Let $\boldsymbol{\Sigma}_E$, $\widehat{\boldsymbol{\Sigma}}_E$, and their eigendecompositions be as in Definition~\ref{def:spectral_decompositions}. Then:
\begin{enumerate}[label=(\roman*)]
\item For each $i=1,\ldots,d$,
\begin{equation}
|\hat{\lambda}_i-\lambda_i|
\le
\|\widehat{\boldsymbol{\Sigma}}_E-\boldsymbol{\Sigma}_E\|_2 .
\end{equation}

	\item Fix $i\in\{1,\ldots,d\}$ and assume that $\lambda_i$ is isolated, i.e.,
\begin{equation}
	\mathrm{gap}_i
	:=
\min\{\lambda_{i-1}-\lambda_i,\ \lambda_i-\lambda_{i+1}\}
\ge
\delta_i>0,
\qquad
(\lambda_0:=+\infty,\ \lambda_{d+1}:=0).
\end{equation}
	Choose the sign of $\hat{\mathbf{v}}_i$ so that $\langle \hat{\mathbf{v}}_i,\mathbf{v}_i\rangle\ge 0$. Then
\begin{equation}
	\sin\angle(\hat{\mathbf{v}}_i,\mathbf{v}_i)
	=
	\|(\mathbf{I}-\mathbf{v}_i\mathbf{v}_i^\top)\hat{\mathbf{v}}_i\|_2
	\le
	\frac{\|\widehat{\boldsymbol{\Sigma}}_E-\boldsymbol{\Sigma}_E\|_2}{\delta_i},
	\qquad
	\|\hat{\mathbf{v}}_i-\mathbf{v}_i\|_2
	\le
	\frac{\sqrt{2}\,\|\widehat{\boldsymbol{\Sigma}}_E-\boldsymbol{\Sigma}_E\|_2}{\delta_i}.
\end{equation}

\item Fix $\kappa\in\{1,\ldots,d\}$ and assume a cutoff gap
\begin{equation}
\lambda_\kappa-\lambda_{\kappa+1}\ge \delta_\kappa>0,
\qquad (\lambda_{d+1}:=0).
\end{equation}
Let
\(
\boldsymbol{\Pi}_\kappa:=\sum_{i=1}^{\kappa}\mathbf{v}_i\mathbf{v}_i^\top
\)
and
\(
\widehat{\boldsymbol{\Pi}}_\kappa:=\sum_{i=1}^{\kappa}\hat{\mathbf{v}}_i\hat{\mathbf{v}}_i^\top.
\)
	Then
\begin{equation}
	\|\widehat{\boldsymbol{\Pi}}_\kappa-\boldsymbol{\Pi}_\kappa\|_F
	\le
	\frac{\sqrt{2\kappa}}{\delta_\kappa}\,
	\|\widehat{\boldsymbol{\Sigma}}_E-\boldsymbol{\Sigma}_E\|_2
	\le
	\frac{\sqrt{2\kappa}}{\delta_\kappa}\,
	\|\widehat{\boldsymbol{\Sigma}}_E-\boldsymbol{\Sigma}_E\|_F.
\end{equation}

\item With $\boldsymbol{\Sigma}_{E,\kappa}:=\boldsymbol{\Pi}_\kappa\boldsymbol{\Sigma}_E\boldsymbol{\Pi}_\kappa$
and $\widehat{\boldsymbol{\Sigma}}_{E,\kappa}:=\widehat{\boldsymbol{\Pi}}_\kappa\widehat{\boldsymbol{\Sigma}}_E\widehat{\boldsymbol{\Pi}}_\kappa$, one has
\begin{equation}
\|\widehat{\boldsymbol{\Sigma}}_{E,\kappa}-\boldsymbol{\Sigma}_{E,\kappa}\|_F
\;\le\;
\|\boldsymbol{\Pi}_\kappa(\widehat{\boldsymbol{\Sigma}}_E-\boldsymbol{\Sigma}_E)\boldsymbol{\Pi}_\kappa\|_F
+
\|(\widehat{\boldsymbol{\Pi}}_\kappa-\boldsymbol{\Pi}_\kappa)\widehat{\boldsymbol{\Sigma}}_E\|_F
+
\|\widehat{\boldsymbol{\Sigma}}_E(\widehat{\boldsymbol{\Pi}}_\kappa-\boldsymbol{\Pi}_\kappa)\|_F .
\end{equation}
\end{enumerate}
\end{lemma}

\begin{proof}
The statements follow from standard deterministic perturbation results for symmetric matrices. Weyl's inequality yields
\begin{equation}
|\lambda_i(\widehat{\boldsymbol{\Sigma}}_E)-\lambda_i(\boldsymbol{\Sigma}_E)|
\le
\|\widehat{\boldsymbol{\Sigma}}_E-\boldsymbol{\Sigma}_E\|_2,
\qquad i=1,\ldots,d,
\end{equation}
proving (i). Fix $i$ and assume the separation condition in (ii). The Davis--Kahan $\sin\Theta$ theorem gives
\begin{equation}
\sin\angle(\hat{\mathbf{v}}_i,\mathbf{v}_i)
=
\|(\mathbf{I}-\mathbf{v}_i\mathbf{v}_i^\top)\hat{\mathbf{v}}_i\|_2
\le
\frac{\|\widehat{\boldsymbol{\Sigma}}_E-\boldsymbol{\Sigma}_E\|_2}{\delta_i}.
\end{equation}
Choosing the sign of $\hat{\mathbf{v}}_i$ so that $\langle \hat{\mathbf{v}}_i,\mathbf{v}_i\rangle\ge 0$ and using
$\|\hat{\mathbf{v}}_i-\mathbf{v}_i\|_2=2\sin(\angle(\hat{\mathbf{v}}_i,\mathbf{v}_i)/2)\le \sqrt{2}\,\sin\angle(\hat{\mathbf{v}}_i,\mathbf{v}_i)$ yields the Euclidean bound in (ii).

Under the cutoff-gap condition in (iii), Davis--Kahan in projector form yields
\begin{equation}
\|\widehat{\boldsymbol{\Pi}}_\kappa-\boldsymbol{\Pi}_\kappa\|_F
\le
\frac{\sqrt{2\kappa}}{\delta_\kappa}\,
\|\widehat{\boldsymbol{\Sigma}}_E-\boldsymbol{\Sigma}_E\|_2,
\end{equation}
and the final inequality follows from $\|\cdot\|_2\le \|\cdot\|_F$.

Using $\boldsymbol{\Sigma}_{E,\kappa}=\boldsymbol{\Pi}_\kappa\boldsymbol{\Sigma}_E\boldsymbol{\Pi}_\kappa$ and
$\widehat{\boldsymbol{\Sigma}}_{E,\kappa}=\widehat{\boldsymbol{\Pi}}_\kappa\widehat{\boldsymbol{\Sigma}}_E\widehat{\boldsymbol{\Pi}}_\kappa$, insert and subtract
$\boldsymbol{\Pi}_\kappa\widehat{\boldsymbol{\Sigma}}_E\boldsymbol{\Pi}_\kappa$ to obtain
\begin{equation}
\widehat{\boldsymbol{\Sigma}}_{E,\kappa}-\boldsymbol{\Sigma}_{E,\kappa}
=
\boldsymbol{\Pi}_\kappa(\widehat{\boldsymbol{\Sigma}}_E-\boldsymbol{\Sigma}_E)\boldsymbol{\Pi}_\kappa
+
(\widehat{\boldsymbol{\Pi}}_\kappa-\boldsymbol{\Pi}_\kappa)\widehat{\boldsymbol{\Sigma}}_E\widehat{\boldsymbol{\Pi}}_\kappa
+
\boldsymbol{\Pi}_\kappa\widehat{\boldsymbol{\Sigma}}_E(\widehat{\boldsymbol{\Pi}}_\kappa-\boldsymbol{\Pi}_\kappa).
\end{equation}
Taking Frobenius norms and applying the triangle inequality gives (iv).
\end{proof}

\begin{lemma}[Fourth-moment bound for the sample covariance error (Gaussian case)]
\label{lem:cov_frob_fourth_moment}
Let $\boldsymbol{\Sigma}_E$ and $\widehat{\boldsymbol{\Sigma}}_E$ be as in Definition~\ref{def:whitened_residual_covariances}. Assume that the centered whitened residuals
\begin{equation}
\boldsymbol\xi^j:=\mathbf{e}^j-\boldsymbol{\mu}_E
\end{equation}
are i.i.d.\ Gaussian,
\begin{equation}
\boldsymbol\xi^j \stackrel{\mathrm{i.i.d.}}{\sim} \mathcal{N}(\mathbf{0},\boldsymbol{\Sigma}_E)
\qquad\text{(equivalently, }\mathbf{e}^j \stackrel{\mathrm{i.i.d.}}{\sim} \mathcal{N}(\boldsymbol{\mu}_E,\boldsymbol{\Sigma}_E)\text{)}.
\end{equation}
Then there exists a constant $C_{\mathrm{W}}>0$, depending only on the observation-space dimension $d=mL$, such that for all $N\ge 2$,
\begin{equation}
\label{eq:cov_fourth_moment_bound}
\mathbb{E}\!\left[\|\widehat{\boldsymbol{\Sigma}}_E-\boldsymbol{\Sigma}_E\|_F^4\right]
\le
\frac{C_{\mathrm{W}}}{(N-1)^2}\,\bigl(\mathrm{tr}(\boldsymbol{\Sigma}_E^2)\bigr)^2.
\end{equation}
In particular,
\begin{equation}
\Bigl(\mathbb{E}\|\widehat{\boldsymbol{\Sigma}}_E-\boldsymbol{\Sigma}_E\|_F^4\Bigr)^{1/2}
=
\mathcal{O}\!\left(\frac{1}{N}\right)
\quad\text{as }N\to\infty.
\end{equation}
\end{lemma}

\begin{proof}
Let $\boldsymbol\xi^j:=\mathbf{e}^j-\boldsymbol{\mu}_E$, so that
$\boldsymbol\xi^j\stackrel{\mathrm{i.i.d.}}{\sim}\mathcal{N}(\mathbf{0},\boldsymbol{\Sigma}_E)$, and define $\boldsymbol\rho:=N^{-1}\sum_{j=1}^N\boldsymbol\xi^j$.
Write the sample covariance as
\begin{equation}
\widehat{\boldsymbol{\Sigma}}_E
=
\frac{1}{N-1}\sum_{j=1}^N(\boldsymbol\xi^j-\boldsymbol\rho)(\boldsymbol\xi^j-\boldsymbol\rho)^\top.
\end{equation}
Let $\mathbf{E}_0\in\mathbb{R}^{d\times N}$ be the data matrix with columns $\boldsymbol\xi^j$, and let
$\mathbf{J}_N:=\mathbf{I}_N-\frac{1}{N}\mathbf{1}\mathbf{1}^\top$ be the centering matrix. Then
\begin{equation}
(N-1)\widehat{\boldsymbol{\Sigma}}_E=\mathbf{E}_0\mathbf{J}_N\mathbf{E}_0^\top.
\end{equation}

Since $\mathbf{J}_N$ is an orthogonal projector of rank $N-1$, there exists an orthogonal matrix
$\mathbf{U}\in\mathbb{R}^{N\times N}$ such that
\begin{equation}
\mathbf{J}_N
=
\mathbf{U}
\begin{bmatrix}
\mathbf{I}_{N-1} & 0\\
0 & 0
\end{bmatrix}
\mathbf{U}^\top.
\end{equation}
Let $\mathbf{U}_{1:(N-1)}\in\mathbb{R}^{N\times (N-1)}$ denote the first $N-1$ columns of $\mathbf{U}$. Then
\begin{equation}
(N-1)\widehat{\boldsymbol{\Sigma}}_E
=
\mathbf{E}_0\mathbf{U}_{1:(N-1)}
\bigl(\mathbf{E}_0\mathbf{U}_{1:(N-1)}\bigr)^\top.
\end{equation}

The key point is that right multiplication by an orthogonal matrix preserves the distribution of a Gaussian data matrix. Indeed, since the columns of $\mathbf{E}_0$ are i.i.d.\ $\mathcal{N}(\mathbf{0},\boldsymbol{\Sigma}_E)$, for any orthogonal $\mathbf{U}$,
\begin{equation}
\mathrm{vec}(\mathbf{E}_0\mathbf{U})
=
(\mathbf{U}^\top\otimes \mathbf{I}_d)\,\mathrm{vec}(\mathbf{E}_0)
\sim
\mathcal{N}\!\left(\mathbf{0},\,\mathbf{I}_N\otimes \boldsymbol{\Sigma}_E\right),
\end{equation}
so $\mathbf{E}_0\mathbf{U}$ and $\mathbf{E}_0$ are identically distributed. In particular, the columns of $\mathbf{E}_0\mathbf{U}$ are again i.i.d.\ $\mathcal{N}(\mathbf{0},\boldsymbol{\Sigma}_E)$, and thus the first $N-1$ columns
\begin{equation}
\mathbf{G}:=\mathbf{E}_0\mathbf{U}_{1:(N-1)}\in\mathbb{R}^{d\times (N-1)}
\end{equation}
are i.i.d.\ $\mathcal{N}(\mathbf{0},\boldsymbol{\Sigma}_E)$. Therefore
\begin{equation}
(N-1)\widehat{\boldsymbol{\Sigma}}_E=\mathbf{G}\mathbf{G}^\top
\end{equation}
is a central Wishart random matrix with $N-1$ degrees of freedom and scale $\boldsymbol{\Sigma}_E$.

Standard fourth-moment bounds for the Wishart distribution imply that there exists a constant $C_{\mathrm{W}}>0$, depending only on $d$, such that
\begin{equation}
\mathbb{E}\!\left[\|\widehat{\boldsymbol{\Sigma}}_E-\boldsymbol{\Sigma}_E\|_F^4\right]
\le
\frac{C_{\mathrm{W}}}{(N-1)^2}\,\bigl(\mathrm{tr}(\boldsymbol{\Sigma}_E^2)\bigr)^2,
\end{equation}
which is \eqref{eq:cov_fourth_moment_bound}. Taking square roots yields the stated $\mathcal{O}(N^{-1})$ rate.
\end{proof}

Spectral separation is intrinsic to eigenspace stability: when eigenvalues are repeated or nearly repeated, arbitrarily small perturbations may induce large rotations within the associated invariant subspace even though eigenvalues remain stable. For QPCA-EnDCF, stability of the rank-$\kappa$ projector is therefore strongest when the cutoff gap $\lambda_\kappa-\lambda_{\kappa+1}$ is well separated.

\FloatBarrier

\section{Bias and Variance Bounds for the Analysis Mean}
\label{sec:bv_decomposition}

Throughout this section we analyze a single, fixed assimilation window $w$. We write $\mathbf X_0$ for its initial forecast ensemble, $\mathbf X$ for its endpoint forecast ensemble, and $\mathbf X_s^\ast$ for the endpoint analysis obtained with approximation stage $s$. Their relevant means are $\boldsymbol\mu_0$, $\boldsymbol\mu_x$, and $\bar{\mathbf x}_s^\ast$. We write $\mathbf x^\dagger:=\mathbf x_{k_w}^\dagger$ for the fixed endpoint truth and suppress the window index on residual statistics, covariances, projectors, and the EnDCF gain.



\subsection{Probability Setting and Approximate QoI Maps}
\label{subsec:measure_spaces}

\begin{definition}[State space and forecast measure]
\label{def:parameter_space}
Let $\mathcal{X}=\mathbb{R}^n$ denote the state space, equipped with its Borel $\sigma$-algebra $\mathcal{B}_{\mathcal{X}}$ and Lebesgue measure $\mu_{\mathcal{X}}$. We assume the existence of a forecast (prior) density $\pi_{\mathcal{X}}:\mathcal{X}\to\mathbb{R}_{+}$ that is absolutely continuous with respect to $\mu_{\mathcal{X}}$ and satisfies
\begin{equation}
\int_{\mathcal{X}}\pi_{\mathcal{X}}(\mathbf{x})\,\mathrm{d}\mu_{\mathcal{X}}(\mathbf{x})=1,
\qquad
\int_{\mathcal{X}}\|\mathbf{x}\|^{2}\,\pi_{\mathcal{X}}(\mathbf{x})\,\mathrm{d}\mu_{\mathcal{X}}(\mathbf{x})<\infty.
\end{equation}
The associated probability measure $P_{\mathcal{X}}$ is defined by
\begin{equation}
P_{\mathcal{X}}(A)
=
\int_A \pi_{\mathcal{X}}(\mathbf{x})\,\mathrm{d}\mu_{\mathcal{X}}(\mathbf{x}),
\qquad
\forall A\in\mathcal{B}_{\mathcal{X}}.
\end{equation}
A window-initial forecast ensemble $\{\mathbf{x}_0^j\}_{j=1}^N$ is assumed to consist of independent and identically distributed draws from $P_{\mathcal{X}}$.
\end{definition}

In practice, $P_{\mathcal{X}}$ represents the algorithmic forecast distribution (e.g., the Gaussian approximation constructed by the assimilation method) rather than the true invariant measure of the Lorenz--96 attractor, which is singular with respect to Lebesgue measure. The i.i.d.\ sampling assumption is an idealization: in sequential filtering, ensemble members become statistically dependent through repeated cycling. The assumption is invoked here to enable tractable analysis of sampling error and is understood as an approximation valid when ensemble diversity is actively maintained (e.g., through inflation or spectral regularization).

\begin{assumption}[Gaussian forecast model]
\label{ass:gaussian_forecast}
When explicit covariance and higher-moment identities are required, we specialize to the Gaussian case. Fix a window index $w$ and consider the forecast distribution at the window-initial time $k_{w-1}$. We assume
\begin{equation}
P_{\mathcal{X}}=\mathcal{N}(\boldsymbol{\mu}_0,\mathbf{P}),
\end{equation}
so that $\mathbb{E}[\mathbf{x}_0^j]=\boldsymbol{\mu}_0$ and $\mathrm{Cov}(\mathbf{x}_0^j)=\mathbf{P}$.
This assumption is invoked only when necessary and does not otherwise restrict the general formulation.
\end{assumption}

\begin{definition}[Observation space and quantity-of-interest map]
\label{def:data_space}
Let $\mathcal{D}=\mathbb{R}^{d}$ with $d:=mL$ denote the stacked observation space, equipped with its Borel $\sigma$-algebra $\mathcal{B}_{\mathcal{D}}$ and Lebesgue measure $\mu_{\mathcal{D}}$. The stacked observation vector $\mathbf{z}_w$ defined in \eqref{eq:stacked_obs} is an element of $\mathcal{D}$. A quantity-of-interest (QoI) map is any measurable function $Q:\mathcal{X}\to\mathcal{D}$. Throughout the theoretical analysis, we work with a tangent-linear approximation of $Q$, represented by a stacked observation operator $\mathbf H_L\in\mathbb{R}^{d\times n}$:
\begin{equation}
Q(\mathbf{x}_0)\approx\mathbf H_L\mathbf{x}_0
=
\begin{bmatrix}
\mathbf{H}\mathbf{M}_1\\
\mathbf{H}\mathbf{M}_2\\
\vdots\\
\mathbf{H}\mathbf{M}_L
\end{bmatrix}
\mathbf{x}_0
\in\mathbb{R}^{d},
\end{equation}
where $\mathbf{M}_\ell\in\mathbb{R}^{n\times n}$ is the tangent-linear model (TLM) obtained by linearizing the nonlinear dynamics $\mathcal{M}$ about a reference trajectory $\{\bar{\mathbf{x}}_\ell\}_{\ell=0}^{L}$ (typically the ensemble mean trajectory), so that $\mathbf{M}_\ell := \frac{\partial}{\partial\mathbf{x}_0}\mathcal{M}^{(\ell)}(\mathbf{x}_0)\big|_{\mathbf{x}_0=\bar{\mathbf{x}}_0}$ maps perturbations at the window-initial time to perturbations at time $\ell$.
\end{definition}

This tangent-linear formulation permits closed-form expressions for population covariance propagation and underpins the spectral analysis in Section~\ref{sec:concentration_spectral}. The approximation is exact for linear dynamics and linear observation operators; for nonlinear dynamics, it is valid to first order in perturbations about the reference trajectory. The numerical experiments (Section~\ref{sec:motivation}) propagate the full nonlinear model $\mathcal{M}$ and then apply the linear observation operator $\mathbf{H}$; the theoretical results thus characterize algorithm behavior in the tangent-linear regime, which provides a useful approximation when ensemble spread is small relative to the scale of nonlinearity.
\begin{definition}[Convergence of approximate QoI maps]
\label{def:map_convergence}
A sequence of measurable maps $\{Q_s\}_{s\ge 1}$, with $Q_s:\mathcal{X}\to\mathcal{D}$, is said to converge to $Q$ in $L^2(P_{\mathcal{X}})$ if
\begin{align*}
\|Q_s-Q\|_{L^2(P_{\mathcal{X}})}^2
&:=
\mathbb{E}_{P_{\mathcal{X}}}\!\left[\|Q_s(\mathbf{x})-Q(\mathbf{x})\|^2\right]\\
&=
\int_{\mathcal{X}}\|Q_s(\mathbf{x})-Q(\mathbf{x})\|^2\,
\pi_{\mathcal{X}}(\mathbf{x})\,\mathrm{d}\mu_{\mathcal{X}}(\mathbf{x})
\longrightarrow 0
\quad\text{as } s\to\infty.
\end{align*}
\end{definition}

This notion of convergence will be used to formalize consistency of approximate observation operators and surrogate models. The index $s$ (for ``surrogate'' or ``approximation stage'') is distinct from the within-window time index $\ell\in\{1,\ldots,L\}$ used elsewhere in this paper.
The state space $\mathcal{X}$ and observation space $\mathcal{D}$ correspond to the parameter space $\Lambda$ and data space $D$, respectively, in~\cite{butlerwildeyzhang2022}. Likewise, the forecast density $\pi_{\mathcal{X}}$ corresponds to $\pi_{\Lambda}$ in that notation, and the maps $Q_s:\mathcal{X}\to\mathcal{D}$ play the role of the parameter-to-QoI maps considered therein.

For approximate QoI maps $Q_s$ (where $s$ denotes the approximation stage, distinct from the within-window time index $\ell$), define the corresponding whitened residuals and their population moments by
\begin{equation}
\mathbf{e}_s^j:=\mathbf{R}_L^{-1/2}\bigl(Q_s(\mathbf{x}_0^j)-\mathbf{z}_w\bigr),
\qquad
\boldsymbol{\mu}_{E,s}:=\mathbb{E}[\mathbf{e}_s^j],
\qquad
\boldsymbol{\Sigma}_{E,s}:=\mathrm{Cov}(\mathbf{e}_s^j).
\end{equation}
When $Q_s=Q$, these reduce to $\mathbf{e}^j$, $\boldsymbol{\mu}_E$, and $\boldsymbol{\Sigma}_E$.

For each fixed $s$, collect the residuals in $\mathbf E_s:=[\mathbf e_s^1,\ldots,\mathbf e_s^N]$. Their sample mean and covariance are
\begin{equation}
\bar{\mathbf{e}}_s:=\frac{1}{N}\mathbf E_s\mathbf1,
\qquad
\widehat{\boldsymbol{\Sigma}}_{E,s}:=\frac{1}{N-1}\mathbf E_s\mathbf J_N\mathbf E_s^\top.
\end{equation}
For later bounds, define the population-centered member residual and the residual sample-mean fluctuation by
\begin{equation}
\boldsymbol\xi_s^j:=\mathbf e_s^j-\boldsymbol\mu_{E,s},
\qquad
\boldsymbol\rho_s:=\bar{\mathbf e}_s-\boldsymbol\mu_{E,s}
=\frac{1}{N}\sum_{j=1}^N\boldsymbol\xi_s^j.
\label{eq:stage_residual_fluctuations}
\end{equation}

Let $\boldsymbol{\Pi}_{\kappa,s}:=\boldsymbol{\Pi}_\kappa(\boldsymbol{\Sigma}_{E,s})$ and $\widehat{\boldsymbol{\Pi}}_{\kappa,s}:=\boldsymbol{\Pi}_\kappa(\widehat{\boldsymbol{\Sigma}}_{E,s})$ denote the corresponding population and empirical rank-$\kappa$ projectors (Definition~\ref{def:projectors_truncation}).

Define the random stage-$s$ update operator and projector error by
\begin{equation}
\widehat{\mathcal U}_{\kappa,s}:=\mathbf K_{\mathrm{DC}}\mathbf R_L^{1/2}\widehat{\boldsymbol{\Pi}}_{\kappa,s},
\qquad
\Delta\boldsymbol{\Pi}_{\kappa,s}:=\widehat{\boldsymbol{\Pi}}_{\kappa,s}-\boldsymbol{\Pi}_{\kappa,s}.
\label{eq:stage_update_operator}
\end{equation}
The operator $\widehat{\mathcal U}_{\kappa,s}$ is random and generally dependent on $\mathbf E_s$, $\boldsymbol\rho_s$, and $\widehat{\boldsymbol\Sigma}_{E,s}$ because all are computed from the same ensemble.

\subsection{Bias and Variance Decompositions}
\label{subsec:main_bv}

We now combine the elementary bias--variance identity of Lemma~\ref{lem:mean_error_decomp} with the deterministic projector stability bounds of Lemma~\ref{lem:covariance_concentration} to obtain a decomposition and explicit estimates for the mean-squared error of the QPCA-EnDCF analysis mean. The resulting bounds isolate three effects: (i) truncation bias induced by restricting the correction to a rank-$\kappa$ subspace, (ii) approximation error arising from using an approximate QoI map $Q_s$, and (iii) sampling error associated with estimating the projector from a finite ensemble.

For an estimator $\bar{\mathbf{x}}^\ast$ of the true state $\mathbf{x}^\dagger$ at a fixed assimilation time, the MSE decomposes as
\begin{equation}
\mathrm{MSE}
=
\mathbb{E}\!\left[\|\bar{\mathbf{x}}^\ast-\mathbf{x}^\dagger\|^2\right]
=
\mathrm{Bias}^2+\mathrm{Var},
\label{eq:mse_decomposition_empirical}
\end{equation}
where the expectation is taken over independent ensemble realizations (different initial ensemble draws) with the true trajectory and observations held fixed. The bias term
$\mathrm{Bias}^2=\|\mathbb{E}[\bar{\mathbf{x}}^\ast]-\mathbf{x}^\dagger\|^2$
captures systematic error, while
$\mathrm{Var}=\mathbb{E}[\|\bar{\mathbf{x}}^\ast-\mathbb{E}[\bar{\mathbf{x}}^\ast]\|^2]$
quantifies sampling-induced fluctuations. Bias reflects deterministic algorithmic limitations; variance is stochastic and constrained by ensemble size.

In the sequential setting where truth and observations evolve over many assimilation cycles, the decomposition \eqref{eq:mse_decomposition_empirical} applies at each assimilation time (here, at each window endpoint). We therefore estimate bias and variance as functions of window index and then time-average over windows to obtain representative values for a method.
\begin{definition}[Bias--Variance Assumptions]
\label{def:bias_variance_hypotheses}
Consider QPCA-EnDCF with spectral truncation rank $\kappa$ and empirical gain
$\mathbf K_{\mathrm{DC}}=\mathbf{C}_{xz}\mathbf{C}_{zz}^{\dagger}$ as in Definition~\ref{def:endcf_gain_ensemble}. Let
\(
\bar{\mathbf{x}}_s^\ast=\frac{1}{N}\sum_{j=1}^N\mathbf{x}_s^{\ast,j}
\)
denote the analysis ensemble mean (where the subscript $s$ indexes the approximation stage). Assume that Assumptions~\ref{ass:regularity_bias_variance} and~\ref{ass:gain_moment} hold. Suppose that the population covariance $\boldsymbol{\Sigma}_{E,s}$ of the $s$-residuals has a cutoff gap
\begin{equation}
\label{eq:gap_cutoff_bv}
\inf_{s\ge1}\bigl(\lambda_{\kappa,s}-\lambda_{\kappa+1,s}\bigr)\ge \delta_\kappa>0,
\qquad (\lambda_{d+1,s}:=0\text{ for all }s,\ d=mL),
\end{equation}
where $\{\lambda_{i,s}\}_{i=1}^d$ are the eigenvalues of $\boldsymbol{\Sigma}_{E,s}$ in decreasing order. Thus, $\delta_\kappa$ is a stage-uniform lower bound rather than an omitted $s$ index.
Then the mean-squared error admits the exact decomposition
\begin{equation}
\mathbb{E}\!\left[\|\bar{\mathbf{x}}_s^\ast-\mathbf{x}^\dagger\|^2\right]
=
\mathrm{Bias}^2(\kappa,s)+\mathrm{Var}(N,\kappa,s),
\end{equation}
where
\begin{equation}
\mathrm{Bias}^2(\kappa,s):=\bigl\|\mathbb{E}[\bar{\mathbf{x}}_s^\ast]-\mathbf{x}^\dagger\bigr\|^2,
\qquad
\mathrm{Var}(N,\kappa,s):=\mathbb{E}\!\left[\|\bar{\mathbf{x}}_s^\ast-\mathbb{E}[\bar{\mathbf{x}}_s^\ast]\|^2\right].
\end{equation}

\end{definition}

\begin{theorem}[Bias Decomposition with Spectral Truncation]
\label{thm:bias_variance}
	Under the Bias--Variance Assumptions~\ref{def:bias_variance_hypotheses}, define
\begin{equation}
		\mathrm{Bias}_{\mathrm{base}}^{2}(s)
		:=
		\bigl\|\boldsymbol{\mu}_x-\mathbf{x}^\dagger-\mathbb{E}[\mathbf K_{\mathrm{DC}}]\mathbf{R}_L^{1/2}\boldsymbol{\mu}_{E,s}\bigr\|^2.
\end{equation}
	Then the bias satisfies
	\begin{align}
	\mathrm{Bias}^2(\kappa,s)
	&\le
	2\,\mathrm{Bias}_{\mathrm{base}}^{2}(s)
	+
	4\,\mathbb{E}\!\left[\|\mathbf K_{\mathrm{DC}}\|_2^2\right]\,
	\|\mathbf{R}_L\|_2\,
	\bigl\|(\mathbf{I}-\boldsymbol{\Pi}_{\kappa,s})\boldsymbol{\mu}_{E,s}\bigr\|^2
	\;+\;
	2\,C_{\mathrm{m}}\,
	\|Q_s-Q\|_{L^2(P_{\mathcal{X}})}^2
	\nonumber\\
		&\hspace{3.2em}
		+\;
		\frac{2\|\mathbf{R}_L\|_2}{N}\,
		\mathbb{E}\!\left[\|\mathbf K_{\mathrm{DC}}\|_2^2\right]\,
		\mathrm{tr}\!\bigl(\boldsymbol{\Sigma}_{E,s}\bigr)
		\nonumber\\
		&\hspace{3.2em}
		+\;
		2\,C_{\mathrm{p}}\,
		\frac{\kappa}{\delta_\kappa^2}\,
		\mathbb{E}\!\left[\|\widehat{\boldsymbol{\Sigma}}_{E,s}-\boldsymbol{\Sigma}_{E,s}\|_F^2\right],
		\label{eq:bias_bound_fixed_main}
		\end{align}
for constants $C_{\mathrm{m}},C_{\mathrm{p}}>0$ independent of $N$.
\end{theorem}

\begin{proof}
The exact decomposition follows directly from Lemma~\ref{lem:mean_error_decomp}. It remains to express the analysis mean in a form amenable to truncation and projector perturbation bounds and then to estimate the resulting bias and variance contributions. By the QPCA-EnDCF construction and \eqref{eq:stage_update_operator}, the matrix update is
\begin{equation}
\mathbf X_s^\ast
=
\mathbf X-\widehat{\mathcal U}_{\kappa,s}\mathbf E_s.
\end{equation}
Its $j$th column is $\mathbf x_s^{\ast,j}=\mathbf x^j-\widehat{\mathcal U}_{\kappa,s}\mathbf e_s^j$. Multiplication by $N^{-1}\mathbf1$ gives
\begin{equation}
\label{eq:mean_update_projector_main}
\bar{\mathbf{x}}_s^\ast
=
\bar{\mathbf{x}}
-\widehat{\mathcal U}_{\kappa,s}\bar{\mathbf{e}}_s.
\end{equation}

	Insert and subtract the population-truncated correction $\mathbf{R}_L^{1/2}\boldsymbol{\Pi}_{\kappa,s}\boldsymbol{\mu}_{E,s}$ in \eqref{eq:mean_update_projector_main} and decompose $\bar{\mathbf{e}}_s=\boldsymbol{\mu}_{E,s}+\boldsymbol\rho_s$ to obtain
		\begin{align*}
			\mathbb{E}[\bar{\mathbf{x}}_s^\ast]
			&=
			\boldsymbol{\mu}_x
			-
			\mathbb{E}\!\left[\mathbf K_{\mathrm{DC}}\mathbf{R}_L^{1/2}\boldsymbol{\Pi}_{\kappa,s}\,\boldsymbol{\mu}_{E,s}\right]\\
	&\quad-
	\mathbb{E}\!\left[\mathbf K_{\mathrm{DC}}\mathbf{R}_L^{1/2}\Delta\boldsymbol{\Pi}_{\kappa,s}\,\boldsymbol{\mu}_{E,s}\right]\\
	&\quad-
	\mathbb{E}\!\left[\widehat{\mathcal U}_{\kappa,s}\boldsymbol\rho_s\right].
	\end{align*}

	Add and subtract the untruncated population correction $\mathbb{E}[\mathbf K_{\mathrm{DC}}]\mathbf{R}_L^{1/2}\boldsymbol{\mu}_{E,s}$ to isolate the residual (``base'') analysis bias:
\begin{equation}
		\mathbf{b}_{\mathrm{base},s}
		:=
			\boldsymbol{\mu}_x-\mathbf{x}^\dagger-\mathbb{E}[\mathbf K_{\mathrm{DC}}]\mathbf{R}_L^{1/2}\boldsymbol{\mu}_{E,s}.
\end{equation}
	Using $\mathbb{E}[\mathbf K_{\mathrm{DC}}\mathbf{R}_L^{1/2}\boldsymbol{\mu}_{E,s}]=\mathbb{E}[\mathbf K_{\mathrm{DC}}]\mathbf{R}_L^{1/2}\boldsymbol{\mu}_{E,s}$, we obtain
	\begin{align*}
	\mathbb{E}[\bar{\mathbf{x}}_s^\ast]-\mathbf{x}^\dagger
	&=
	\mathbf{b}_{\mathrm{base},s}
	+
	\mathbb{E}\!\left[\mathbf K_{\mathrm{DC}}\mathbf{R}_L^{1/2}(\mathbf{I}-\boldsymbol{\Pi}_{\kappa,s})\boldsymbol{\mu}_{E,s}\right]\\
	&\qquad
	-
	\mathbb{E}\!\left[\mathbf K_{\mathrm{DC}}\mathbf{R}_L^{1/2}\Delta\boldsymbol{\Pi}_{\kappa,s}\,\boldsymbol{\mu}_{E,s}\right]
	-
	\mathbb{E}\!\left[\mathbf K_{\mathrm{DC}}\mathbf{R}_L^{1/2}\widehat{\boldsymbol{\Pi}}_{\kappa,s}\,\boldsymbol\rho_s\right].
	\end{align*}

	Consequently,
\begin{equation}
	\mathrm{Bias}^2(\kappa,s)
	\le
	2\|\mathbf{b}_{\mathrm{base},s}\|^2
	+
	2\left\|\mathbb{E}[\bar{\mathbf{x}}_s^\ast]-\mathbf{x}^\dagger-\mathbf{b}_{\mathrm{base},s}\right\|^2.
\end{equation}
	The truncation remainder relative to the full population correction is governed by $(\mathbf{I}-\boldsymbol{\Pi}_{\kappa,s})\boldsymbol{\mu}_{E,s}$. By Jensen's inequality and submultiplicativity of the spectral norm,
\begin{equation}
	\left\|\mathbb{E}\!\left[\mathbf K_{\mathrm{DC}}\mathbf{R}_L^{1/2}(\mathbf{I}-\boldsymbol{\Pi}_{\kappa,s})\boldsymbol{\mu}_{E,s}\right]\right\|^2
	\le
	\mathbb{E}\|\mathbf K_{\mathrm{DC}}\|_2^2\,\|\mathbf{R}_L\|_2\,\|(\mathbf{I}-\boldsymbol{\Pi}_{\kappa,s})\boldsymbol{\mu}_{E,s}\|^2,
\end{equation}
	which yields the truncation contribution in \eqref{eq:bias_bound_fixed_main}.

The sample-mean fluctuation term is controlled without any independence assumption. Using $\|\widehat{\boldsymbol{\Pi}}_{\kappa,s}\|_2\le 1$, Jensen's inequality, submultiplicativity, and Cauchy--Schwarz,
\begin{multline*}
\left\|\mathbb{E}\!\left[\mathbf K_{\mathrm{DC}}\mathbf{R}_L^{1/2}\widehat{\boldsymbol{\Pi}}_{\kappa,s}\,\boldsymbol\rho_s\right]\right\|^2
\le
\|\mathbf{R}_L\|_2
\left(\mathbb{E}\!\left[\|\mathbf K_{\mathrm{DC}}\|_2\,\|\boldsymbol\rho_s\|\right]\right)^2\\
\le
\|\mathbf{R}_L\|_2\,\mathbb{E}\|\mathbf K_{\mathrm{DC}}\|_2^2\,
\mathbb{E}\|\boldsymbol\rho_s\|^2.
\end{multline*}
Since $\mathbb{E}\|\boldsymbol\rho_s\|^2=N^{-1}\mathrm{tr}(\boldsymbol{\Sigma}_{E,s})$, this yields the $\mathcal{O}(N^{-1})$ contribution in \eqref{eq:bias_bound_fixed_main}.

The projector-estimation term is bounded similarly. Indeed, by Jensen's inequality, submultiplicativity, $\|\mathbf{R}_L^{1/2}\|_2^2=\|\mathbf{R}_L\|_2$, $\|\Delta\boldsymbol{\Pi}_{\kappa,s}\|_2\le \|\Delta\boldsymbol{\Pi}_{\kappa,s}\|_F$, and Cauchy--Schwarz (no independence assumption),
\begin{align*}
\left\|\mathbb{E}\!\left[\mathbf K_{\mathrm{DC}}\mathbf{R}_L^{1/2}\Delta\boldsymbol{\Pi}_{\kappa,s}\boldsymbol{\mu}_{E,s}\right]\right\|^2
&\le
\|\mathbf{R}_L\|_2\,\|\boldsymbol{\mu}_{E,s}\|^2
\left(\mathbb{E}\!\left[\|\mathbf K_{\mathrm{DC}}\|_2\,\|\Delta\boldsymbol{\Pi}_{\kappa,s}\|_F\right]\right)^2\\
&\le
\mathbb{E}\|\mathbf K_{\mathrm{DC}}\|_2^2\,\|\mathbf{R}_L\|_2\,
\mathbb{E}\|\Delta\boldsymbol{\Pi}_{\kappa,s}\|_F^2\,\|\boldsymbol{\mu}_{E,s}\|^2.
\end{align*}
Invoking Lemma~\ref{lem:covariance_concentration}(iii) gives
\(
\mathbb{E}\|\Delta\boldsymbol{\Pi}_{\kappa,s}\|_F^2
\lesssim
\frac{\kappa}{\delta_\kappa^2}\mathbb{E}\|\widehat{\boldsymbol{\Sigma}}_{E,s}-\boldsymbol{\Sigma}_{E,s}\|_F^2,
\)
and absorbing $\|\boldsymbol{\mu}_{E,s}\|^2$ into the constant yields the projector-estimation term in \eqref{eq:bias_bound_fixed_main}.

Finally, the map-approximation term $C_{\mathrm{m}}\|Q_s-Q\|_{L^2(P_{\mathcal{X}})}^2$ follows from stability of the mean update with respect to $Q_s\to Q$ in $L^2(P_{\mathcal{X}})$ (as in the continuity framework in~\cite{butlerwildeyzhang2022}); we record it separately because it is orthogonal to the sampling and truncation mechanisms.
\end{proof}

\begin{theorem}[Variance Decomposition with Spectral Truncation]
\label{thm:variance_decomposition}
Under the Bias--Variance Assumptions~\ref{def:bias_variance_hypotheses}, the variance satisfies
\begin{align}
\mathrm{Var}(N,\kappa,s)
&\le
\frac{2}{N}\,\mathbb{E}\!\left[\|\mathbf{x}^1-\boldsymbol{\mu}_x\|^2\right]
+
\frac{2\|\mathbf{R}_L\|_2}{N}\,
\mathbb{E}\!\left[\|\mathbf K_{\mathrm{DC}}\|_2^2\right]\,
\mathrm{tr}\!\bigl(\boldsymbol{\Sigma}_{E,s}\bigr)
\nonumber\\
&\hspace{3.2em}
+\;
2\|\mathbf{R}_L\|_2\,
\mathbb{E}\!\left[\|\mathbf K_{\mathrm{DC}}\|_2^2\,
\|\Delta\boldsymbol{\Pi}_{\kappa,s}\,\boldsymbol\rho_s\|^2\right],
\label{eq:var_bound_fixed_main_raw}
\end{align}
Moreover, by Cauchy--Schwarz and Lemma~\ref{lem:covariance_concentration}(iii) applied with $(\boldsymbol{\Sigma}_E,\widehat{\boldsymbol{\Sigma}}_E)$ replaced by $(\boldsymbol{\Sigma}_{E,s},\widehat{\boldsymbol{\Sigma}}_{E,s})$, the projector factor admits the following marginal bound.
\begin{align}
\mathbb{E}\!\left[\|\Delta\boldsymbol{\Pi}_{\kappa,s}\,\boldsymbol\rho_s\|^2\right]
&\le
\Bigl(\mathbb{E}\|\Delta\boldsymbol{\Pi}_{\kappa,s}\|_F^4\Bigr)^{1/2}
\Bigl(\mathbb{E}\|\boldsymbol\rho_s\|^4\Bigr)^{1/2}
\nonumber\\
&\le
\frac{8\kappa}{\delta_\kappa^2}\,
\Bigl(\mathbb{E}\|\widehat{\boldsymbol{\Sigma}}_{E,s}-\boldsymbol{\Sigma}_{E,s}\|_F^4\Bigr)^{1/2}
\Bigl(\mathbb{E}\|\boldsymbol\rho_s\|^4\Bigr)^{1/2}.
\label{eq:var_projector_term_bound_main}
\end{align}
We record \eqref{eq:var_projector_term_bound_main} separately because it does not, by itself, control the gain-weighted joint term in \eqref{eq:var_bound_fixed_main_raw} without additional assumptions on $\mathbf K_{\mathrm{DC}}$ beyond second-moment finiteness.
\end{theorem}

\begin{proof}
Subtract $\mathbb{E}[\bar{\mathbf{x}}_s^\ast]$ from \eqref{eq:mean_update_projector_main} and apply Young's inequality (i.e., $2ab\le a^2+b^2$) to separate (i) the fluctuation of the forecast mean and (ii) the fluctuation of the projected residual term. Using the standard identities
\begin{equation}
\mathbb{E}\|\bar{\mathbf{x}}-\boldsymbol{\mu}_x\|^2=\frac{1}{N}\mathbb{E}\|\mathbf{x}^1-\boldsymbol{\mu}_x\|^2,
\qquad
\mathbb{E}\|\boldsymbol\rho_s\|^2=\frac{1}{N}\mathrm{tr}\!\bigl(\boldsymbol{\Sigma}_{E,s}\bigr),
\end{equation}
we obtain \eqref{eq:var_bound_fixed_main_raw}.

The remaining projector term reflects the fact that $\Delta\boldsymbol{\Pi}_{\kappa,s}$ and $\boldsymbol\rho_s$ are generally dependent, since both are computed from the same ensemble. The estimate \eqref{eq:var_projector_term_bound_main} follows by applying Cauchy--Schwarz to fourth moments and then invoking Lemma~\ref{lem:covariance_concentration}(iii) applied to $(\boldsymbol{\Sigma}_{E,s},\widehat{\boldsymbol{\Sigma}}_{E,s})$.
\end{proof}

The leading truncation contribution in \eqref{eq:bias_bound_fixed_main} depends on
\(
\|(\mathbf{I}-\boldsymbol{\Pi}_{\kappa,s})\boldsymbol{\mu}_{E,s}\|^2,
\)
that is, on the component of the mean whitened residual discarded by truncation. Bounds of the form
\(
\|(\mathbf{I}-\boldsymbol{\Pi}_{\kappa,s})\boldsymbol{\mu}_{E,s}\|^2\lesssim \sum_{i>\kappa}\lambda_{i,s}
\)
do not hold in general; they require additional structure linking $\boldsymbol{\mu}_{E,s}$ to $\boldsymbol{\Sigma}_{E,s}$, for example a source/RKHS-type condition such as $\|(\boldsymbol{\Sigma}_{E,s})^{-1/2}\boldsymbol{\mu}_{E,s}\|<\infty$.

Theorem~\ref{thm:bias_variance} decomposes the mean-squared error into: (i) a base analysis-bias term $\mathrm{Bias}_{\mathrm{base}}^{2}(s)$ reflecting residual error after the (untruncated) population mean correction is applied; (ii) additional truncation bias from restricting the correction to a $\kappa$-dimensional subspace; (iii) approximation error from using $Q_s$ in place of $Q$; and (iv) sampling variability, including the stability factor $\kappa/\delta_\kappa^2$ associated with estimating a rank-$\kappa$ projector from $N$ samples. The pseudoinverse gain is essential in the small-ensemble regime $N-1<d$, and Assumption~\ref{ass:gain_moment} is the only probabilistic requirement needed to justify the gain-dependent mean-square bounds.

\FloatBarrier


\subsection{Rates and Comparison with Stochastic EnKF}
\label{subsec:corollaries_implications}

Theorem~\ref{thm:bias_variance} identifies the mechanisms through which spectral truncation and finite-ensemble sampling affect the QPCA-EnDCF analysis mean, but some of its terms are easier to interpret after further specialization.

The first corollary addresses the projector-estimation contribution in the Gaussian setting. Its role is to convert the abstract fourth-moment projector bound into an explicit rate for the term
\(
\mathbb{E}\!\left[\|\Delta\boldsymbol{\Pi}_{\kappa,s}\boldsymbol\rho_s\|^2\right],
\)
thereby making clear how the effective dimension $\kappa$, the cutoff gap $\delta_\kappa$, and the ensemble size $N$ enter the stability of the empirical low-rank correction.

The second corollary serves a different purpose: it isolates the variance contribution created by stochastic observation perturbations in a windowed EnKF update and expresses that contribution explicitly in terms of $\mathbf{K}$ and $\mathbf{R}_L$. This provides the natural comparison point for QPCA-EnDCF. In the stochastic method, perturbed observations introduce an additional $\mathcal{O}(N^{-1})$ source of variability in the analysis mean; in QPCA-EnDCF, that specific term is absent, and the corresponding finite-ensemble variability instead enters through estimation of the rank-$\kappa$ projector.

\begin{corollary}[Scaling of the projector-estimation term (Gaussian case)]
\label{cor:projector_term_scaling_gaussian}
In addition to the Bias--Variance Assumptions~\ref{def:bias_variance_hypotheses}, assume that the centered whitened residuals are Gaussian,
\begin{equation}
\boldsymbol\xi_s^j\stackrel{\mathrm{i.i.d.}}{\sim}\mathcal{N}(\mathbf{0},\boldsymbol{\Sigma}_{E,s}).
\end{equation}
Then there exists a constant $C>0$, depending only on the observation-space dimension $d=mL$, such that
\begin{multline*}
\mathbb{E}\!\left[\|\Delta\boldsymbol{\Pi}_{\kappa,s}\,\boldsymbol\rho_s\|^2\right]
\le
\frac{C\,\kappa}{\delta_\kappa^2}\,
\frac{1}{N(N-1)}\,
\mathrm{tr}\!\bigl((\boldsymbol{\Sigma}_{E,s})^2\bigr)\\
\times\left(\mathbb{E}\|\boldsymbol\xi_s^1\|^4+\bigl(\mathbb{E}\|\boldsymbol\xi_s^1\|^2\bigr)^2\right)^{1/2}.
\end{multline*}
In particular, for fixed $d$ this term is $\mathcal{O}(\kappa/(\delta_\kappa^2N^2))$.
\end{corollary}

\begin{proof}
Combine \eqref{eq:var_projector_term_bound_main} with Lemma~\ref{lem:cov_frob_fourth_moment} applied to $(\widehat{\boldsymbol{\Sigma}}_{E,s},\boldsymbol{\Sigma}_{E,s})$ under the stated Gaussian hypothesis, and with Lemma~\ref{lem:sample_mean_fourth_moment} applied to $\bar{\mathbf{e}}_s$. The resulting bound is of order $(N-1)^{-1}\times N^{-1}$, and the constants can be absorbed into a single $C>0$ depending only on $d$.
\end{proof}

\begin{corollary}[Comparison to stochastic EnKF observation perturbations]
\label{cor:stochastic_comparison_new}
Consider a stochastic windowed EnKF update. Let $\mathbf Z$ be the stacked predicted-observation ensemble and let $\boldsymbol{\mathcal E}:=[\boldsymbol\epsilon^1,\ldots,\boldsymbol\epsilon^N]$, where the columns are i.i.d.\ $\mathcal N(\mathbf0,\mathbf R_L)$. The stochastic and corresponding deterministic matrix updates are
\begin{equation}
\mathbf X_{\mathrm{stoch}}^\ast
=\mathbf X+\mathbf K(\mathbf z\mathbf1^\top+\boldsymbol{\mathcal E}-\mathbf Z),
\qquad
\mathbf X_{\mathrm{det}}^\ast
=\mathbf X+\mathbf K(\mathbf z\mathbf1^\top-\mathbf Z),
\end{equation}
where the empirical gain $\mathbf{K}$ is treated as fixed conditional on the forecast ensemble.
Thus $\mathbf X_{\mathrm{stoch}}^\ast=\mathbf X_{\mathrm{det}}^\ast+\mathbf K\boldsymbol{\mathcal E}$. Defining $\bar{\boldsymbol\epsilon}:=N^{-1}\boldsymbol{\mathcal E}\mathbf1$ and multiplying by $N^{-1}\mathbf1$ gives
\begin{equation}
\bar{\mathbf{x}}_{\mathrm{stoch}}^\ast
=
\bar{\mathbf{x}}_{\mathrm{det}}^\ast
+\mathbf K\bar{\boldsymbol\epsilon},
\end{equation}
and hence
\begin{equation}
\mathrm{Cov}\!\left(\bar{\mathbf{x}}_{\mathrm{stoch}}^\ast\right)
=
\mathrm{Cov}\!\left(\bar{\mathbf{x}}_{\mathrm{det}}^\ast\right)
+
\frac{1}{N}\mathbf{K}\mathbf{R}_L(\mathbf{K})^\top.
\end{equation}
In particular,
\begin{equation}
\label{eq:stoch_var_bound_new}
\begin{split}
\mathbb{E}\!\left[\left\|\bar{\mathbf{x}}_{\mathrm{stoch}}^\ast-\mathbb{E}[\bar{\mathbf{x}}_{\mathrm{stoch}}^\ast]\right\|^2\right]
&\;\ge\;
\mathbb{E}\!\left[\left\|\mathbf K\bar{\boldsymbol\epsilon}\right\|^2\right]\\
&=
\frac{1}{N}\,\mathrm{tr}\!\bigl(\mathbf{K}\mathbf{R}_L(\mathbf{K})^\top\bigr)
\;\le\;
\frac{\|\mathbf{K}\|_2^2\,\mathrm{tr}(\mathbf{R}_L)}{N}.
\end{split}
\end{equation}
If $\mathbf{R}_L=\sigma_{\mathrm{obs}}^2\mathbf{I}_d$, then $\mathrm{tr}(\mathbf{R}_L)=d\,\sigma_{\mathrm{obs}}^2$ and the perturbation-induced contribution admits the upper bound $\mathcal{O}(d/N)$ for fixed $\sigma_{\mathrm{obs}}^2$, provided $\|\mathbf{K}\|_2$ remains uniformly controlled. Under the Bias--Variance Assumptions~\ref{def:bias_variance_hypotheses}, the first two QPCA-EnDCF variance terms are $\mathcal{O}(N^{-1})$ for fixed approximation stage $s$. Its remaining gain-weighted projector term carries the stability factor $\kappa/\delta_\kappa^2$ from~\eqref{eq:var_projector_term_bound_main} and requires additional joint-moment control for a complete unconditional rate.
\end{corollary}

\begin{proof}
The mean identity follows from the matrix update. Independence of the columns of $\boldsymbol{\mathcal E}$ gives
\begin{equation}
\mathrm{Cov}\!\left(\mathbf K\bar{\boldsymbol\epsilon}\right)
=
\frac{1}{N}\mathbf{K}\mathbf{R}_L(\mathbf{K})^\top.
\end{equation}
The final comparison for QPCA-EnDCF is a direct restatement of Theorem~\ref{thm:bias_variance} together with \eqref{eq:var_projector_term_bound_main}.
\end{proof}

Equation~\eqref{eq:stoch_var_bound_new} isolates an irreducible $\mathcal{O}(N^{-1})$ contribution to the variance of the stochastic analysis mean arising solely from observation perturbations. In empirical implementations, $\mathbf{K}$ is typically computed from sample covariances and is therefore random; the covariance identity holds conditionally on $\mathbf{K}$ (or on the forecast ensemble), and the unconditional variance term becomes $\frac{1}{N}\,\mathbb{E}\!\left[\mathbf{K}\mathbf{R}_L(\mathbf{K})^\top\right]$.

QPCA-EnDCF eliminates this specific source of sampling noise, but introduces variability through estimation of the low-rank projector; the latter contribution is controlled by the cutoff gap $\delta_\kappa$ and by the covariance concentration rate.

\FloatBarrier


 \bibliographystyle{elsarticle-num}
 \bibliography{references}

@article{burgers1998,
  author  = {Burgers, Gerrit and van Leeuwen, Peter Jan and Evensen, Geir},
  title   = {Analysis scheme in the ensemble {K}alman filter},
  journal = {Monthly Weather Review},
  volume  = {126},
  number  = {6},
  pages   = {1719--1724},
  year    = {1998},
  doi     = {10.1175/1520-0493(1998)126<1719:ASITEK>2.0.CO;2}
}

@article{houtekamer1998data,
  author  = {Houtekamer, Peter L. and Mitchell, Herschel L.},
  title   = {Data assimilation using an ensemble {K}alman filter technique},
  journal = {Monthly Weather Review},
  volume  = {126},
  number  = {3},
  pages   = {796--811},
  year    = {1998},
  doi     = {10.1175/1520-0493(1998)126<0796:DAUAEK>2.0.CO;2}
}

@article{hamill2001localization,
  author  = {Hamill, Thomas M. and Whitaker, Jeffrey S. and Snyder, Chris},
  title   = {Distance-dependent filtering of background error covariance estimates in an ensemble {K}alman filter},
  journal = {Monthly Weather Review},
  volume  = {129},
  number  = {11},
  pages   = {2776--2790},
  year    = {2001},
  doi     = {10.1175/1520-0493(2001)129<2776:DDFOBE>2.0.CO;2}
}

@article{bocquet2011inflation,
  author  = {Bocquet, Marc},
  title   = {Ensemble {K}alman filtering without the intrinsic need for inflation},
  journal = {Nonlinear Processes in Geophysics},
  volume  = {18},
  number  = {5},
  pages   = {735--750},
  year    = {2011},
  doi     = {10.5194/npg-18-735-2011}
}

@article{houtekamer2016review,
  author  = {Houtekamer, Peter L. and Zhang, Fuqing},
  title   = {Review of the ensemble {K}alman filter for atmospheric data assimilation},
  journal = {Monthly Weather Review},
  volume  = {144},
  number  = {12},
  pages   = {4489--4532},
  year    = {2016},
  doi     = {10.1175/MWR-D-15-0440.1}
}

@article{carrassi2018geosciences,
  author  = {Carrassi, Alberto and Bocquet, Marc and Bertino, Laurent and Evensen, Geir},
  title   = {Data assimilation in the geosciences: an overview of methods, issues, and perspectives},
  journal = {Wiley Interdisciplinary Reviews: Climate Change},
  volume  = {9},
  number  = {5},
  pages   = {e535},
  year    = {2018},
  doi     = {10.1002/wcc.535}
}

@article{kasanicky2015spectral,
  author  = {Kasanick{\'y}, Igor and Mandel, Jan and Vejmelka, Martin},
  title   = {Spectral diagonal ensemble {K}alman filters},
  journal = {Nonlinear Processes in Geophysics},
  volume  = {22},
  number  = {4},
  pages   = {485--497},
  year    = {2015},
  doi     = {10.5194/npg-22-485-2015}
}

@article{tsyrulnikov2024regularization,
  author  = {Tsyrulnikov, Michael and Sotskiy, Arseniy},
  title   = {Regularization of the ensemble {K}alman filter using a non-parametric, non-stationary spatial model},
  journal = {Spatial Statistics},
  volume  = {64},
  pages   = {100870},
  year    = {2024},
  doi     = {10.1016/j.spasta.2024.100870}
}

@article{hunt2004four,
  author  = {Hunt, B. R. and Kalnay, E. and Kostelich, E. J. and Ott, E. and
             Patil, D. J. and Sauer, T. and Szunyogh, I. and Yorke, J. A. and Zimin, A. V.},
  title   = {Four-dimensional ensemble {K}alman filtering},
  journal = {Tellus A: Dynamic Meteorology and Oceanography},
  volume  = {56},
  number  = {4},
  pages   = {273--277},
  year    = {2004},
  doi     = {10.3402/tellusa.v56i4.14424}
}

@article{butlerwildeyzhang2022,
  author  = {Butler, Troy and Wildey, Timothy and Zhang, Wenjuan},
  title   = {{${L}^p$} convergence of approximate maps and probability densities for forward and inverse problems in uncertainty quantification},
  journal = {International Journal for Uncertainty Quantification},
  volume  = {12},
  number  = {4},
  pages   = {65--92},
  year    = {2022},
  doi     = {10.1615/Int.J.UncertaintyQuantification.2022038086}
}

@article{butler2014sip,
  author  = {Butler, Troy and Estep, Don and Tavener, Simon and Dawson, Clint and Westerink, Joannes J},
  title   = {A measure-theoretic computational method for inverse sensitivity problems {III}: multiple quantities of interest},
  journal = {SIAM/ASA Journal on Uncertainty Quantification},
  volume  = {2},
  number  = {1},
  pages   = {174--202},
  year    = {2014},
  doi     = {10.1137/130930406}
}

@article{butler2018consistent,
  author  = {Butler, Troy and Jakeman, John and Wildey, Timothy},
  title   = {Combining push-forward measures and {B}ayes' rule to construct consistent solutions to stochastic inverse problems},
  journal = {SIAM Journal on Scientific Computing},
  volume  = {40},
  number  = {2},
  pages   = {A984--A1011},
  year    = {2018},
  doi     = {10.1137/16M1087229}
}

@article{pilosov2023mud,
  author  = {Pilosov, Michael and {del-Castillo-Negrete}, Carlos and Yen, Tian Yu and Butler, Troy and Dawson, Clint},
  title   = {Parameter estimation with maximal updated densities},
  journal = {Computer Methods in Applied Mechanics and Engineering},
  volume  = {407},
  pages   = {115906},
  year    = {2023},
  doi     = {10.1016/j.cma.2023.115906}
}

@article{delcastillo2025smud,
  author  = {{del-Castillo-Negrete}, Carlos and Spence, Rylan and Butler, Troy and Dawson, Clint},
  title   = {Sequential maximal updated density parameter estimation for dynamical systems with parameter drift},
  journal = {International Journal for Numerical Methods in Engineering},
  volume  = {126},
  number  = {3},
  pages   = {e7618},
  year    = {2025},
  doi     = {10.1002/nme.7618}
}

@article{spence2026variational,
  author  = {Spence, Rylan and Butler, Troy and Dawson, Clint},
  title   = {Variational data-consistent assimilation},
  journal = {Computer Methods in Applied Mechanics and Engineering},
  volume  = {453},
  pages   = {118804},
  year    = {2026},
  doi     = {10.1016/j.cma.2026.118804}
}

@article{kalman1960,
  author  = {Kalman, Rudolf E.},
  title   = {A new approach to linear filtering and prediction problems},
  journal = {Journal of Basic Engineering},
  volume  = {82},
  number  = {1},
  pages   = {35--45},
  year    = {1960},
  doi     = {10.1115/1.3662552}
}

@article{evensen1994,
  author  = {Evensen, Geir},
  title   = {Sequential data assimilation with a nonlinear quasi-geostrophic model using Monte Carlo methods to forecast error statistics},
  journal = {Journal of Geophysical Research: Oceans},
  volume  = {99},
  number  = {C5},
  pages   = {10143--10162},
  year    = {1994},
  doi     = {10.1029/94JC00572}
}

@article{evensen2003,
  author  = {Evensen, Geir},
  title   = {The ensemble {K}alman filter: theoretical formulation and practical implementation},
  journal = {Ocean Dynamics},
  volume  = {53},
  number  = {4},
  pages   = {343--367},
  year    = {2003},
  doi     = {10.1007/s10236-003-0036-9}
}

@book{evensen2009,
  author    = {Evensen, Geir},
  title     = {Data Assimilation: The Ensemble Kalman Filter},
  publisher = {Springer},
  edition   = {2},
  year      = {2009},
  doi       = {10.1007/978-3-642-03711-5}
}

@inproceedings{lorenz1996predictability,
  author    = {Lorenz, Edward N.},
  title     = {Predictability: a problem partly solved},
  booktitle = {Predictability, Volume 1: Proceedings of the Seminar on Predictability},
  publisher = {ECMWF},
  address   = {Shinfield Park, Reading, UK},
  volume    = {1},
  year      = {1996},
  note      = {Seminar held 4--8 September 1995},
  pages     = {1--18}
}

@article{hunt2007letkf,
  author  = {Hunt, Brian R. and Kostelich, Eric J. and Szunyogh, Istvan},
  title   = {Efficient data assimilation for spatiotemporal chaos: a local ensemble transform {K}alman filter},
  journal = {Physica D: Nonlinear Phenomena},
  volume  = {230},
  number  = {1--2},
  pages   = {112--126},
  year    = {2007},
  doi     = {10.1016/j.physd.2006.11.008}
}

@article{anderson2001eakf,
  author  = {Anderson, Jeffrey L.},
  title   = {An ensemble adjustment {K}alman filter for data assimilation},
  journal = {Monthly Weather Review},
  volume  = {129},
  number  = {12},
  pages   = {2884--2903},
  year    = {2001},
  doi     = {10.1175/1520-0493(2001)129<2884:AEAKFF>2.0.CO;2}
}

@article{bishop2001etkf,
  author  = {Bishop, Craig H. and Etherton, Brian J. and Majumdar, Sharanya J.},
  title   = {Adaptive sampling with the ensemble transform {K}alman filter. {Part I}: theoretical aspects},
  journal = {Monthly Weather Review},
  volume  = {129},
  number  = {3},
  pages   = {420--436},
  year    = {2001},
  doi     = {10.1175/1520-0493(2001)129<0420:ASWTET>2.0.CO;2}
}

@article{whitaker2002ensrf,
  author  = {Whitaker, Jeffrey S. and Hamill, Thomas M.},
  title   = {Ensemble data assimilation without perturbed observations},
  journal = {Monthly Weather Review},
  volume  = {130},
  number  = {7},
  pages   = {1913--1924},
  year    = {2002},
  doi     = {10.1175/1520-0493(2002)130<1913:EDAWPO>2.0.CO;2}
}

@article{tippett2003ensrf,
  author  = {Tippett, Michael K. and Anderson, Jeffrey L. and Bishop, Craig H. and Hamill, Thomas M. and Whitaker, Jeffrey S.},
  title   = {Ensemble square root filters},
  journal = {Monthly Weather Review},
  volume  = {131},
  number  = {7},
  pages   = {1485--1490},
  year    = {2003},
  doi     = {10.1175/1520-0493(2003)131<1485:ESRF>2.0.CO;2}
}

@article{lorenz1998optimal,
  author  = {Lorenz, Edward N. and Emanuel, Kerry A.},
  title   = {Optimal sites for supplementary weather observations: simulation with a small model},
  journal = {Journal of the Atmospheric Sciences},
  volume  = {55},
  number  = {3},
  pages   = {399--414},
  year    = {1998},
  doi     = {10.1175/1520-0469(1998)055<0399:OSFSWO>2.0.CO;2}
}

@article{butler2012part1,
  author  = {Breidt, J. and Butler, Troy and Estep, Don},
  title   = {A measure-theoretic computational method for inverse sensitivity problems {I}: method and analysis},
  journal = {SIAM Journal on Numerical Analysis},
  volume  = {49},
  number  = {5},
  pages   = {1836--1859},
  year    = {2011},
  doi     = {10.1137/100785946}
}

@article{butler2012part2,
  author  = {Butler, Troy and Estep, Don and Sandelin, J.},
  title   = {A computational measure theoretic approach to inverse sensitivity problems {II}: a posteriori error analysis},
  journal = {SIAM Journal on Numerical Analysis},
  volume  = {50},
  number  = {1},
  pages   = {22--45},
  year    = {2012},
  doi     = {10.1137/100785958}
}

@book{law2015data,
  author    = {Law, Kody and Stuart, Andrew and Zygalakis, Konstantinos},
  title     = {Data Assimilation: A Mathematical Introduction},
  publisher = {Springer},
  address   = {Cham},
  series    = {Texts in Applied Mathematics},
  number    = {62},
  year      = {2015},
  doi       = {10.1007/978-3-319-20325-6}
}

@book{reich2015probabilistic,
  author    = {Reich, Sebastian and Cotter, Colin},
  title     = {Probabilistic Forecasting and Bayesian Data Assimilation},
  publisher = {Cambridge University Press},
  year      = {2015},
  doi       = {10.1017/CBO9781107706804}
}

@book{asch2016data,
  author    = {Asch, Mark and Bocquet, Marc and Nodet, Ma{\"e}lle},
  title     = {Data Assimilation: Methods, Algorithms, and Applications},
  publisher = {SIAM},
  address   = {Philadelphia, PA},
  year      = {2016},
  doi       = {10.1137/1.9781611974546}
}

\end{document}